\documentclass[oneside,english]{amsart}
\usepackage{amsmath}
\usepackage{mathtools}
\mathtoolsset{showonlyrefs=true}

\usepackage{lmodern}

\usepackage[T1]{fontenc}
\usepackage[latin9]{inputenc}
\usepackage{mathrsfs}
\usepackage{amsbsy}
\usepackage{amstext}
\usepackage{amsthm}
\usepackage{amssymb}
\usepackage{esint}
\usepackage{setspace}
\usepackage{xcolor}
\usepackage[normalem]{ulem}
\usepackage{comment}
 \usepackage[unicode=true, 
  bookmarks=true,bookmarksnumbered=false,bookmarksopen=true,bookmarksopenlevel=1,
  breaklinks=false,pdfborder={0 0 1},backref=false,colorlinks=false]
  {hyperref}
\usepackage{enumitem}
\usepackage{bm} 
\usepackage{bbm} 

\usepackage{fancyvrb}
 \usepackage{cprotect}

 \numberwithin{equation}{section}

\newcommand{\bbd}{\mathbbm{d}}



\usepackage{babel}
  \providecommand{\definitionname}{Definition}
  \providecommand{\lemmaname}{Lemma}
  \providecommand{\propositionname}{Proposition}
\providecommand{\theoremname}{Theorem}
\providecommand{\corollaryname}{Corollary}
\providecommand{\remarkname}{Remark}
\providecommand{\examplename}{Example}

\newcommand{\bbC}{\mathbb{C}}
\newcommand{\bbD}{\mathbb{D}}
\newcommand{\bbE}{\mathbb{E}}

\newcommand{\bbL}{\mathbb{L}}

\newcommand{\bbN}{\mathbb{N}}

\newcommand{\bbR}{\mathbb{R}}

\newcommand{\bbT}{\mathbb{T}}

\newcommand{\bbZ}{\mathbb{Z}}

\newcommand{\cB}{\mathcal{B}}

\newcommand{\cE}{\mathcal{E}}

\newcommand{\cH}{\mathcal{H}}

\newcommand{\cK}{\mathcal{K}}
\newcommand{\cL}{\mathcal{L}}

\newcommand{\cR}{\mathcal{R}}

\newcommand{\cV}{\mathcal{V}}

\newcommand{\cX}{\mathcal{X}}

\newcommand{\ff}{\mathfrak{f}}

\newcommand{\fh}{\mathfrak{h}}

\newcommand{\fk}{\mathfrak{k}}

\newcommand{\fs}{\mathfrak{s}}

\newcommand{\fu}{\mathfrak{u}}

\newcommand{\fI}{\mathfrak{I}}

\newcommand{\fL}{\mathfrak{L}}
\newcommand{\fM}{\mathfrak{M}}

\newcommand{\fU}{\mathfrak{U}}

\newcommand{\supp}{\operatorname{supp}}

\newcommand{\und}[1]{\underline{#1}}
\newcommand{\uund}[1]{\underline{\underline{#1}}}
\newcommand{\wh}[1]{\widehat{#1}}
\newcommand{\wt}[1]{\widetilde{#1}}

\newcommand{\Ind}{\operatorname{Ind}}
\global\long\def\go{G^{(0)}}

\global\long\def\R{\mathbb{R}}

\global\long\def\om{\omega}
\global\long\def\Ginf{G_{\infty}}
\global\long\def\Xinf{X_{\infty}}

\newcommand{\tS}{\tilde{S}}
\newcommand{\tU}{\tilde{U}}
\newcommand{\tE}{\tilde{E}}
\newcommand{\mf}{\bm{m}} 
\newcommand{\Iso}{{\operatorname{Iso}}}

\newcommand{\projE}{\bm{E}}

\allowdisplaybreaks

\newcounter{parno}[paragraph]
\renewcommand{\theparno}{\thesection.\arabic{parno}}\newcommand{\p}{\refstepcounter{parno}\noindent{\textbf{\theparno}}\ }

\newcounter{subparno}[parno]
\renewcommand{\thesubparno}{\thesection.\arabic{parno}.\arabic{subparno}}\newcommand{\subp}{\refstepcounter{subparno}\noindent{\textbf{\thesubparno}}\ }

\numberwithin{equation}{section}
\numberwithin{figure}{section}
\theoremstyle{plain}
\newtheorem{thm}{\protect\theoremname}[parno]
  \theoremstyle{plain}
  \newtheorem{prop}[thm]{\protect\propositionname}
  \theoremstyle{plain}
  \newtheorem{lem}[thm]{\protect\lemmaname}
  \theoremstyle{definition}
  \newtheorem{defn}[thm]{\protect\definitionname}
  \newtheorem{cor}[thm]{\protect\corollaryname}
    \newtheorem{example}[thm]{\protect\examplename}
  \theoremstyle{plain}
  \newtheorem{remark}[thm]{\protect\remarkname}
  \theoremstyle{plain}
  
  \theoremstyle{plain}
\newtheorem*{lem*}{Lemma}
\newtheorem*{def*}{Definition}
\newtheorem*{prop*}{Proposition}
\newtheorem*{thm*}{Theorem}
\newtheorem*{cor*}{Corollary}

\dedicatory{To the memory of Iain Raeburn who was a dear friend, a wonderful collaborator, and an inspiring mathematician to all who knew him.}

\begin{document}

\title{A panoramic view of groupoids  and MRAs}

\author{Marius Ionescu}
\address{ Marius Ionescu\\ United States Naval Academy\\ Annapolis\\ MD 21402-5002 \\ USA}
\email{ionescu@usna.edu}
\thanks{The views expressed in this paper are those of the
authors and do not reflect the official policy or position of the U.S. Naval
Academy, Department of the Navy, the Department of Defense, or the U.S.
Government.}
\author{Paul S. Muhly}
\address{ Paul S. Muhly\\ University of Iowa\\ Iowa City\\ Iowa 52242 \\ USA}
\email{paul-muhly@uiowa.edu}

\begin{abstract}
This sequel to \cite{im2008} has three interconnected parts.  
In Part I, Holkar's theory of topological correspondences \cite{hol:JOT17_2} is applied to reveal a unitary equivalence between groupoid representations studied by Renault in \cite{ren:SIGMA14} and induced representations developed by Rieffel in \cite{rie:aim74}. 
In Part II, Part I is applied to the Deaconu-Renault groupoid, $G(X,\sigma)$, determined by a generic surjective local homeomorphism, $\sigma$, on a compact Hausdorff space $X$. 
Each such groupoid admits a pullback extension, $G_{\infty}(X,\sigma)$, that contains a nested family of projections $\{\bbE_j\}_{-\infty < j< \infty}$ and a unitary operator $U$ in the multiplier algebra of its $C^*$-algebra such that $U\bbE_jU^* = \bbE_{j-1}$ for all $j$. 
The pair they determine is called a \emph{proto-multiresolution analysis} (proto-MRA). 
Part I helps to secure the Hilbert space representation theory of proto-MRAs in terms of  representations $\pi$ of $C^*(G(X,\sigma))$ yielding systems $(\{\bbE_j\otimes I_{\pi}\}_{j= -\infty}^{\infty}, U\otimes I_{\pi})$  in such a way that reveals a criterion for deciding when a representation has the property that $\bigwedge (\bbE_i\otimes I_{\pi}) = 0$. When this equation is satisfied, then the family of ranges of the $\bbE_i\otimes I_{\pi}$  is called  a \emph{generalized multiresolution analysis}.
Part III applies the results from the first two parts to this criterion in order to determine  $\pi$'s that give rise to generalized multiresolution analyses. 
These $\pi$'s are parametrized by certain transfer operators for the endomorphism of $C(X)$ determined by $\sigma$.  
Finally, also in Part III, Mallat's famous theorems \cite[Theorem 1 and 2]{Mall_TAMS89} about scaling functions and quadrature mirror filters 
are given new proofs.
In this context, a scaling function is treated as a map from $X$ to $\bbR$ that implements a Hilbert space embedding of the Hilbert space for the generalized MRA into $L^2(\bbR)$.
The image of the generalized MRA under this embedding is called, simply, a \emph{multiresolution analysis} of $G(X,\sigma)$ in $L^2(\bbR)$.
The secret sauce in the construction of scaling functions is a method for ``punching holes'' in the potentials generating the transfer operators for $\sigma$.
Examples are given in the context of Blaschke composition operators and related fractal structures. 

This work was inspired in large part by \cite{Bagg_co_JFAA09,Bag_co_JFA10,LarRae_CM06,Larsen-Raeburn2007}, written by Iain Raeburn and co-authors.

\end{abstract}
\maketitle

\section{\label{sec:Intro,Backgrnd}Introduction}

\p In \cite{im2008}, we began to investigate how groupoids might be used to study wavelets. 
We were inspired by the interesting mix of harmonic analysis with the ``Cuntz Relations'' that had appeared recently in the literature (see Bratteli and Jorgensen's important book \cite{brajor_97} and the references cited therein), but we were also discomfited by a lack of evident unifying structure that behaved with some semblance of functoriality. 
The relations among the parameters involved were unclear to us and needed clarification.  
The pioneering paper of Mallat \cite{Mall_TAMS89} offered (and continues to offer) inspiration on how to develop a functorial setting for wavelets and fractal analysis. 
So, we shall begin there.

\p \label{def:MRA} To understand how Mallat's insights inspired us, recall the definition of a multiresolution analysis (MRA) that forms the foundation of the subject\footnote{In \cite{Mall_TAMS89}, Mallat used the term ``multiresolution approximation (of $L^2(\R)$)''. Subsequently ``MRA'' was adopted by the community of waveleteers. }.

\begin{defn}\label{defn:MRA} An MRA is a sequence $\{V_j\}_{j\in \bbZ}$ of closed subspaces of $L^2(\bbR)$ such that:
\begin{enumerate}[label=(MRA\arabic*),ref=MRA\arabic*, leftmargin=1.8cm]
    \item \label{eq:mranest} $V_j\subset V_{j+1}, j\in \bbZ$;
    \item \label{eq:mrafull} $\bigcup_{j= -\infty}^{\infty}V_j$ is dense in $L^2(\bbR)$; 
    \item \label{eq:mrapure} $\bigcap_{j= -\infty}^{\infty}V_j=\{0\}$;
    \item \label{eq:mrascale} $f(x)\in V_j$ iff $f(2x)\in V_{j+1},$ for all $j\in \bbZ$;
    \item \label{eq:mrascaleinv} if $f(x)\in V_j$, then $f(x-2^{-j}k)\in V_j$ for all $k\in \bbZ$; and
    \item \label{eq:mraunitequiv} translation by $\bbZ$ on $\bbR$ restricted to $V_0$ is unitarily equivalent to translation by $\bbZ$ on $\ell^2(\bbZ)$.
\end{enumerate}
\end{defn}

\subp\label{thm:distilledthm} To the uninitiated, it may be surprising that MRAs exist.  In fact, they exist in superabundance. The following theorem, which is distilled from Theorems 1 and 2 of \cite{Mall_TAMS89} begins to exhibit what lies at the heart of the proof of the existence of an MRA.

\begin{thm*}Each MRA may be constructed from a function\footnote{There are some smoothness conditions required on $H$ that we ignore now.}, $H$, on $\bbT$  that satisfies
    \begin{enumerate}
        \item $\vert H(1)\vert =1$
        \item $\vert H(z)\vert^2 + \vert H(-z)\vert^2 =1, \forall z.$
        \item $H(z)\neq 0$, when $\Re(z)>0$.
    \end{enumerate}
    The correspondence between such functions and MRAs is essentially bijective.
\end{thm*}

Each such $H$ is called a \emph{Quadrature Mirror Filter (QMF)} or a \emph{mother wavelet}. Of course, it may seem that we have just shifted one difficult problem to another. In a sense, we have.  However, the theorem gives an important change of focus. 

\subp\label{subp:proofoutline} A brief outline of the proof reveals the relation between MRAs and QMFs: Given $H$, the product $\Pi_{k=1}^{\infty}H(e^{2\pi i(2^{-k}\om)})$, $\om \in \bbR$, converges a.e. to a function, $\wh{\phi}$, in $L^2(\bbR)$ whose Fourier transform, $\phi$, lies in the space $V_0$ of an MRA and the sequence $\{\sqrt{2^j}\phi(2^jx-k)\}_{k\in \bbZ}$ is an orthonormal basis for $V_j$ for each $j$. $\phi$ is called the \emph{scaling function} for the MRA. Conversely, given an MRA, item \eqref{eq:mraunitequiv} guarantees the existence of scaling function, whose Fourier transform, $\wh{\phi}$, satisfies
    \begin{equation}\label{eq:QMF}
        \wh{\phi}(2\om) = M(\om)\wh{\phi}(\om),\qquad \om \in \bbR,
    \end{equation}
    where $M$ is a $2\pi$-periodic function on $\bbR$. This $M$ passes to a QMF $H$ on $\bbT$.

Mallat's path between the $H$'s and wavelets is serpentine. The structure of an $H$ is revealed through its interaction with the irreversible dynamical system on $\bbT$\[
    z\to z^2\qquad z\in \bbT.
    \] 
In many, if not all, the myriad variations on, or generalizations of, Mallat's analysis, there are local homeomorphisms in the background that play fundamental roles similar to that of $z\to z^2$. We are thus led to study these, which we do through the lens of groupoid harmonic analysis that is focused in particular on the \emph{Deaconu-Renault groupoid} of a local homeomorphism.

\p \label{par:Outline}
Our work involves two areas of mathematics that are not ordinarily seen as overlapping. 
It is the possibility of a fruitful overlap that intrigues us and has driven this investigation.
So, to make our presentation understandable to the broadest possible audience, we have tried to provide a soupcon of background to each step we take. 
In consequence, to clarify the significance of these steps the paper is divided into 12 sections, besides this one, and these, in turn, may be lumped into parts.
Part I consists of Sections \ref{sect:GpdsWHSysts} - \ref{sec:induc-repr}, in which we develop some general background about groupoid harmonic analysis. 
The key feature about groupoids that we want to emphasize is that they often provide useful coordinates for $C^*$-algebras of interest.
Further, algebraic operations on groupoid $C^*$-algebras are often mirrored directly in the groupoids.
In particular, Rieffel's theories of induced representations and Morita equivalence are reflected in the notion of equivalence of groupoids first developed in \cite{mrw:jot87}. 
Further, still, what proves especially important for the purposes of this paper and is the vessel for our ``secret sauce'', is the recent contribution of Rohit Dilip Holkar \cite{hol:JOT17_2} in which he advances the notion of a \emph{topological correspondence} and the attendant concept of an \emph{adjoining function}. 
The reader should bear in mind that $C^*$-algebras form a category in which a morphism from a $C^*$-algebra $A$ to a $C^*$-algebra $B$ is a certain type of $A-B$-bimodule called a \emph{$C^*$-correspondence}. This is a \emph{directed} construct: If $\cX$ is such a bimodule, one says that $\cX$ is a $C^*$-correspondence \emph{from} $A$ \emph{to} $B$. 
In a sense that we will make clear, topological correspondences are the morphisms in the category whose objects are groupoids with Haar systems. 
The equivalences from \cite{mrw:jot87} are the isomorphisms. 

Part II, which covers Sections \ref{sect:DR-gpd} -- \ref{sec:PMRAs}, begins in Section \ref{sect:DR-gpd} with some review materials from \cite{im2008} which was devoted entirely to the $C^*$-algebra of the Deaconu-Renault groupoid $G(X,\sigma)$ of a generic local homeomorphism $\sigma$ of a compact Hausdorff space $X$.
Section \ref{sec:PBsIRs} is devoted to building a new groupoid from $G(X,\sigma)$, denoted $G_{\infty}(X,\sigma)$, which is the ``pullback'' of $G(X,\sigma)$ by the limit, $X_{\infty}$, of the canonical projective system built from iterates of $\sigma$.

The groupoid $G_{\infty}(X,\sigma)$ is equivalent to $G(X,\sigma)$ in the sense of \cite{mrw:jot87} but, by virtue of being a pullback, the $C^*$-algebra of $G_{\infty}(X,\sigma)$  carries a copy of $C^*(G(X,\sigma))$ inside it and, as we show in Section \ref{sec:CondExps}, there is a conditional expectation from $C^*(G_{\infty}(X,\sigma))$ onto the copy of $C^*(G(X,\sigma))$.
We use this conditional expectation coupled with  Holkar's analysis to build a  $C^*$-correspondence $\cX$ from $C^*(G_{\infty}(X,\sigma))$ to $C^*(G(X,\sigma))$ that may \emph{not} be invertible in the collection of $C^*$-correspondences from $C^*(G_{\infty}(X,\sigma))$ to $C^*(G(X,\sigma))$. Further, in the space of bounded adjointable operators on $\cX$, $\cL(\cX)$, we build a sequence of projections $\{\bbE_j\}_{j= -\infty}^{\infty}$ and a unitary operator $U$ with the following properties:
\begin{enumerate}[label=(PMRA\arabic*), ref=PMRA\arabic*, leftmargin=2cm]\label{def:protoMRA}
    \item\label{eq:pmra} $\bbE_j\leq \bbE_i$ when $j\leq i$,
    \item \label{eq:pmrafull} $I_{\cX}=\bigvee_{-\infty < j < \infty} \bbE_j$,
    \item \label{eq:pmrascal} $U\bbE_jU^*= \bbE_{j-1}$ for all $j$, and
    \item \label{eq:pmraunitequiv} the copy of $C^*(G(X,\sigma))$ in $C^*(G_{\infty}(X,\sigma))$ leaves the range of $\bbE_j\ominus \bbE_{j-1}$ invariant and $U^{j-i}$ implements a unitary equivalence between the restriction of the copy of $C^*(G(X,\sigma))$ to the range of $\bbE_j\ominus \bbE_{j-1}$ and the restriction of the copy to the range of $\bbE_i\ominus \bbE_{i-1}$.
\end{enumerate}
We call the pair $(\{\bbE_j\}_{j= -\infty}^{\infty}, U)$ a \emph{proto-multiresolution analysis} because of the evident relations with Definition \ref{defn:MRA}.
Equation \eqref{eq:pmra} is a clear variant of equation \eqref{eq:mranest}, while \eqref{eq:pmrafull} corresponds to \eqref{eq:mrafull}. 
Equations \eqref{eq:mrascale}, \eqref{eq:mrascaleinv} and \eqref{eq:mraunitequiv} are captured in part by \eqref{eq:pmrascal} and  \eqref{eq:pmraunitequiv}.

What is conspicuously absent from the defining properties of a proto-multiresolution analysis is anything that relates to \eqref{eq:mrapure}.
Now, we do not pursue all the possibilities here.  
They are not necessary for our immediate purposes.
Rather, we take our next step and turn our attention to what happens when we choose a $C^*$-representation $\pi$ of $C^*(G(X,\sigma))$ acting on a Hilbert space $\cH_{\pi}$. We form the induced representation $
\Ind_{\pi}$ of $C^*(G_{\infty}(X,\sigma))$, which acts on $\cX\otimes_{\pi}H_{\pi}$, and form the induced system of operators, $\{\bbE_j\otimes I\}_{j= -\infty}^{\infty}$ and $U\otimes I$, which live in 
$B(\cX\otimes H_{\pi})$. Then, when we form the spaces $V_j:=\bbE_j\otimes I(\cX\otimes H_{\pi})$, $-\infty < j < \infty$, we obtain a family of \emph{Hilbert spaces} that satisfy the evident analogues of 
\eqref{eq:pmra}-\eqref{eq:pmraunitequiv}.   But also, as we show in Theorem \ref{thm:GMRA}, we can determine, in terms of $\pi$, when $\cap_j V_j=\{0\}$. When this intersection is $\{0\}$, then we say that
$\{V_j\}_{-\infty < j < \infty}$ is a \emph{generalized multiresolution analysis}.

In Part III, which comprises sections \ref{sec:pot,filt,scalfns}-\ref{sec:arbitraryscale2}, we analyze when representations $\pi$ of $C^*(G(X,\sigma))$ give rise to generalized multiresolution analyses.
Our analysis culminates in Theorem \ref{thm:tS_fu_pure} which shows that $\pi$ gives rise to a generalized MRA if and only if the endomorphism of $C(X)$ determined by $\sigma$ admits a so-called transfer operator with properties that are fairly easy to determine.
We apply Theorem \ref{thm:tS_fu_pure} to $\sigma$s which are local homeomorphisms of the circle $\bbT$ determined by finite Blaschke products in Section \ref{sec:arbitraryscale2}.
In Theorem \ref{thm:Mallat} we give our proof of Mallat's theorem from subparagraph \ref{thm:distilledthm}.
Our proof reveals the role of a scaling function: It implements a Hilbert space isomorphism between $\cX\otimes H_{\pi}$ and $L^2(\bbR)$ that maps the $\{V_j\}_{-\infty < j <\infty}$ to a complete nest in $L^2(\bbR)$ that faithfully reflects the features of $\Ind_{\pi}$. Thus the scaling function is the link between the dynamical system $(X,\sigma)$
 and $L^2(\bbR)$. So, in the event there is a scaling function, we refer to all the structure it reveals simply as a \emph{multiresolution analysis} in $L^2(\bbR)$.

Now scaling functions are constructed by carefully punching holes in the potentials that are used to generate transfer operators for $\sigma$. (See Lemma \ref{lem:pos_sup}.)  These potentials, in turn, are built from Holkar's adjoining functions, which are the key parameters needed to construct the topological correspondence from $G_{\infty}(X,\sigma)$ to $G(X,\sigma)$. 

We called Mallat's path between QMFs and wavelets ``serpentine''. The reader may feel that our approach is much more so. However, it has three advantages. First, and foremost, it puts into evidence precisely the roles that the various parameters play. Second, it provides new proofs of other results in the literature. We have in mind, in particular, \cite[Theorem 3.1]{brajor_97},\cite[Theorem 8]{Bagg_co_JFAA09}, and \cite[Theorem 3.1]{Bag_co_JFA10}. Third, it opens up a whole new world of examples that deserves to be explored.

\section*{{\Large{Part I}}\\Selected Generalities about Groupoids}

\section{Groupoids and Haar Systems}\label{sect:GpdsWHSysts}

\p
We follow the books by Jean Renault \cite{Ren_LNM793} and Dana Williams \cite{WilliamsGpdBook} for basic facts about groupoids and notation. A generic second countable, locally compact, Hausdorff groupoid will be denoted by $G$. Its unit space will be denoted $\go$ and $r$ and $s$ will denote the \emph{range} and \emph{source} maps, mapping $G$ to $\go$. It is part of the definition of ``locally compact groupoid'' that these maps are continuous and open.  Particular groupoids that we encounter and study will be denoted by $G$'s with additional labels. This convention will  reinforce the functorial relations among the constructs considered.

\subp \label{subp:warning} A key feature of groupoid theory is the superabundance of fibred products: Suppose $X$, $Y$, and $Z$ are three sets and $f:X\to Z$ and $g:Y\to Z$ are two maps then the \emph{fibred product} of $X$ and $Y$ determined by $f$ and $g$ is
\[
X {_f \ast_g} Y:=X\ast Y:= \{(x,y)\in X\times Y \mid f(x)=g(y)\}.
\]
This feature of the theory leads to multiple notation for the same object. 
The choice of notation is dictated by properties of the objects under consideration and the emphases necessary in the moment.

\subp 
It will be helpful also to emphasize at the outset some notions that are developed in \cite{anaren:amenable00} by Claire Annantharaman-Delaroche and Renault.
\begin{defn}\label{def:systsmaps}\cite[Definition 1.1.1]{anaren:amenable00}
Let $Y$ and $X$ be two locally compact spaces and let $\pi:Y\to X$ be a continuous map. A \emph{continuous $\pi$-system} (or a \emph{continuous system of measures on $Y$ adapted to $\pi$}) is a family $\{\alpha^x\mid x\in X\}$ of positive Radon measures on $Y$ such that
\begin{enumerate}
\item the support of each $\alpha^x$, $\supp{\alpha^x}$, is contained in $\pi^{-1}(x)$, and
\item the map $\alpha$ defined on $C_c(Y)$ by the formula $\alpha(f)(x)= \int f\, d\alpha^x$ carries $C_c(Y)$ into $C_c(X)$.  
\end{enumerate}
We write $\alpha = \{\alpha^x\}_{x\in X}$ and call $\alpha$ \emph{proper} provided $\alpha^x\neq 0$ for all $x\in X$, while we call $\alpha$  \emph{full} in case  the support of each $\alpha^x$ is all of $\pi^{-1}(x)$.
\end{defn}

\begin{remark}\label{rem:RevAltWay} A useful and revealing alternate way to view a continuous $\pi$-system is this: The map $\pi:Y\to X$ makes $C_c(Y)$ a right module over $C_c(X)$ via the formula
\begin{equation}
    \phi\cdot\psi(y):= \phi(y)\psi\circ \pi(y), \qquad \phi\in C_c(Y),\psi\in C_c(X), y\in Y.
\end{equation}
Also, one may view $C(X)$ as a right module over itself induced by the identity map on $X$.
So, from the perspective of these modules, a continuous $\pi$-system, $\alpha$, defines a $C_c(X)$-module map from $C_c(Y)_{C_c(X)}$ to $C_c(X)_{C_c(X)}$ that maps nonnegative functions on $Y$ to nonnegative functions on $X$. The Riesz representation theorem implies, conversely, that every such module map defines a continuous $\pi$-system. Further, it is a consequence of what we like to call the Williams-Blanchard theorem \cite[Corollary B.18]{WilliamsGpdBook} that \emph{full} continuous $\pi$-systems exist if and only if $\pi$ is an open map.
\end{remark}

\begin{defn}\label{def:gpdAction}  
A groupoid $G$ \emph{acts} on a set $X$ (on the \emph{left}) in case there is a map $\rho: X\to \go$, called the \emph{moment map} or \emph{anchor map}, and a composition map \verb|\cdot|,  $\cdot$, from the fibred product $G\,{_s\ast_{\rho}}\,X$ to $X$ satisfying
    \begin{enumerate}
        \item $\rho(x)\cdot x=x$, for all $x\in X$, 
        \item $\rho (\gamma \cdot x)= r(\gamma)$, for all $(\gamma,x)\in G\,{_s\ast_{\rho}}\,X$ and
        \item $(\alpha\beta)\cdot x= \alpha\cdot(\beta\cdot x)$ for all $\alpha$, $\beta$, and $x$ such that $s(\alpha)=r(\beta)$, $s(\beta)=\rho(x)$.
        \end{enumerate}
        If $G$ is a topological groupoid and $X$ is a topological space, then a $G$ action on $X$ is continuous if all the maps involved are continuous and, in addition, the moment map is open\footnote{In \cite[Remark 2.2]{WilliamsGpdBook} Dana Williams comments on the assumption that the moment map is open, noting in particular that it is not viewed as necessary in many points in the literature. We include the assumption in our definition because we envision frequently occurring  circumstances when it becomes necessary for our analyses.}.
\end{defn}

Right actions of groupoids are defined similarly.  One needs a moment map $\sigma:X \to \go$\footnote{This use of $\sigma$ should not be confused with the local homeomorphism $\sigma$ that is the primary object of study in this paper (see Section \ref{sect:DR-gpd} et.seq.) It's use here is to signal an analogue of the source map in a groupoid and will not show up when we are discussing the groupoids associated to our local homeomorphism.} and a composition map, $\cdot$\,, that maps the fibred product $X\,{_{\sigma}\ast_{r}}\,G$ to $X$ so that the analogues of the equations in Definition \ref{def:gpdAction} hold. For example, $x \cdot (\alpha\beta)=(x\cdot \alpha)\cdot \beta$ for all $\alpha$, $\beta$, and $x$ such that $\sigma(x)= r(\alpha)$, and $\alpha$ and $\beta$ are composable in $G$.

\begin{defn}\label{def:invAct}
Let $G$ act continuously on $X$ and on $Y$ and let $\pi:Y\to X$ be an invariant continuous map, i.e., $\pi$ intertwines the two actions - $\pi(\gamma \cdot x)=\gamma \cdot \pi(x)$, then a continuous $\pi$-system of measures $\alpha = \{\alpha^x\}_{x\in X}$ is \emph{invariant} for the actions in case the map $\alpha$ viewed as a map from $C_c(Y)$ to $C_c(X)$ is invariant, i.e., 
\begin{equation}
    \int f(\gamma\cdot x)\,d\alpha^{s(\gamma)}(x) = \int f(x)\,d\alpha^{r(\gamma)}(x).
\end{equation}
\end{defn}

One sometimes uses ``equivariant'' for what we are calling ``invariant'' when discussing groupoid actions and $\pi$-systems.

\subp If $X=Y=G$, while $\pi$ is the identity map, and if $\rho$ is the range map, $r$, then a \emph{full} invariant system of measures is a \emph{Haar} system on $G$.  
Typically, we will denote Haar systems by $\lambda$, and write also $\lambda = \{\lambda^u\}_{u\in G^0}$. 
Together with a Haar system $\lambda$, the space $C_c(G)$, consisting of continuous complex-valued functions on $G$ endowed with the inductive limit topology, becomes a locally convex topological $*$-algebra, called the \emph{convolution algebra} of $G$ and $\lambda$ and denoted $C_c(G,\lambda)$. 
The product is defined by the formula
\begin{equation}
    f*g(\gamma):=\int_G f(\alpha)g(\alpha^{-1}\gamma)\, d\lambda^{r(\gamma)}(\alpha),\qquad f,g\in C_c(G,\lambda),
\end{equation} and the involution is defined by 
\begin{equation}
f^*(\gamma):=\overline{f(\gamma^{-1})},\qquad f\in C_c(G,\lambda).
\end{equation}

The $C^*$-algebra of $G$ and a Haar system $\lambda$ is   the enveloping $C^*$-algebra of $C_c(G,\lambda)$ and will be denoted $C^*(G,\lambda)$ (See \cite[Section 2.1]{Ren_LNM793} and \cite[Section 1.4]{WilliamsGpdBook}). 
If the Haar system $\lambda$ is fixed, we abbreviate the notation for the $C^*$-
algebra to $C^*(G)$.

\p \label{par:HspaceReps} Here we provide a telegraphic account of groupoid representations that involve Hilbert \emph{spaces}.  Later, we shall discuss representations of groupoids on certain Hilbert $C^*$-\emph{modules}. 

\subp \label{subp:HspaceReps} Representations of groupoids take place on Hilbert bundles that are fibred over their unit spaces. 
The representations of the associated convolution algebras take place on the Hilbert spaces obtained by integrating the fibres of the bundles using quasi-invariant measures on the unit spaces \cite[Chapter 3]{WilliamsGpdBook}. 
How to build quasi-invariant measures for our groupoids is the focus of much of Parts II and III  of this paper.

\subp\label{subp:qimeas} If $G$ is a locally compact groupoid with Haar system $\lambda:=\{\lambda^u\}_{u\in \go}$ and if $\mu$ is a (positive) measure on $\go$, then one may form the measure on $G$ that is denoted $\mu \circ \lambda$, and is defined by the formula:
\begin{equation}
    \int_G f(\gamma)\, d\mu \circ \lambda(\gamma):= \int_{\go}\int_G f(\gamma)\, d\lambda^{r(\gamma)}(\gamma)\,d\mu(r(\gamma)),\qquad f\in C_c(G)\footnote{This definition of $\mu \circ \lambda$ may be a bit discomfiting when the right hand side of the defining equation suggests that one should write $\lambda \circ \mu$. However, nowadays measures are viewed as maps acting on function spaces. So, from that perspective, one should write integrals as we did on the left hand side.}.
\end{equation}
Composing $\mu\circ \lambda$ with the inverse map on $G$ creates a new measure, denoted $(\mu\circ \lambda)^{-1}$. The measure $\mu$ is called \emph{quasi invariant} in case $\mu\circ \lambda$ and $(\mu\circ \lambda)^{-1}$ are mutually absolutely continuous, in which case, the function $\Delta_\mu:=\frac{d(\mu\circ \lambda)}{d(\mu \circ \lambda)^{-1}}$ is called the \emph{modular function of $\mu$ \textbf{and} $\lambda$}. The modular function $\Delta_\mu$, while defined only almost everywhere with respect to $\mu\circ \lambda$ can be chosen so that it is a homomorphism from $G$ to the positive real numbers \cite[Proposition 7.6]{WilliamsGpdBook}.
The modular function, $\Delta_\mu$, is also called the \emph{modular cocycle}. It is of interest to know when $\Delta$ is a \emph{coboundary}, meaning that there is a positive function $b$ on $\go$ such that 
$\Delta_\mu(\alpha\beta)
=b(\alpha)b(\beta)^{-1}$
for all $\alpha,\beta\in G$.
In that event, the measure $\tilde{\mu}:=b\cdot \mu$, is invariant, i.e., $\tilde{\mu}\circ \lambda = (\tilde{\mu}\circ \lambda)^{-1}$.

\subp \label{subp:RepsofG} A \emph{(unitary) representation} of a groupoid $G$ is a \emph{triple} $(\mu,\go\ast\cH,\hat{L})$, where $\mu$ is a quasi-invariant measure on $\go$, $\go\ast\cH$ is a Borel Hilbert bundle on $\go$, and $\hat{L}:G\to \Iso(\go\ast \cH)$ is a homomorphism of $G$ into the unitary groupoid, $\Iso(\go\ast \cH)$, of the bundle $\go\ast \cH$. The \emph{triple} $(\mu,\go\ast\cH,\hat{L})$ is usually abbreviated to $\hat{L}$, but doing so sometimes (often?) causes confusion. 

The ingredients that make up a representation require a bit more explanation. 
The Hilbert bundle  $\go\ast\cH$ is a Borel space with a Borel projection $\pi$ onto $\go$ with the property that each fibre of $\pi$, $\pi^{-1}(u)$, is a Hilbert space $H_u$. 
Further, $\go*\cH$ has a structure that allows one to add global (cross) sections to $\pi$ and to multiply them by complex scalars in a fashion that is consistent with the structures on the fibres. Nevertheless, one often writes, informally, $\go\ast \cH=\{H_u\}_{u\in \go}$.
We omit the details and refer the reader to \cite[3.5]{WilliamsGpdBook} for an expatiation.  

\subp\label{subp:intreps}The 
unitary groupoid of $\go\ast \cH$ is $\text{Iso}(\go\ast \cH):=\{(u,U,v)\mid U\in B_{\text{iso}}(H_v,H_u)\}$, where $B_{\text{iso}}(H_v,H_u)$ is the space of all Hilbert space isomorphisms from $H_v$ to $H_u$. It carries a Borel structure induced by that of $\go*\cH$. The product of $(u,U,v)$ and $(w,V,z)$ is defined only when $v=w$ and in that case, the product is $(u,UV,z)$.  Since $G$ carries a  Borel structure that is generated by it's topology, it makes perfectly good sense to insist that $\hat{L}$ be a Borel homomorphism from $G$ to $\text{Iso}(\go\ast \cH)$ cf.\cite[Proposition 3.38]{WilliamsGpdBook}.

Since $\hat{L}$ is a homomorphism from $G$ into $\Iso(\go\ast \cH)$, it must be written to fit the notation for $\Iso(\go\ast \cH)$. So, we write 
\begin{equation}
    \hat{L}(\gamma):=(r(\gamma),L_{\gamma},s(\gamma)),\qquad \gamma\in G, 
\end{equation}
where $L_{\gamma}$ lies in $B_{\text{iso}}(H_{s(\gamma)},H_{r(\gamma)})$.

\subp\label{subp:BundleStruct} Now, direct integral theory allows one to write a given bundle as a direct sum of sub-bundles each isomorphic to a bundle of the form $U\times H_U$ for a suitable subset $U\subset \go$ and Hilbert space $H_U$. 
Further, it can be arranged so that the sets and spaces may be chosen in such a way that $H_U$ and $H_V$ have the same Hilbert space dimension if and only if $U=V$.
In that event, of course, two Hilbert bundles with the same structure (i.e, direct sum decomposition, dimension function, etc.) are isomorphic, \emph{but not naturally so}. Each isomorphism involves choices of measurable cross sections of $\pi$. These exist in our situation\footnote{There are many proofs of this statement, usually called the \emph{principal of measurable choice}. Our favorite is the proof of Theorem 1 of Ed Azoff's brief note \cite{Azoff74}.}, but definitely they are not unique. 

Given a Borel Hilbert bundle $\go\ast \cH$, and (sigma-finite) Borel measure $\mu$, one can form the Hilbert \emph{space} consisting of all its Borel sections $\fh$ such that $\int_{\go}\langle \fh(u),\fh(u)\rangle_u\, d\mu(u)$ is finite. This space will be denoted $L^2(\go\ast \cH,\mu)$ and the inner product of two sections $\fh$ and $\fk$ will be given by the expected integral
\begin{equation}
    \bigl( \fh,\fk\bigr): = \int_{\go}\bigl(\fh(u),\fk(u)\bigr)_{H_u} \, d\mu(u)\footnote{Of course, there is nothing special about $\go$, $\mu$ and $\go*\cH$. The basic features are the same for any Hilbert bundle built on a space $Y$, $Y\ast \cK$, with a sigma finite measure $\nu$.}
.
\end{equation} 

\subp\label{subp:1dimfibres} When the dimensions of the fibres are one, we may choose a non-vanishing cross section $\fk$ of norm one in $L^2(\go\ast \cH,\mu)$. 
Then, if $\fh$ is another cross section, it can be written as $\fh=c\cdot \fk$, where $c$ is a uniquely determined complex-valued Borel function on $\go$. 
The map $\fh\to c$ is a Hilbert space isomorphism from $L^2(\go\ast \cH,\mu)$ onto the space of scalar-valued, square-integrable functions on $\go$, $L^2(\go,\mu)$, that we denote by $V$ and call an \emph{(archetypal) decomposable isomorphism} between $L^2(\go\ast \cH,\mu)$ and $L^2(\go,\mu)$. Evidently, $V$ is decomposable because it intertwines the actions of multiplication by functions in $L^{\infty}(\go,\mu)$ on $L^2(\go\ast \cH,\mu)$ and on $L^2(\go,\mu)$. Further, while it is determined by the choice of $\fk$, any other $\fk'$, that is nonvanishing and has norm $1$, is related to $\fk$ by an equation $\fk'=c\cdot \fk$, where $c:X\to \bbT$ is a Borel function. Consequently, the corresponding isomorphism $V'$ is related to $V$ through multiplication by $c$.  Therefore, for all our intents and purposes, we may replace $L^2(\go\ast\cH,\mu)$ with $L^2(\go,\mu)$ and carry out computations there.

\subp\label{subp:AbstRepTheory}Now, let $(\mu,\go\ast\cH,\hat{L})$ be a unitary representation of $G$ in
the unitary groupoid, $\Iso(\go\ast\cH)$. The 
\emph{integrated form}, $\fL$, of $L$ represents the convolution algebra of $G$, $C_c(G,\lambda)$, on  $L^2(\go*\cH,\mu)$ via the formula
\begin{multline}\label{eq:repgenG}
    \fL(f)\fh(r(\gamma)):=\int_{G}f(\gamma)L_{\gamma}(\fh(s(\gamma))\Delta(\gamma)^{-1/2}\, d\lambda^{r(\gamma)}(\gamma),\\  f\in C_c(G,\lambda), \fh\in L^2(\go\ast \cH,\mu).
\end{multline}
This formula extends to all Borel functions $f$ such that $\fL(f)$ makes sense as a bounded operator on $L^2(\go\ast \cH,\mu)$.  However, it is difficult to identify precisely which $f$ have this property. 

The formula reveals a convention that we shall follow: A representation of $G$ into the unitary groupoid of a bundle will always be written with Roman or Greek letters, the integrated form of the representation acting on sections will be denoted by the corresponding fractur letters. 
The sections, in turn, will also be written with fractur letters, except in the setting when the fibres all have dimension one. In that case, as we discussed in subparagraph \ref{subp:1dimfibres}, $L^2(\go\ast \cH,\mu)$ is decomposably, isometrically isomorphic to $L^2(\go,\mu)$, and so all our computations may be carried out in $L^2(\go,\mu)$. 

\p \label{sub:HilbMod_and_Holkar} 
In this subsection, we want to provide some background about Hilbert $C^*$-modules over groupoid $C^*$-algebras and, most important, we want to use it to reveal the source of our ``secret sauce'' mentioned in the abstract. It is based on some ideas of Holkar that were conceived in his thesis \cite{holk_phd} and developed further in \cite{hol:JOT17_2}.

Suppose $A$ and $B$ are $C^*$-algebras and $\cX$ is (right) Hilbert module over $B$ with a  $B$-valued inner product $\langle \cdot , \cdot \rangle$ in the sense of Lance \cite[Chap. 1]{lan:hilbert}.
We say that $\cX$ is a $C^*$-\emph{correspondence from $A$ to $B$} in case $A$ is represented in the algebra of all bounded adjointable operators on $\cX$, $\cL(\cX)$
\cprotect\footnote{Usually, we shall denote
the left action of $A$ on $\cX$ 
simply using a \verb|\cdot|, but 
there are occasions when it is 
helpful to write $L(a)\xi$ for 
$a\cdot \xi$.}.
In this event, following Rieffel \cite{rie:aim74}, one may induce $C^*$-representations of $A$ from $C^*$-representations of $B$ simply by forming balanced tensor products, as follows: For a $C^*$-representation of $B$, $\pi$, on the Hilbert space $\cH$, $\cX\otimes_{\pi} \cH$ becomes a Hilbert space in the inner product that is defined on decomposable tensors by 
$
\langle \xi \otimes h, \eta \otimes k\rangle: = \langle h,\pi(\langle \xi, \eta\rangle)k\rangle.
$
The representation, $\Ind_{\pi}:A\to B(\cX \otimes_{\pi} \cH)$, then, is given by $\Ind_{\pi}(a)(\xi \otimes h):=(a\cdot \xi)\otimes h,$ $a\in A$, $\xi\otimes h \in \cX\otimes \cH$, i.e., $\Ind_{\pi}(a) = L(a)\otimes I_{\cH}$.

Suppose now that $A=C^*(H,\beta)$ and $B=C^*(G,\lambda)$ for groupoids $H$ and $G$ endowed with Haar systems $\beta$ and $\lambda$. 
Holkar showed how to build a correspondence from $A$ to $B$ using what he calls a topological correspondence from $(H,\beta)$ to $(G,\lambda)$.
We review Holkar's definition and construction since our notation is slightly different from his.

\begin{defn}\label{def:topcorr}
A topological correspondence from $(H,\beta)$ to $(G,\lambda)$ is a pair $(Z,\alpha)$ where:
\begin{enumerate}
        \item $Z$ is a locally compact $H-G$ bispace, meaning that $Z$ is a right $G$-space with moment map $\sigma$ and a left $H$-space with moment map $\rho$ such that $\eta\cdot (z\cdot \gamma)=(\eta\cdot z)\cdot \gamma$, for all $\eta\in H$, $\gamma \in G$ and $z\in Z$ for which  
        the formula is defined.
        \item The action of $G$ on $Z$ is proper in the sense of \cite[Definition 2.15]{WilliamsGpdBook}.
        \item $\alpha=\{\alpha_u\}_{u\in \go}$ is a $G$-invariant, continuous $\sigma$-system of measures on $Z$ in the sense of Definition \ref{def:invAct} that is also proper in the sense of Definition \ref{def:systsmaps}.
        \item There is a continuous positive function $\Delta_a$ defined on $H*Z$ such that
        \[
        \int_{Z}\int_H F(\eta^{-1},z)\,d\beta^{\rho(z)}(\eta)d\alpha_u(z)=\int_Z\int_{H} F(\eta,\eta^{-1}z)\Delta_a(\eta,\eta^{-1}z)\,d\beta^{\rho(z)}(\eta)d\alpha_u(z),
        \]
        for all $u\in \go$ and $F\in C_c(H*Z)$. The function $\Delta_a$ is called the \emph{adjoining function} for $(Z,\alpha)$.
    \end{enumerate}
\end{defn}
We want to emphasize that the system of measures $\alpha$ must be proper but \emph{not} necessarily full. We use the notation $\Delta_a$ for the adjoining function to 
distinguish it from the modular function, $\Delta$ or $\Delta_\mu$, of a quasi-invariant measure $\mu$.

The main theorem of \cite{hol:JOT17_2}, Theorem 2.10, proves that $C_c(Z)$ endowed with the following operations can be completed to a $C^*$-correspondence from 
$C^*(H,\beta)$ to $C^*(G,\lambda)$\footnote{Algebra-valued inner products will be denoted with $\langle$ and $\rangle$, $\bbC$-valued inner products will be denoted with parentheses.}. :
\begin{align}
  (\xi\cdot f)(z) &= \int_G \xi(z\gamma)f(\gamma^{-1})\,d\lambda^{\sigma(z)}(\gamma), \\
  \langle\,\xi\, ,\zeta\,\rangle_{C_c(G)}(\gamma) &= \int_X \overline{\xi(z)}\zeta(z\gamma)\,d\alpha_{r(\gamma)}(z), \\
  (g\cdot \xi)(z) &= \int_{H} g(\gamma)\xi(\gamma^{-1}z)\Delta_a^{1/2}(\gamma,\gamma^{-1}z)\,d\beta^{\rho(z)}(\gamma)
\end{align}
for all $f\in C_c(G,\lambda)$, $g\in C_c(H,\beta)$, and $\xi,\eta\in C_c(Z)$.

\subp\label{subp:Equivalence_as_corresp}
An equivalence $(Z,\alpha)$, where $\alpha$ is a full $\sigma$-system, between $(H,\beta)$ and $(G,\lambda)$ in the sense of \cite[Definition 2.1]{mrw:jot87}
is a particular example of a topological correspondence. The adjoining function $\Delta_a$ is the constant function 1 in this situation (\cite[Example 3.9]{hol:JOT17_2}).
The $C^*$-algebras $C^*(H,\beta)$ and $C^*(G,\lambda)$ are Morita equivalent (\cite{mrw:jot87}) and the $C^*$-correspondence defined above becomes an imprimitivity bimodule. 
We will specialize this setting further in Sections \ref{sect:DR-gpd} and \ref{sec:PMRAs} by setting $H=G$, $Z=G$, and $\alpha_x=\lambda_x$ for all $x\in X$,  where $\{\lambda_{u}\}_{u\in\go}$ is the
\emph{right} Haar system determined by $\lambda$, $\int f(\gamma)d\lambda_{u}(\gamma):=\int f(\gamma^{-1})d\lambda^{u}(\gamma)$. Then $(G,\alpha)$ is an equivalence between $(G,\beta)$ and $(G,\lambda)$, where
$\beta$ can be a different Haar system on $G$. For our purposes, we view $(G,\alpha)$ as a topological correspondence from $(G,\lambda)$ to $(G,\lambda)$ and we write $\cX_0$ for the 
$C^*$-correspondence from $C^*(G,\lambda)$ to $C^*(G,\lambda)$ that $(G,\alpha)$ defines.

\section{Pullbacks, Imprimitivity Groupoids and Haar Systems}{\label{sec:Haar-Systems}


We want to develop some features of imprimitivity groupoids with  Haar systems that will be tailored to our special setting later.  
Throughout this section, $G$ will be a groupoid with  Haar system $\lambda$, $Y$ will be a second countable, locally compact Hausdorff space and $\Phi:Y\to \go$ will be a continuous, open, and surjective map. 
We will explore  Haar systems on two related groupoids, $G^Z$ and $Y\ast G \ast Y$, that are built from $G$, $\lambda$, $\Phi$, and continuous $\Phi$-systems of measures on $Y$.

\p\label{par:GsuperZ} The space $Z$ is defined to be the fibred product,
\begin{equation}
Z:=Y\,{_{\Phi}*_r}\,G:=\{(y,\gamma)\in Y\times G\,:\,\Phi(y)=r(\gamma)\}.\label{eq:Z}
\end{equation}
It carries the evident right $G$ action through right translation on $G$:
\[
(z,\alpha),\beta \to (z,\alpha\beta), \qquad (z,\alpha)\in Z\ast G, \, \beta\in G.
\]
This action of $G$ on $Z$ is free and proper because the action of
$G$ on $G$ via right translation is free and proper. 

To define the imprimitivity groupoid
determined by $G$ and $Z$, $G^Z$, let $\sigma:Z\to G$ be the map, $\sigma(y,\gamma):=s(\gamma)$ and set 
\[
Z*Z:=\{(x,y)\in Z\times Z\mid \sigma(x)=\sigma(y)\}.
\]
Then $G$ acts on the right on $Z*Z$ by the diagonal action
$(x,y)\cdot\gamma=(x\gamma,y\gamma)$. The \emph{imprimitivity groupoid}
determined by $G$ and $Z$ is $G^{Z}:=Z*_{G}Z:=(Z*Z)/G$.
The unit space of $G^{Z}$ can be identified with $Z/G$, which in turn may be identified with $Y$, and the range
and source maps on $G^{Z}$ are given by
\[
r([x,y])=x\cdot G\quad\text{and}\quad s([x,y])=y\cdot G.
\]

The space $Z$ becomes a free and proper left $G^{Z}$-space as follows. The moment map $\rho:Z\to Z/G$ is simply the quotient
map. The space $G^{Z}*Z$, then, is $\{([x,y],z)\in G^{Z}\times Z\mid y\cdot G=z\cdot G\}$.
The left action of $G^{Z}$ on $Z$ is
defined via the formula
\[
[x,y]\cdot z:=x\gamma,
\]
where $\gamma$ is the unique element in $G$ such that $z=y\cdot\gamma$.
The space $Z$ is a free and proper left $G^{Z}$-space and, in fact,
$Z$ is an equivalence between $G^{Z}$ and $G$ (\cite[Section 2]{mrw:jot87}).

\p\label{par:pullback}The next proposition provides an alternate description of the imprimitivity groupoid  $G^{Z}$.
\begin{prop}
\label{lem:Zgisomorphicto} The imprimitivity groupoid $G^{Z}$ is isomorphic to the \emph{pullback groupoid determined by $Y$ and $\Phi$} \footnote{The term ``blow up'' is also used in the operator algebra literature. However, we prefer the term ``pullback'' because it reveals more clearly what is going on. It seems more appropriate for the categorical point of view we are trying to promote. Also, using ``pullback'' instead of ``blow up'' prevents confusion with the older, highly venerated use of ``blowup'' in algebraic geometry.}, 
$Y*G*Y$, where
\begin{multline}
Y*G*Y:= Y{\,{_\Phi \ast_r}\,}G{\,{_s \ast_\Phi}\,}Y: \\= \{(x,\gamma,y)\in Y\times G\times Y\mid\Phi(x)=r(\gamma)\text{ and }\Phi(y)=s(\gamma)\}
\end{multline}
is endowed with the operations
\[
(x,\gamma,y)\cdot(y,\delta,z):=(x,\gamma\delta,z)
\]
and
\[
(x,\gamma,y)^{-1}:=(y,\gamma^{-1},x).
\]
\end{prop}
\begin{proof}
Define $\Psi:G^{Z}\to Y*G*Y$ via
\begin{equation}
\Psi([(x,\alpha),(y,\beta)])=(x,\alpha\beta^{-1},y).\label{eq:isomorphism}
\end{equation}
To see that $\Psi$ is well defined note that $[(x,\alpha),(y,\beta)]=[(u,\gamma),(w,\delta)]$
if and only if $x=u,$ $y=w$, and there is $g\in G$ such that $\gamma=\alpha g$
and $\delta=\beta g$. Consequently
\[
\Psi([(u,\gamma),(w,\delta)])=(u,\gamma\delta^{-1},w)=(x,\alpha gg^{-1}\beta^{-1},y)=(x,\alpha\beta^{-1},y).
\]
Hence $\Psi$ is well defined. Similar reasoning shows that $\Psi$
is surjective. If\\ $[(x,\alpha),(y,\beta)],[(y,\beta),(u,\gamma)]\in G^{Z}$,
then
\[
\Psi([(x,\alpha),(y,\beta)]\cdot[(y,\beta),(u,\gamma)])=\Psi([x,\alpha],[u,\gamma])=(x,\alpha\gamma^{-1},u),
\]
while
\[
\Psi([(x,\alpha),(y,\beta)])\cdot\Psi([(y,\beta),(u,\gamma)])=(x,\alpha\beta^{-1},y)\cdot(y,\beta\gamma^{-1},u)=(x,\alpha\gamma^{-1},u).
\]
Therefore $\Psi$ is a groupoid morphism. Further, $\Psi$ is one-to-one,
because if\\ $\Psi([(x,\alpha),(y,\beta)])=\Psi([(u,\gamma),(w,\delta)])$,
for two pairs $[(x,\alpha),(y,\beta)]$ and $[(u,\gamma),(w,\delta)]$ in $G^{Z}$,
then
\[
(x,\alpha\beta^{-1},y)=(u,\gamma\delta^{-1},w),
\]
which implies that $x=u,$ $y=w$, and $\alpha=\gamma g$, where $g=\delta^{-1}\beta.$
So 
\[
\bigl((u,\gamma),(w,\delta)\bigr)\cdot g=\bigl((x,\gamma g),(y,\delta g)\bigr)=\bigl((x,\alpha),(y,\beta)\bigr),
\]
showing that $[(x,\alpha),(y,\beta)]=[(u,\gamma),(w,\delta)]$, i.e.,
$\Psi$ is one-to-one. The inverse map is given by
\[
\Psi^{-1}((x,\gamma,y))=[(x,\gamma),(y,\gamma^{-1}\gamma)],
\]
and both $\Psi$ and $\Psi^{-1}$ are continuous by the definition of
the topologies involved. 
\end{proof}

\p\label{par:indweakHaarSys} We next want to discuss how the Haar system $\lambda:=\{\lambda^{u}\}_{u\in\go}$ on $G$ may be ``induced'' to  Haar systems on $G^Z$ and $Y\ast G \ast Y$. We start by choosing an arbitrary continuous $\Phi$-system of measures
$\nu:=\{\nu_{u}\}_{u\in\go}$ on $Y$ (as per Definition \ref{def:systsmaps}). While, as the following lemma shows, we need $\nu$ to be a full $\Phi$-system in order
to obtain a Haar system on $G^Z$, we proceed with arbitrary systems. We will use proper non-full systems to define groupoid correspondences in later sections. We then promote $\nu$ to a system of measures $\fk: = \{\fk_u\}_{u\in \go}$ on $Z$ defined by the formula:
\begin{equation}
\mathfrak{k}_{u}(f):=\int_{G_{u}}\int_{Y}f(y,\gamma)d\nu_{r(\gamma)}(y)d\lambda_{u}(\gamma),\qquad u\in\go,\label{eq:measonZ}
\end{equation}
for $f\in C_{c}(Z)$. 
\begin{lem}
\label{lem:equiv_s-system}The system $\fk$
is an invariant $s$-system of measures on $Z$. It is proper if and only if 
$\nu$ is proper and it is full if and only if $\nu$ is full. \end{lem}
\begin{proof}
First note that $\fk$ is a well-defined $s$-system.
The reason is that the map defined by the inner integral $\gamma\to\int_{Y}f(y,\gamma)d\nu_{r(\gamma)}$
is clearly a $\pi_{2}$-system, where $\pi_{2}$ is the projection
of $Z=Y*G$ onto $G$. We can then compose it with the source map
on $G$ to get an $s$-system on $Z$. That composition is effected
by the outer integral in (\ref{eq:measonZ}). 
To check that it is an equivariant system, let $\eta\in G$.
Then, using the right invariance of the system $\{\lambda_{u}\}_{u\in\go}$,
one can compute as follows: 
\begin{multline*}
\int_{Z}f(z\cdot\eta)d\mathfrak{k}_{r(\eta)}(z)=\int_{Z}f((y,\gamma)\cdot\eta)d\mathfrak{k}_{r(\eta)}(y,\gamma)\\
=\int_{G_{r(\eta)}}\int_{\Phi^{-1}(r(\gamma))}f(y,\gamma\eta)d\nu_{r(\gamma)}(y)d\lambda_{r(\eta)}(\gamma)\\
=\int_{G_{r(\eta)}}\int_{\Phi^{-1}(r(\gamma))}f(y,\gamma)d\nu_{r(\gamma)}(y)d\lambda_{s(\eta)}(\gamma)=\int_{Z}f(z)d\mathfrak{k}_{s(\eta)}(z),
\end{multline*}
which shows that $\fk$ satisfies the invariance
equation in Definition \ref{def:invAct}. 

The assertion about ``properness'' is clear, and the ``fullness'' assertion is also a straightforward calculation.  So, both will be omitted. 
\end{proof}
We will use $\fk$ to build a Haar system on
the imprimitivity groupoid $G^{Z}$, but first, note that with the identification
of $G^{Z}$ with $Y*G*Y$, using the map $\Psi$ from Proposition \ref{lem:Zgisomorphicto}, the unit space of $G^{Z}$
is identified with $Y$ through the map $(x,\Phi(x),x)\mapsto x$.
We will make this identification in the sequel without further comment.
Set $\pi_{z}:=\delta_{\{z\}}\times\mathfrak{k}_{s(z)}$ for all $z\in Z$.
Then by Lemma \ref{lem:equiv_s-system} the family $\{\pi_{z}\}_{z\in Z}$ is a proper $G$-invariant 
system of measures on $Z*Z$. Lemme 1.3 of \cite{ren:jot87} then shows that $\{\pi_{z}\}_{z\in Z}$ passes to a well defined system of measures $\dot{\pi}:= \{\dot{\pi}_{zG}\}_{zG\in Z/G}$ on $Z/G\simeq Y$ given by the equation
\[
\dot{\pi}(f)(z\cdot G): = \int_{Z/G} f(x\cdot G)\, d\pi_z(x), \qquad f\in C_c(Z/G),
\]
that is manifestly also proper.
Proposition 5.2 of \cite{kmrw:ajm98} shows that $\dot{\pi}$ is a Haar system on $G^{Z}$ provided that $\nu$ is a full $\Phi$-system. We will make this assumption for the remainder of the section.

\begin{defn}\label{def:dotpi}
    We call $\dot{\pi}$ the Haar system on $G^Z$ determined by $\lambda$, $\Phi$, and $\nu$.
\end{defn}

Turning now to the pullback of $G$ by $\Phi$, $Y\ast G \ast Y$, we use the identification of $G^Z$ with 
$Y*G*Y$ in Proposition \ref{lem:Zgisomorphicto} to prove
\begin{prop}\label{prop:Haarsystemimp}
 The equation 
\begin{equation}
\beta(f)(x):=\int_{G^{\Phi(x)}}\int_{Y}f(x,\gamma,y)d\nu_{s(\gamma)}(y)d\lambda^{\Phi(x)}(\gamma),\qquad x\in Y,\label{eq:HaarSystYGY}
\end{equation}
defines a Haar system, $\beta$, on $Y*G*Y$. 
\end{prop} 
\begin{proof}
Recall the isomorphism $\Psi$ in (\ref{eq:isomorphism}) between
$G^{Z}$ and $Y*G*Y$. Then if $f\in C_{c}(Y*G*Y)$ and $x\in Y\simeq(Y*G*Y)^{(0)}$
one can check that
\[
\beta(f)(x)=\dot{\pi}(f\circ\Psi)(\Psi^{-1}(x)).
\]
Therefore, $\beta$ is a Haar system on $Y*G*Y$.
\end{proof}
An important corollary (of the proof) of Proposition \ref{prop:Haarsystemimp} is
\begin{cor}\label{cor:betadotpi}The map 
\begin{equation}
    \wt{\Psi}:C_c(Y\ast G\ast Y)\to C_c(G^Z)
\end{equation}
defined by the formula $\wt{\Psi}(f):=f\circ \Psi$, is an isomorphism of convolution algebras that extends to an isomorphism between the $C^*$-algebras,  $C^*(Y\ast G \ast Y,\beta)$ and $C^*(G^Z,\dot{\pi})$.   
\end{cor}

\p\label{par:topocorrespnd} It will be important for us that we can induce representations from $(G,\lambda)$ to $(G^Z,\beta)$ as in Section \ref{sec:induc-repr} using $\Phi$-system of measures that are not full using the theory of groupoid correspondences. We provide here the main set-up in a slightly more general setting. Assume 
that $G$ is a locally compact Hausdorff groupoid
endowed with a Haar system $\lambda$. Assume that $Z$ is a
right $G$-space with a continuous, open, and surjective anchor map $\sigma:Z\to \go$ and let $\alpha:=\{\alpha_u\}_{u\in\go}$ be a
full $\sigma$-system that is invariant under the action of $G$ on $Z$ in the sense of Definition \ref{def:invAct}. Since $\sigma(z\cdot \gamma) = s(\gamma)$, this means that $\int f(z\cdot \gamma)\,d \alpha_{r(\gamma)}(z)=\int f(z)\,d\alpha_{s(\gamma)}(z)$. Let
$G^Z:=Z{_\sigma*_\sigma}Z/G$ be the imprimitivity groupoid determined by $Z$. Then, as described before 
Definition \ref{def:dotpi}, 
$\alpha$ determines a Haar system $\beta$ on $G^Z$ via
\begin{equation}
  \label{eq:Haar_impgr}
  \int_{G^Z}f([x,y])\,d\beta^{xG}([x,y])=\int_Z f([x,y])\,d\alpha_{s(x)}(y),
\end{equation}
for all $xG\in Z/G$ (\cite[Proposition 5.2]{kmrw:ajm98}). This Haar system on
$G^Z$ is fixed for the remainder of the section.
 
 Let $\delta$ be a nonnegative function on $Z$ that is invariant under the action of $G$. This implies, of course, that the pointwise product $\delta\cdot \alpha$ is an invariant system of measures. In addition, we assume that $\delta \cdot \alpha$ is a \emph{proper} invariant system, which we will denote by $\alpha_{\delta}$. This assumption places constraints on how the support of $\delta$ is distributed. We shall write \[
\int_Z f(x)\, d\alpha_{\delta,u}(x)=\int_Z f(x)\, d(\alpha_{\delta})_u(x) = \int_Z f(x)\delta(x) \,d\alpha_u(x).
\]

\begin{lem}\label{lem:topcor}
$(Z,\alpha_\delta)$ is a topological 
correspondence in the sense of Definition \ref{def:topcorr} 
with the adjoining function $\Delta_a([x,y],z)=\delta(y)/\delta(x)$.
\end{lem}
\begin{proof}
  Conditions (i), (ii), and (iii) of Definition \ref{def:topcorr} 
  are satisfied by our assumptions. We check the
  fourth condition of the definition. Let $F\in C_c(G^Z*Z)$. For
  $u\in\go$, we have
  \begin{multline*}
    \int_{Z}\int_{G^Z}F([x,y]^{-1},z)\,d\beta^{zG}([x,y])d\alpha_{\delta,u}(z)=\int_{Z}\int_Z
    F([z,y]^{-1},z)\,d\alpha_{u}(y)\,d\alpha_{\delta,u}(z)\\
    =\int_Z\int_Z
    F([y,z],z)\delta(z)\,d\alpha_u(y)\,d\alpha_u(z)\\=\int_Z\int_ZF([z,y],[y,z]z)\delta([y,z]z)d\alpha_u(y)d\alpha_u(z)\\
    =\int_Z\int_Z F([z,y],y)\delta(y)\,d\alpha_u(y)\,d\alpha_u(z)\\
    =\int_Z\int_Z
    F([z,y],y)\delta(y)/\delta(z)\,d\alpha_u(y)\,\delta(z)d\alpha_u(z)\\
    =\int_Z\int_ZF([z,y],y)\Delta_a([z,y],y)d\alpha_u(y)\,d\alpha_{\delta,u}(z)\\
    =\int_Z\int_{G^Z}F([z,y],y)\Delta_a([z,y],y)d\beta^{zG}([z,y])\,d\alpha_{\delta,u}(z).
  \end{multline*}
  The first line follows from  \eqref{eq:Haar_impgr},
  the second line from the definition of $\alpha_\delta$, and the third
  line follows since $(Z,\alpha)$ is an equivalence between $G^Z$ and
  $G$ (see \cite[Example 3.9 and Example 3.8]{hol:JOT17_2}). The
  remaining lines follow by defining
  $\Delta_a([x,y],z)=\delta(y)/\delta(x)$, from \eqref{eq:Haar_impgr},
  and from the definition of $\alpha_\delta$. 
\end{proof}

\begin{example}
  Returning to the main case we study, assume that $\Phi:Y\to\go$ is a continuous, open, surjective map. 
  Let
  $\nu=\{\nu_u\}_{u\in\go}$ be a full $\Phi$-system of measures. Let
  $Z=Y*G$ and $\alpha=\{\alpha_u\}_{u\in\go}$ be the invariant full $s$-system
  defined via
  \[
    \int_Z f(x,\gamma)\,d\alpha_u(x,\gamma)=\int_G\int_Y
    f(x,\gamma)\,d\nu_{r(\gamma)}(x)\,d\lambda_u(\gamma). 
  \]
  That is, $\alpha_u=\fk_u$ of Definition \ref{eq:measonZ}.
  Let $\bbd:Y\to\bbR^+$ be a continuous function. Then
  $\delta:Z\to\bbR^+$ defined via $\delta(x,\gamma)=\bbd(x)$ satisfies
  the hypothesis that
  $\delta((x,\gamma)\eta)=\delta(x,\gamma)$. Lemma \ref{lem:topcor}
  implies that $(Z,\alpha_\delta)$ is a topological correspondence,
  where
  \begin{equation}
  \label{eq:alpha_delta}
    \int_Z f(x,\gamma)\,d\alpha_{\delta,u}(x,\gamma)=\int_G \int_Y
    f(x,\gamma)\bbd(x)\,d\nu_{r(\gamma)}(x)\,d\lambda_u(\gamma) 
  \end{equation}
  and the adjoining function is defined via $\Delta_a((x,\gamma,y),(y,\eta))=\bbd(y)/\bbd(x)$.
\end{example}

\section{Inducing representations to the pullback groupoid}
\label{sec:induc-repr}

\p Continuing with the notation from Section \ref{sec:Haar-Systems}, observe that the title of the current section may appear confusing because the groupoid $G$ is not a subgroupoid of the pullback groupoid, $Y\ast G \ast Y$. Rather, it is a \emph{quotient} of $Y\ast G \ast Y$. However, as we have shown in Proposition \ref{lem:Zgisomorphicto}, $Y\ast G \ast Y$ is isomorphic to $G^Z$, which is the imprimitivity groupoid of the right $G$-space $Z:=Y\ast G$. 
So, it is natural to try to use the theory developed by Rieffel in \cite{rie:aim74} coupled with the equivalence theorem of \cite[Theorem 2.8]{mrw:jot87} to induce representations of $G$ to $G^Z$ and relate them to representations of $Y\ast G \ast Y$ using the isomorphism $\Psi$. 
One of the difficulties we face when trying to do this is the necessity in this paper to use non-full $\Phi$-systems of measures on $Y$. This leads to groupoid correspondences from $Y*G*Y$ to $G$ and we can 
use \cite{hol:JOT17_2} to induce representations  from $C^*(G,\lambda)$ to $C^*(Y*G*Y,\beta)$ and
\cite{ren:SIGMA14} to induce representations from $(G,\lambda)$ to $(Y*G*Y,\beta)$.
Our analysis leads to two different Hilbert space representations of the convolution algebra $C_c(Y\ast G \ast Y,\beta)$ that one may induce from a Hilbert space representation of the convolution algebra $C_c(G,\lambda)$. 
These are denoted $\Ind_{\text{J}}\fL$ and $\Ind_{\text{M}}\fL$ where $\fL$ is a prescribed Hilbert space representation of $C_c(G,\lambda)$ that comes from a representation $\hat{L}$ of the groupoid $G$ in the isomorphism groupoid of a Hilbert bundle built over $\go$.   
Our purpose in this section is to define $\Ind_{\text{J}}\fL$ and $\Ind_{\text{M}}\fL$ and to show that they are unitarily equivalent. (See Theorem \ref{thm:equivIndRep}.)

\p We begin at the level of groupoids with a unitary representation,
$\hat{L}=(\mu,\go*\cH,\hat{L})$, of $(G,\lambda)$ in the unitary
groupoid $\Iso(\go \ast \cH)$ and form its integrated form $\fL$,
representing the convolution algebra, $C_c(G,\lambda)$, on $L^2(\go
\ast \cH,\mu)$ (see \ref{subp:RepsofG}). These ingredients will be
fixed throughout this section, as will the modular function
$\Delta_{\mu}$ determined by $\mu$ and the  Haar system $\lambda$
on $G$. We also fix a full $\Phi$-system $\nu$ on $Y$ that determines a full invariant $s$-system 
$\alpha$ on $Z=Y*G$.
The Haar system $\beta$ on $Y*G*Y$ is fixed and given as in \eqref{eq:HaarSystYGY}.

We let   $\bbd:Y\to \bbR^+$ be a continuous function such that  $\bbd|_{\Phi^{-1}(u)}\ne 0$ for all $u\in \go$. Then $\bbd(x)$ defines a proper $\Phi$-system of measures $\nu_{\bbd}$ 
via
\begin{equation}
    \label{eq:nu_d}
    \int_Y f(x)\,d\nu_{\bbd,u}(x)=\int_Y f(x)\bbd(x)\,d\nu_u(x)
\end{equation}
for all $u\in \go$. Let $\delta(x,\gamma):=\bbd(x)$ and let $\alpha_\delta$ be the proper invariant $\sigma$-system defined as in \eqref{eq:alpha_delta}. 
Lemma \ref{lem:topcor}  implies that $(Z,\alpha_\delta)$ is a topological correspondece. 
We describe next  the process of inducing the groupoid representation $\hat{L}=(\mu,\go*\cH,\hat{L})$, of $(G,\lambda)$  to a groupoid representation of
$(Y*G*Y,\beta)$ via $(Z,\alpha_\delta)$ following Renault's ideas presented in \cite{ren:SIGMA14}. 
The result will be a triple,   
$\Ind \hat{L}:=(\mu_{\Ind},Y*\cK,\Ind \hat{L})$, which we call the \emph{(unitary) induced representation of $(Y\ast G \ast Y, \beta)$}. The measure,
$\mu_{\Ind}$, is defined on $Y$, the unit space of $Y\ast G \ast Y$, through the formula
\begin{equation}
    \label{eq:ind_meas}
  \int_Y f(x)\,d\mu_{\Ind}(x):=\int_{\go}\int_Y f(x)\,d\nu_{\bbd,u}(x)d\mu(u),
\end{equation}
for all $f\in C_c(Y)$. 
The following computation, \eqref{eq:muindqi}, shows that $\mu_{\Ind}$ is indeed quasi-invariant on $Y$ when $Y*G*Y$ endowed with the
 Haar system $\beta$ defined in \eqref{eq:HaarSystYGY}. The computation also shows that the modular function $\Delta_{\mu_{\Ind}}$
is given by the equation $\Delta_{\mu_{\Ind}}(x,\gamma,y)=\Delta_\mu(\gamma)\bbd(x)/\bbd(y)$ for all
$(x,\gamma,y)\in Y*G*Y$. The computation, itself is justified using Fubini's theorem and the quasi-invariance of $\mu$ on $\go$. 
\begin{multline}\label{eq:muindqi}
 \mu_{\Ind}\circ \beta(f)=\int_{Y}\int_{Y*G*Y}f(x,\gamma,y)\,d\beta^x(x,\gamma,y)\,d\mu_{\Ind}(x)\\
 =\int_{\go}\int_Y \int_{G} \int_Y f(x,\gamma,y)\,d\nu_{s(\gamma)}(y)d\lambda^u(\gamma)d\nu_{\bbd,u}(x)d\mu(u)\\
 =\int_{\go}\int_{G}\left(\int_Y\int_Y f(x,\gamma,y)\bbd(x)d\nu_{s(\gamma)}(y)d\nu_{r(\gamma)}(x)\right)\,d\lambda^u(\gamma)d\mu(u)\\
 =\int_{\go}\int_{G}\left( \int_Y\int_Yf(x,\gamma,y)\bbd(x)d\nu_{s(\gamma)}(y)d\nu_{r(\gamma)}(x)\right)\Delta_\mu(\gamma)\,d\lambda_u(\gamma)d\mu(u)\\
 =\int_{\go}\int_Y\int_{G}\int_Y f(x,\gamma,y)\Delta_\mu(\gamma)\bbd(x)/\bbd(y)\,d\nu_{r(\gamma)}(x)d\lambda_u(\gamma)\bbd(y)d\nu_u(y)d\mu(u)\\
 =\int_{\go}\int_Y\int_{G}\int_Y f(x,\gamma,y)\Delta_\mu(\gamma)\bbd(x)/\bbd(y)\,d\nu_{r(\gamma)}(x)d\lambda_u(\gamma)d\nu_{\bbd,u}(y)d\mu(u)\\
 =\int_Y\int_{Y*G*Y}f(x,\gamma,y)\Delta_\mu(\gamma)\bbd(x)/\bbd(y)\,d\beta_y(x,\gamma,y)\,d\mu_{\Ind}(y)\\
 =\Delta_{\mu_{\Ind}} (\mu_{\Ind}\circ \beta)^{-1}(f),
\end{multline}
for all $f\in C_c(Y\ast G \ast Y)$.
The Hilbert space bundle $Y*\cK$ 
is the pull-back Hilbert space
 bundle $Y*\cH: =\{(x,h)\mid x\in Y\text{ and }h\in \cH(\Phi(x))\}$ over $Y$, 
and the induced action of the pullback groupoid $Y*G*Y$ on $Y*\cH$ is given by
\begin{equation}
  \label{eq:ind_L}
(\Ind \hat{L})_{(x,\gamma,y)}(y,h): =(x,\hat{L}_\gamma h)  \quad 
(x,\gamma,y)\in Y*G*Y,\,\,
h\in \cH(\Phi(y)).
\end{equation}


The integrated form  of $\Ind \hat{L}$ is denoted $\Ind_{\text{J}} \fL$ and is defined on  $L^2(Y*\cH,\mu_{\Ind})$ via the equation
\begin{multline}\label{eq:integratedform}
\left( (\Ind_\text{J} \fL)(f) \fh\,,\,\fk \right)=\\
\int_Y\int_{Y*G*Y}f(x,\gamma,y)\bigl(
(\Ind L)_{(x,\gamma,y)}\fh(\Phi(y))\,,\,\fk(\Phi(x))
\bigr)\Delta_{\mu_{\Ind}}(x,\gamma,y)^{-1/2}
                                       \,d\beta^x(x,\gamma,y)\\ \cdots d\mu_{\Ind}(x)\\
=\int_{\go}\int_Y\int_G\int_Y
    f(x,\gamma,y)\bigl(L_\gamma(\fh(\Phi(y)))\,,\fk(\Phi(x))
    \bigr)\Delta_\mu(\gamma)^{-1/2}\bbd(y)^{1/2}/\bbd(x)^{1/2}\\ \cdots
  d\nu_{s(\gamma)}(y)d\lambda^u(\gamma)d\nu_{\bbd,u}(x)d\mu(u)\\
  =\int_{\go}\int_Y\int_G\int_Y
    f(x,\gamma,y)\bigl(L_\gamma(\fh(\Phi(y)))\,,\fk(\Phi(x))
    \bigr)\Delta_\mu(\gamma)^{-1/2}\bbd(x)^{1/2}\bbd(y)^{1/2}\\ \cdots
  d\nu_{s(\gamma)}(y)d\lambda^u(\gamma)d\nu_{u}(x)d\mu(u),
\end{multline}
for all $f\in C_c(Y*G*Y)$, $\fh,\fk\in L^2(Y*\cH,\mu_{\Ind})$.

We shall abbreviate $L^2(Y*\cH,\mu_{\Ind})$ with $H_{\Ind_{\text{J}}\fL}$.

\p\label{par:Rieffel_pers}  We turn now to a description of Rieffel's perspective on induced representations.
Recall from Lemma \ref{lem:topcor} that $(Z,\alpha_\delta)$, where $\delta(x,\gamma)=\bbd(x)$, is topological correspondence from $(Y*G*Y,\beta)$ to $(G,\lambda)$.
Theorem 2.10 of \cite{hol:JOT17_2} (see \ref{def:topcorr}) implies that $C_c(Z)=C_c(Y*G)$ can be completed to a $C^*$-correspondence from $C^*(Y*G*Y,
\beta)$ to $C^*(G,\lambda)$ via the formulas listed below.
The right action of $C_c(G,\lambda)$ on $C_c(Z)$ is given by the integral
\begin{equation}\label{eq:rightZBact}
    (\xi\cdot b)(x,\gamma)=\int_G \xi(x,\gamma\eta)b(\eta^{-1})\,d\lambda^{s(\gamma)}(\eta), \qquad \xi \in C_c(Z),\, b\in C_c(G,\lambda).
\end{equation}
The inner product on $C_c(Z)$ with values in $C_c(G,\lambda)$ is defined by the integral 
\begin{equation}\label{eq:CcGinnerprod}
  \langle \xi,\eta
  \rangle_{C_c(G)}(\gamma)  =\int_G\int_{Y} \overline{\xi(x,\zeta)}\eta(x,\zeta\gamma)\,d\nu_{\bbd,r(\zeta)}(x)d\lambda_{r(\gamma)}(\zeta),  
\end{equation}
and the left action of $C_c(Y\ast G \ast Y, \beta)$ on $C_c(Z)$ is given by the integral
\begin{equation}\label{eq:leftYGYZact}
   (a\cdot
  \xi)(x,\gamma)=\int_G\int_{Y}a(x,\zeta,y)\xi(y,\zeta^{-1}\gamma)\bbd(y)^{1/2}/\bbd(x)^{1/2}\,d\nu_{s(\zeta)}(y)d\lambda^{r(\gamma)}(\zeta),  
\end{equation}
where $ a \in C_c(Y\ast G \ast Y,\beta)$ and $\xi \in C_c(Z)$.

So we may define the representation of $C_c(Y\ast G \ast Y, \beta)$ that is induced in the sense of Rieffel from the representation $\fL$ which, recall, is the integrated form of the representation $\hat{L}=(\mu,\go\ast \cH,\hat{L})$ of the groupoid $G$. 
The Hilbert space for $\fL$ is $H_{\fL}:=L^2(\go\ast \cH,\mu)$.
The Hilbert
space $H_{\Ind_{M}\fL}$ of Rieffel's  induced
representation, $\Ind_{M}\fL$,  is the completion of the algebraic tensor
product   $C_c(Z)\odot 
H_\fL$ under the inner product
\begin{align*}
  \bigl( \xi\otimes \fh\,,\,\eta\otimes  \fk \bigr)&=\bigl( \fL(\langle
 \eta,\xi\rangle_{C^*(G)})\fh\,,\,\fk\bigr)
\end{align*}
for $\xi,\eta\in C_c(Z)$ and $\fh,\fk\in H_{\fL}$.  The induced
representation $\Ind_{\text{M}}\fL$ acts on $H_{\Ind_{\text{M}}\fL}$ via the formula
\begin{equation}
  \label{eq:ind_rep}
  (\Ind_{\text{M}} \fL)(f)\xi\otimes \fh: =(f\cdot \xi)\otimes \fh,
\end{equation}
for all $f\in C_c(Y*G*Y)$, $\xi\in C_c(Z)$, and $\fh\in H_{\fL}$.

The following theorem is anticipated in the discussion at the bottom of the first page of \cite{ren:SIGMA14}. We need a precise statement and a detailed proof.

\begin{thm}\label{thm:equivIndRep}{(cf. \cite[Lemma 2.3]{ren:jot91})} 
  Define $\fU:C_c(Z)\odot H_\fL\to H_{\Ind_{\text{J}}\fL}$ via
  \begin{equation}
    \label{eq:equiv_inducedrep}
    \fU(\xi\otimes \fh)(x)=\int_{G^{\Phi(x)}}\xi(x,\gamma)L_\gamma \fh(s(\gamma))\Delta_\mu(\gamma)^{-1/2}d\lambda^{\Phi(x)}(\gamma).
  \end{equation}
  Then $\fU$ extends to a Hilbert space isomorphism from $H_{\Ind_{\text{M}}\fL}$ to 
  $H_{\Ind_{\text{J}}\fL}$ that intertwines $\Ind_\text{M} \fL$ and $\Ind_{\text{J}} \fL$.
\end{thm}
\begin{proof}
We prove first the $\fU$ is an isometry. 
Let $\xi,\eta\in C_c(Z)$ and $\fh,\fk\in
H_\fL=L^2(\go*\cH,\mu)$. 
Then
\footnote{When a formula is broken over two or more lines because it won't fit on one line, the subsequent parts are preceded by ellipses ( $\dots$). It is intended that the subsequent parts are direct continuations of the preceding parts and nothing else is involved in the formulas.}
\begin{multline*}
  \bigl( \xi\otimes \fh\,,\,\eta\otimes \fk \bigr)=\\
 =\int_{\go}\int_G \langle \eta,\xi \rangle_{C_c(G)}(\gamma)\bigl(
    L_\gamma \fh(s(\gamma))\,,\,\fk(r(\gamma))
    \bigr)\Delta_\mu^{-1/2}(\gamma)\,d\lambda^{u}(\gamma)d\mu(u)\\
  =\int_{\go}\int_G\int_G\int_Y
    \overline{\eta(x,\delta)}\xi(x,\delta\gamma)\bigl( L_\gamma
    \fh(s(\gamma))\,,\,\fk(r(\gamma)) \bigr)\Delta_\mu(\gamma)^{-1/2}\bbd(x)\\
  \cdots d\nu_{r(
    \delta)}(x)d\lambda_u(\delta)d\lambda^u(\gamma)d\mu(u).
\end{multline*}
By Fubini's theorem, the equation may be continued with
\begin{multline*}
  \int_{\go}\int_G\int_G\int_Y
    \overline{\eta(x,\delta)}\xi(x,\delta\gamma)\bigl( L_\gamma
    \fh(s(\gamma))\,,\,\fk(r(\gamma)) \bigr)\Delta_\mu(\gamma)^{-1/2}\bbd(x)\\
  \cdots d\nu_{r(
    \delta)}(x)d\lambda^u(\gamma)d\lambda_u(\delta)d\mu(u).
\end{multline*}
By the quasi-invariance of $\mu$ and homomorphic property of $\Delta_{\mu}$ applied to the function
defined by the two inner integrals, the equation may be continued with
\begin{multline*}
  \int_{\go}\int_G\int_G\int_Y
  \overline{\eta(x,\delta)}\xi(x,\delta\gamma)\bigl(
  L_\gamma \fh(s(\gamma))\,,\,\fk(r(\gamma))
  \bigr)\Delta_\mu(\gamma)^{-1/2}\Delta_\mu(\delta)^{-1}\bbd(x) \\ 
  \cdots d\nu_u(x)d\lambda^{s(\delta)}(\gamma)d\lambda^u(\delta)d\mu(u),
\end{multline*}
which by the invariance of the Haar system, and the substitution $\delta\gamma\mapsto
\gamma$, yields
\begin{multline*}
  \int_{\go}\int_G\int_Y\int_G
  \overline{\eta(x,\delta)}\xi(x,\gamma)\bigl(
  L_{\delta^{-1}\gamma}\fh(s(\gamma))\,,\,\fk(r(\gamma))
  \bigr)\Delta_\mu(\delta^{-1}\gamma)^{-1/2}\Delta_\mu(\delta)^{-1}\bbd(x)\\
  \cdots d\lambda^u(\gamma)d\nu_u(x)d\lambda^u(\delta)d\mu(u)\\
  =\int_{\go}\int_Y\int_G\int_G \overline{\eta(x,\delta)}\xi(x,\gamma)\bigl(
  L_{\delta^{-1}\gamma}\fh(s(\gamma))\,,\,\fk(r(\gamma))
  \bigr)\Delta_\mu(\gamma)^{-1/2}\Delta_\mu(\delta)^{-1/2}\bbd(x)\\
  \cdots d\lambda^u(\gamma)d\lambda^u(\delta)d\nu_u(x)d\mu(u).
\end{multline*}
\begin{multline}\label{eq:innerpr_Rieffel}
  =\int_{\go}\int_Y \left(\int_G \xi(x,\gamma) L_\gamma
    \fh(s(\gamma))\Delta_\mu(\gamma)^{-1/2}\,d\lambda^u(\gamma)\,,\right.\\
    \cdots \left.\,\int_G
    \eta(x,\delta)L_\delta
    \fk(s(\delta))\Delta_\mu(\delta)^{-1/2}\,d\lambda^u(\delta)
  \right)
   \bbd(x)d\nu_u(x)d\mu(u)\\
   =\int_Y \bigl(\fU(\xi\otimes\fh)(x)\,,\,\fU(\eta\otimes\fk)(x)\bigr)\,d\mu_{\Ind}(x)=\bigl( \fU(\xi\otimes \fh)\,,\,\fU(\eta\otimes\fk)\bigr).
\end{multline}
Thus $(\fU(\xi\otimes
  \fh)\,,\,\fU(\eta\otimes \fk))=(\xi\otimes \fh\,,\,\eta\otimes \fk)$ for all
  $\xi,\eta\in C_c(Y*G)$ and $\fh,\fk \in H_{\cL}$. Hence $\fU$ is an
  isometry. Arguing as  in the proof of  \cite[Lemma 2.3]{ren:jot91}
  one can prove that  $\fU$ has dense range. Thus $\fU$ is a unitary.

  Let $f\in C(Y*G*Y)$, $\xi,\eta\in C_c(Z)$, and $\fh,\fk\in H_\fL$. Then
\begin{multline*}
  \bigl( (\Ind_{\text{M}} \fL)(f)\xi\otimes \fh\,,\,\eta\otimes \fk \bigr)=\bigl( (f\cdot\xi)\otimes\fh\,,\,\eta\otimes\fk \bigr)\\
   =\int_{\go}\int_Y\int_G\int_G \overline{\eta(x,\delta)}(f\cdot \xi)(x,\gamma)\bigl(
  L_{\delta^{-1}\gamma}\fh(s(\gamma))\,,\,\fk(r(\gamma))
  \bigr)\bbd(x)\Delta_\mu(\gamma)^{-1/2}\Delta_\mu(\delta)^{-1/2}\\
  \cdots d\lambda^u(\gamma)d\lambda^u(\delta)d\nu_u(x)d\mu(u)\\
 =\int_{\go}\int_Y
 \int_G\int_G\overline{\eta(x,\delta)}\left(  \int_G\int_Y f(x,\zeta,y)\xi(y,\zeta^{-1}\gamma)\bbd(y)^{1/2}
 /\bbd(x)^{1/2}d\nu_{s(\zeta)}(y)d\lambda^u(\zeta)\right)\\
 \cdots \bigl(
 L_\gamma \fh(s(\gamma))\,,\,L_\delta \fk(s(\delta)) \bigr)\bbd(x)\Delta_\mu(\gamma)^{-1/2}\Delta_\mu(\delta)^{-1/2}d\lambda^u(\gamma)d\lambda^u(\delta)d\nu_u(x)d\mu(u)
\end{multline*}
which, by Fubini's Theorem, equals
\begin{multline*}
  \int_{\go}\int_Y\int_G\int_X f(x,\zeta,y)\left(
    \int_G\int_G\xi(y,\zeta^{-1}\gamma)\overline{\eta(x,\delta)}
    \bigl( L_\gamma \fh(s(\gamma))\,,\,L_\delta \fk(s(\delta))
    \bigr)\right.\\
  \cdots \left.\bbd(y)^{1/2}\bbd(x)^{1/2} \Delta_\mu(\gamma)^{-1/2}
    \Delta_\mu(\delta)^{-1/2}d\lambda^u(\gamma)d\lambda^u(\delta)
  \right) d\nu_{s(\zeta)}(y)d\lambda^u(\zeta)d\nu_u(x)d\mu(u)
\end{multline*}
which, by the invariance of the Haar system $\zeta^{-1}\gamma\mapsto
\gamma$ applied to the most inner integral, equals
\begin{multline*}
  \int_{\go}\int_Y\int_G\int_Y f(x,\zeta,y)\\
  \cdots \left(
    \int_G\int_G\xi(y,\gamma)\overline{\eta(x,\delta)}\bigl( L_\zeta
    L_\gamma \fh(s(\gamma))\,,\,L_\delta \fk(s(\delta))\bigr)
 \Delta_\mu(\zeta)^{-1/2}\Delta_\mu(\gamma)^{-1/2}\Delta_\mu(\delta)^{-1/2}d\lambda^{s(\zeta)}(\gamma)d\lambda^u(\delta)\right)\\
 \cdots \bbd(y)^{1/2}\bbd(x)^{1/2}d\nu_{s(\zeta)}(y)d\lambda^u(\zeta)d\nu_u(x)d\mu(u)\\
 =\int_{\go}\int_Y\int_G\int_Y f(x,\zeta,y)\\
 \cdots \left(
    \int_G\int_G\xi(y,\gamma)\overline{\eta(x,\delta)}\bigl( L_\zeta
    L_\gamma \fh(s(\gamma))\,,\,L_\delta \fk(s(\delta))\bigr)
 \Delta_\mu(\gamma)^{-1/2}\Delta_\mu(\delta)^{-1/2}d\lambda^{s(\zeta)}(\gamma)d\lambda^u(\delta)\right)\\
 \cdots \bbd(y)^{1/2}\bbd(x)^{1/2}\Delta_\mu(\zeta)^{-1/2}d\nu_{s(\zeta)}(y)d\lambda^u(\zeta)d\nu_u(x)d\mu(u). 
\end{multline*}
\begin{multline}\label{eq:induced_Rieffel}
 =\int_{\go}\int_Y\int_G\int_Y f(x,\zeta,y)\\
 \cdots \left( L_\zeta\bigl(
   \int_G \xi(y,\gamma)L_\gamma
   \fh(s(\gamma))\Delta_\mu(\gamma)^{-1/2}d\lambda^{s(\zeta)}(\gamma)
   \bigr)\,,
 \,\int_G \eta(x,\delta)L_\delta
   \fk(s(\delta))\Delta_\mu(\delta)^{-1/2}d\lambda^u(\delta)\right)\\
   \cdots \Delta_\mu(\zeta)^{-1/2}\bbd(y)^{1/2}/\bbd(x)^{1/2}d\nu_{s(\zeta)}(y)d\lambda^u(\zeta)d\nu_{\bbd,u}(x)d\mu(u).
\end{multline}
Hence
  \[
    \bigl( (\Ind_{\text{M}}\fL) (f)\xi\otimes \fh\,,\,\eta\otimes \fk \bigr)=\bigl(
    (\Ind_\text{J} \fL)(f)\fU(\xi\otimes \fh)\, ,\,\fU(\eta\otimes \fk) \bigr)
  \]
  for all $f\in C_c(X*G*X), \xi,\eta\in C_c(Z),\fh,\fk\in H_{\fL}$.  Therefore $\Ind_\text{M}\fL$ 
  and $\Ind_\text{J} \fL$ are unitarily equivalent.
\end{proof}

\section*{{\Large{Part II}}\\From Deaconu-Renault Groupoids to Proto-Multiresolution Analyses}

\section{The Deaconu-Renault Groupoid} \label{sect:DR-gpd}

\p The principal player in this paper is the \emph{Deaconu-Renault groupoid} of a local homeomorphism.

\begin{defn}\label{def:D-Rgpd}Let $X$ be a second countable, compact Hausdorff space and let $\sigma:X\to X$ be a surjective local homeomorphism. The \emph{Deaconu-Renault Groupoid} defined by $X$ and $\sigma$ is    $$G(X,\sigma):= \{(x,k-l,y)\in X\times \bbZ \times X\mid k,l\in \bbZ_+,\,\sigma^k(x)=\sigma^l(y)\}.$$
The product $(x,k-l,y)(w,m-n,z)$ is defined only when $y=w$ and then 
$$(x,k-l,y)(y,m-n,z):=(x,(k+m)-(l+n),z).$$
The inverse of $(x,k-l,y)$ is $(y,l-k,x)$.
\end{defn}

The unit space of $G(X,\sigma)$ is identified with $X$ and the maps $r:G(X,\sigma)\to X$ and $s:G(X,\sigma)\to X$ given by $r(x,k-l,y):= x$ and $s(x,k-l,y):=y$ are the \emph{range} and \emph{source} maps. Because $\sigma$ is assumed to be a surjective local homeomorphism it is easy to see that $G(X,\sigma)$ viewed as a subset of $X\times \bbZ \times X$ with the relative topology is locally compact and Hausdorff and that the maps $r$ and $s$ are local homeomorphisms. Consequently, $G(X,\sigma)$ is an example of what nowadays is  universally known as an \emph{\'{e}tale groupoid}\footnote{The terminology used in \cite{Ren_LNM793} for such groupoids is ``r-discrete''.}.

\subp It will be useful at times to ``change variables'' and rewrite the defining formulas for $G(X,\sigma)$ as Valentin Deaconu did in \cite{Dea_TAMS95}:
\begin{equation}\label{def:D-RgpdAlt}
    G(X,\sigma)=\{(x,n,y)\in X\times \bbZ \times X \mid \exists\, k,l\geq 0,\, n=l-k,\sigma^k(x)=\sigma ^l(y)\}.
\end{equation}
Two triples, $(x,n,y)$ and $(w,m,z)$, can be multiplied iff $y = w$, in which case $(x,n,y)(y,m,z)=(x,m+n,z)$; $(x,n,y)^{-1}=(y,-n,x)$.

This version makes many formulas easier to read than the original, but the original becomes preferable when one wants to ``get under the hood'' to parse formulas.

\subp While at first glance, the definitions of $G(X,\sigma)$ may seem contrived, close inspection reveals the natural connection of $G(X,\sigma)$ with the algebraic analysis of Murray and von Neumann devoted to building von Neumann algebras from transformation groups. One should keep in mind, also, that the dynamics determined by $\sigma$ involve not only the forward orbits of points, $\{\sigma^k(x)\}_{k=0}^{\infty}$, but also their backward orbits, $\{\sigma^{-k}(x)\}_{k=0}^{\infty}$.  If $\sigma$ is not invertible, then the backward orbits involve sets of cardinality different from $1$. The groupoid $G(X,\sigma)$ provides a useful way to account for this feature of $\sigma$.

Deaconu \cite{Dea_TAMS95} received inspiration from Renault's thesis \cite{Ren_LNM793} in which Renault built groupoids naturally connected to the representations of the Cuntz algebras.  In that setting, $X$ is a certain one-sided shift space and $\sigma$ is the shift. In \cite[Proposition 2.8]{Renault-Cuntz-like}, Renault identifies when $G(X,\sigma)$ is isomorphic to the groupoid of germs of maps associated to the set-theoretic powers $\{\sigma^k\circ \sigma^{-l}\}$ where $k$ and $l$ run over all non-negative integers, i.e., $\sigma^k\circ \sigma^{-l}$  is viewed as a map on the power set of $X$. We shall meet these again shortly.

\p \label{p:defDRgpdalg} The fact that $G(X,\sigma)$ is \'{e}tale implies that the family of counting measures on the fibres of $r$, taken together, constitute a Haar system, $\lambda = \{\lambda^x\}_{x\in X}$, on $G(X,\sigma)$\footnote{We emphasize that $\lambda$ is full and that any other Haar system on $G(X,\sigma)$ is proportional to $\lambda$, provided there are no nontrivial $\sigma$-invariant subsets of $X$. If there are such subsets, then the ``constant of proportionality'' may vary across them.}. With this Haar system, the formula for convolution becomes
\begin{multline}\label{eq:mult}
    f\ast g(x,k-l,y):= \int_{G(X,\sigma)}f(\alpha)g(\alpha^{-1}\gamma)\,d\lambda^{r(\gamma)}(\alpha)\\=\sum f(x,m-n,z)\cdot g(z,(n+k)-(m+l),y),
\end{multline} 
where $\alpha:=(z,n-m,w)$, $\gamma: = (x,k-l,y)$, and where
the sum ranges over all $m,n$ and $z$ such that $\sigma^m (x) = \sigma^n (z)$. The adjoint of a function $f$ is defined by
\begin{equation}\label{eq:adj}
    f^*(x,k-l,y):=\overline{f(y,l-k,x)}.
\end{equation}

\subp\label{subp:ZUV} The unwieldy formula \eqref{eq:mult} by itself makes it difficult to grasp the structure of the convolution algebra $C_c(G(X,\sigma),\lambda)$. However, when applied to special functions that are evident generators of $C_c(G(X,\sigma),\lambda)$ the formula becomes easier to manage and reveals what is going on. First, let $U$ and $V$ be open subsets of $X$, suppose $k$ and $l$ are such that the restrictions, $\sigma^l|U$ and $\sigma^k|V$, are homeomorphisms with common range, i.e., $
\sigma^k(U)=\sigma^l(V)$, and let
\begin{equation}
    Z(U,V,k,l):=\{(x,k-l,y)\in G(X,\sigma)\mid x\in U, y\in V\}.
\end{equation}
Then $Z(U,V,k,l)$ is essentially the graph of $(\sigma^l|_V)^{-1}\circ (\sigma^k|_U)$ and is a \emph{$G$-set} in the sense of Renault \cite[Definition 1.10]{Ren_LNM793}. It is also an open subset of $G(X,\sigma)$. Under an evident notion of composition discussed in \cite[Page 10]{Ren_LNM793}, these open $G$-sets form an inverse semigroup which also is a basis for the topology of $G(X,\sigma)$. (The germs of such $G$-sets, mentioned above, are explored in \cite{Renault-Cuntz-like}.)
If a $G$-set $Z:=Z(U,V,k,l)$ happens also to be  closed then its characteristic function $1_Z$ is an element of $C_c(G(X,\sigma),\lambda)$ which, when properly scaled, yields a partial isometry in $C_c(G(X,\sigma),\lambda)$. 

\subp\label{par:Rnm} Another class of subsets of $G(X,\sigma)$ that plays an important role in our analysis is defined by the equations
\begin{equation}\label{eq:Rnm}
R_{n,m}:=\{(x,n-m,y)\in G(X,\sigma)\mid \sigma^n(x)=\sigma^m(y)\}, \qquad n,m\geq 0.
\end{equation}
Each $R_{n,m}$ is compact and open in $G(X,\sigma)$, so its characteristic function $1_{R_{n,m}}$ lies in $C_c(G(X,\sigma),\lambda)$.
Also, $1_{R_{n,m}}^* =1_{R_{m,n}}$. Perhaps most important, $R_{n,m}$ may be viewed as the graph of the multi-valued function $\sigma^{-m}\circ \sigma^n$; if $m=0$, then obviously, $R_{n,0}$ is the graph of $\sigma^n$.

Continuing the discussion from a slightly different perspective, we find that $1_{R_{0,0}}$ is the identity in $C_c(G(X,\sigma),\lambda)$.  
Further, each $R_{n,n}$ may be viewed as an equivalence relation on $X$. 
Moreover, their collection is nested: If $n\leq m$, then $R_{n,n}\subseteq R_{m,m}$. 
We may view each $R_{n,n}$ as a subgroupoid of $G(X,\sigma)$ whose $C^*$-algebra, $C^*(R_{n,n},\lambda)\subset C^*((G(X,\sigma),\lambda)$, is the cross sectional $C^*$-algebra of a matrix bundle over $X$ and so is a $C^*$-algebra with continuous trace. 
The union $\bigcup_n R_{n,n}:=R_{\infty}$ is also a subgroupoid of $G(X,\sigma)$ whose $C^{*}$-algebra, $C^*(R_{\infty},\lambda)$, is the inductive limit, $\underset{\longrightarrow }\lim \,C^*(R_{n,n}, \lambda)$. These facts about the equivalence relations $R_{n,n}$ and their $C^*$-algebras are proved and developed further in \cite{Ren_ETDN_05}. There, the equivalence relation $R_{n,n}$ would be called a \emph{proper equivalence relation} because the quotient space $X/R_{n,n}$ is compact and Hausdorff and the quotient map $\pi:X\to X/R_{n,n}$ is proper. The equivalence relation $R_{\infty}$, then, would be called \emph{approximately proper}.

\subp\label{subp:Dzero} The set $R_{1,0} = Z(X,X,1,0) = \{(x,1,y)\mid y=\sigma(x)\}$ plays a central role in the theory because, as we have already noted, it is essentially the graph of $\sigma$. This observation leads to the straightforward calculation of the product $1_{R_{1,0}}\ast (1_{R_{1,0}})^*$ in $C_c(G(X,\sigma),\lambda)$. The result is: 
\begin{equation}\label{eq:R10}
    1_{R_{1,0}}\ast (1_{R_{1,0}})^*(x,k-l,y)= |\sigma^{-1}(\sigma(x))| 1_{R_{1,1}}(x,k-l,y),
    \end{equation}
where $\sigma^{-1}(\sigma(x)):= \{y\in X \mid \sigma(y)=\sigma(x)\}$ and $\vert \sigma^{-1}(\sigma(x)) \vert $ is its cardinality. Further reflection reveals that the function $x \to |\sigma^{-1}(\sigma(x))|$ is an integer-valued function that lies in $C(X)$ and so is finitely-valued and constant on the connected components of $X$. Further still, it is strictly positive since $\sigma$ is surjective. Its \emph{reciprocal}, denoted here by $D_0$, plays an important role in various parts of wavelet theory. 


\begin{defn}\label{def:masteriso} The function $S$ in $C_c(G(X,\sigma))$ defined by
\begin{equation}
    S(x,k-l,y):= D_0^\frac{1}{2}(x)\cdot 1_{R_{1,0}}(x,k-l,y), \qquad (x,k-l,y)\in G(X,\sigma)
\end{equation}
is called the \emph{master isometry} in $C_c(G(X,\sigma))$\footnote{This terminology comes from \cite{CMS2012} where isometries with properties similar to those of $S$ are discussed.}.
\end{defn}

The equation \eqref{eq:R10} shows that $S$ is, in fact, an isometry. Also note that equation \eqref{eq:R10} shows that $SS^*$ is the projection $1_{R_{1,1}}$ in $C_c(G(X,\sigma))$.

\p \label{p:DefPots}We call $D_0$ the \emph{fundamental potential (function)} of the canonical transfer operator for (or associated to) $\sigma$. To understand where this terminology comes from and to see its importance in our narrative, consider a positive linear map $\cL$ from $C(X)$ to $C(X)$. Then $\cL$ is given  by a unique continuous family of positive Borel measures on $X$, say $\{\alpha_x\}_{x\in X}$, via the formula $$\cL(f)(x):=\int f(y)\,d\alpha_x(y).$$ Consider, also, the map $\pi$ on $C(X)$ given by composition with $\sigma$. Then since $\sigma$ is surjective, $\pi$ is a \emph{unital, injective, $C^*$-algebra endomorphism} of $C(X)$. 

Broadly speaking, potential theory is the study of how these two \emph{positive linear} maps on $C(X)$ interact. We require several features of this interaction. First, observe that we may consider $C(X)$ as a module over itself in two different ways. On the one hand there is the \emph{regular module}, $C(X)_{C(X)}$, in which $C(X)$ acts on itself via the usual pointwise multiplication. Denote this multiplication simply by a \verb|\cdot|, $\cdot$.  On the other, there is the \emph{$\pi$-induced module}, $C(X)_{\pi(C(X))}$, whose multiplication, $\ast$, is defined through composition with $\pi$: $f\ast a:= f\cdot \pi(a)$, $f,a\in C(X)$. 
A linear map on $C(X)$, $\cL$, is called a \emph{module map} in case $\cL$ intertwines these two module structures.

The following proposition may well be known, but we believe that it may not be sufficiently well known for our purposes.  Since it is an essential part of our analysis, we supply a proof.
    
\begin{prop}\label{prop:potential}
    Every positive module map $\cL:C(X)\to C(X)$ is determined by a unique non-negative continuous function $\psi$ on $X$ through the formula
    \begin{equation}\label{eq:potential}
        \cL(a)(x) = \sum_{\sigma(y)=x}\psi(y)a(y),\qquad a\in C(X).
    \end{equation}
\end{prop}

\begin{proof}
    Let 
    $\alpha=\{\alpha_x\}_{x\in X}$ 
    be the continuous family of positive Borel measures that defines $\cL$. 
    Observe first that 
    $\supp \alpha_x\subseteq \sigma^{-1}(x)$. 
    Indeed, assume to the contrary that there are 
    $x,z\in X$ such that 
    $z\in \supp \alpha_x$ and 
    $\sigma(z)\ne x$. 
    We can then find an open set $U_z$ such that $\sigma|_{U_z}$ is a 
    homeomorphism onto 
    $\sigma(U_z)$ and 
    $x\notin \sigma(U_z)$. 
    Let $a\in C(X)$ be such that 
    $a(\sigma(z))=1$ and $a(y)=0$
    if $y\notin \sigma(U_z)$. 
    It follows that $\cL(\pi(a))(x)=\cL(a\circ \sigma )(x)>0$ since 
    $z\in \supp \alpha_x$. On the other hand,
    \[
    \cL(\pi(a))(x)=\cL(1\cdot \pi(a))(x)=(\cL(1)\cdot a)(x)=\cL(1)(x)a(x)=0.
    \]
    This is a contradiction. 
    Hence $\supp \alpha_x\subseteq \sigma^{-1}(x)$ for all $x\in X$. 

    Observe next that since $X$ is compact, $\sigma^{-1}(x)$ is a finite set for each $x\in X$. 
    Consequently, there is a non-negative function $\Delta$ on 
    $X\times X$ such that
    \begin{equation}\label{eq:cL_Delta}
    \cL(f)(x)=\int_X f(y)\,d\alpha_x(y)=\sum_{\sigma(y)=x}f(y)\Delta(x,y).
    \end{equation}
    To prove that $\Delta$ is continuous, let $U$ be an open subset of $X$ on which $\sigma$ is a homeomorphism
    and let $V$ be an open subset of $U$ such that $\overline{V}\subsetneq U$. 
    Let $f\in C(X)$ such that $f\equiv 1$
    on $\overline{V}$ and $f\equiv 0$ on the complement of $U$. 
    Then $\cL(f)\in C(X)$ and \eqref{eq:cL_Delta} yields
    \[
    \cL(f)(x)=\Delta(x,y) \,\,\text{ if }\,x\in \sigma(V)\,\text{ and }y=\sigma^{-1}(x)\in V.
    \]
    Therefore $\Delta$ is continuous on open sets of the form $\sigma(V)\times V$.
    Since the cograph of $\sigma$ is 
    covered by such sets, $\Delta$ is supported and continuous on the cograph of $\sigma$. Therefore the function
    $\psi$ defined via the equation $\psi(y):=\Delta(\sigma(y),y)$ is continuous and non-negative. 
    Moreover, \eqref{eq:cL_Delta} implies that
    \[
    \cL(f)(x)=\sum_{\sigma(y)=x}f(y)\Delta(x,y)=\sum_{\sigma(y)=x}f(y)\Delta(\sigma(y),y)=\sum_{\sigma(y)=x}\psi(y)f(y).
    \]
    The converse assertion and the uniqueness of $\psi$ are immediate.
    \end{proof}

    \begin{remark}\label{rem:transferoppotent}
    In the statistical mechanics literature (cf. \cite{bal:transfer, rue:emia78}) an operator of the form defined in \eqref{eq:potential} is called a \emph{transfer operator} (associated with the dynamical system defined by $\sigma$) and $\psi$ is referred to as its \emph{potential}. We shall adopt that terminology here. However, because of certain conventions in the groupoid theory literature, we will have occasions to use capital $D$'s as well as $\psi$'s to denote potentials. 
    
    Observe that $\cL(1)(x)=\sum_{\sigma(y)=x}\psi(y)=\alpha_x(X)$. So, $\{\alpha_x\}_{x\in X}$ is a proper system of measures on $X$ in the sense of Definition \ref{def:systsmaps} if and only if $\cL(1)$ is a strictly positive function. We shall therefore say \emph{$\cL$ is proper} when $\cL(1)(x)>0$ for all $x\in X$.  
    \end{remark}

\begin{prop}\label{prop:normalizedpotent}Let $\cL_{\psi}$ be a transfer operator with potential $\psi$ associated to the local homeomorphism $\sigma$ and let $\pi$ be the endomorphism of $C(X)$ determined by $\sigma$. Then the following assertions are equivalent.
\begin{enumerate}
    \item $\cL_{\psi}(1)=1$, i.e., $\cL_{\psi}$ is a unital map.
    \item $\cL_{\psi}$ is a left inverse for $\pi$.
    \item $\pi\circ\cL_{\psi}$ is a conditional expectation onto the range of $\pi$\footnote{If $A$ is a $C^*$-algebra and $B$ is a $C^*$-subalgebra of $A$, then a \emph{conditional expectation} from $A$ onto $B$ is a positive linear map on $A$ that restricts to the identity map on $B$.}. 
\end{enumerate}
\end{prop}

If $\cL_{\psi}$ satisfies any of these three equivalent conditions, we call $\psi$ a \emph{normalized} potential.

\begin{proof}
    If $\cL_{\psi}(1)=1$, then for all $f\in C(X)$ we have 
    \begin{multline}
(\cL_{\psi}\circ \pi)(f)(x)= \sum_{\sigma(y)=x} \psi(y)\pi(f)(y) = \sum_{\sigma(y)=x}\psi(y)f(\sigma(y))\\ = \sum_{\sigma(y) = x}\psi(y)f(x) = f(x)(\sum_{\sigma(y)=x}\psi(y)) = f(x),
    \end{multline}
so $\cL_{\psi}$ is a left inverse of $\pi$. 
Now, write $\iota$ for the identity map on $C(X)$. 
Then the assertion that $\cL_{\psi}$ is a left inverse of $\pi$ is equivalent to the equation $\cL_{\psi}\circ \pi = \iota$. 
In this event, $(\pi\circ \cL_{\psi})\circ (\pi\circ \cL_{\psi}) = \pi\circ (\cL_{\psi}\circ \pi)\circ \cL_{\psi} = \pi \circ \iota \circ \cL_{\psi}= \pi \circ \cL_{\psi}$. 
So $\pi\circ \cL_\psi$ a positive idempotent map on $C(X)$ such that $(\pi\circ \cL_\psi)\circ \pi=\pi\circ (\cL_\psi\circ \pi)=\pi$.
Thus $\pi\circ \cL_{\psi}$ is an idempotent map that acts like the identity on the range of $\pi$, i.e., $\pi\circ \cL_{\psi}$ is a conditional expectation onto the range of $\pi$. Finally, if $\pi\circ \cL_\psi$ is a conditional expectation onto the range of $\pi$, $(\pi\circ \cL_{\psi}) \circ \pi = \pi$. But then, $\pi\circ (\cL_{\psi}\circ \pi)=\pi$. Since $\pi$ is injective, $\cL_{\psi}\circ \pi = \iota$, i.e., $\cL_{\psi}$ is a left inverse of $\pi$. However, this manifestly implies $\cL_{\psi}(1)=1.$
\end{proof}

\begin{defn}
    A potential is called \emph{full} in case it is nonvanishing.
\end{defn}

The following proposition is part of Proposition 3.1 in Jean Renault's \cite{Ren_ETDN_05}.

\begin{prop}\label{defpi}
    Every full normalized potential $D$ satisfies an equation of the form 
    \begin{equation}
        D=D_0\cdot b
    \end{equation}
    where $b$ is a nonvanishing function in $\pi(C(X))$.
\end{prop}
\begin{remark}\label{rem:fullnonfull}
    We want to emphasize that non-full potentials are central to our study. These are the source for our ``secret sauce''. For our purposes every potential may be taken to be normalized: For if $\psi$ is a potential and if $$D:=\psi/\cL_{\psi}(1)\circ\sigma,$$ then $D$ is a normalized potential that vanishes precisely at the same points as $\psi$. Full normalized potentials may almost always be taken to be $D_0$. Sections  \ref{sec:pot,filt,scalfns} and \ref{sec:toHilbertspaces} are devoted to constructing non-full normalized potentials from full non-normalized potentials.  This will be how we build \emph{scaling functions} and associated \emph{mother wavelets}, i.e., \emph{quadrature mirror filters}.
\end{remark}

\subp \label{subp:AddRelRandL} We want to draw attention, also, to some additional relations that hold between $\pi$ and one of its left inverses $\cL$. First, recall that by Proposition \ref{prop:normalizedpotent}, $\pi\circ \cL$ is a conditional expectation of C(X) onto the range of $\pi$ and every conditional expectation of $C(X)$ onto the range of $\pi$ is of this form for a suitable $\cL$. So, $\pi^n\circ \cL^n:= \projE_n$ is a conditional expectation onto the range of $\pi^n$ for every non-negative integer $n$. Further, still, we have the relations
\begin{equation}
\projE_n\projE_m = \projE_m\projE_n = \projE_n,\qquad m\leq n.
\end{equation}
That is, the sequence $\{\projE_n\}_{n=0}^{\infty}$ is a decreasing sequence of projections \emph{on} the $C^*$-algebra $C(X)$. We want to understand their relationship with the representation theory of $C_c(G(X,\sigma),\lambda)$ that we develop later in this paper. 

Observe that \emph{every} normalized potential, full or not, leads to an isometry in the convolution algebra of $G(X,\sigma)$. 
In fact, one may do a bit better. If $\fu$ is a continuous function on $X$ such that $|\fu|^2$ is a normalized potential $D$, and if
\begin{equation}\label{eq:defSu}
    S_{\fu}(x,k-l,y):= \fu(x)1_{R_{1,0}}(x,k-l,y),
\end{equation}
then $S_{\fu}$ is an isometry in the convolution algebra, $C_c(G(X,\sigma),\lambda)$:

\begin{multline}\label{eq:S_fu}
        S_\fu^**S_\fu(x,k-l,y)=\sum_{(x,m-n,z)\in G}S_\fu^*(x,m-n,z)S_\fu(z,n+k-(l+m),y)\\
        =\sum_{(x,m-n,z)\in G}\overline{S_\fu(z,n-m,x)}S_\fu(z,n+k-(l+m),y)\\
        =\begin{cases}
            \sum_{\sigma(z)=x}\overline{\fu(z)}\fu(z) & \text{ if }x=y\\
            0 & \text{ otherwise }
        \end{cases}\\
        =\begin{cases}
            \sum_{\sigma(z)=x}D(z) & \text{ if }x=y\\
            0 & \text{ otherwise }
        \end{cases}
        =1_{R_{0,0}}(x,k-l,y).
    \end{multline}
Again, this computation is easy by virtue of the fact that $R_{1,0}$ is essentially the graph of $\sigma$.

From the computation, the \emph{covariance relation} 
    \begin{equation}
    \label{eq:covar-Su}
        S_\fu*f=(f\circ \sigma)*S_\fu=\pi(f)\ast S_\fu
    \end{equation}
    is revealed; it holds for all $f\in C(X)$. Multiplying this equation on the left by $S_{\fu}^*$ and focusing on the fact that composition with $\sigma$ is the endomorphism $\pi$, we find that $S_{\fu}$ implements the left inverse, $\cL_D$, of $\pi$ in $C_c(G(X,\sigma),\lambda)$: i.e.
    \begin{equation}\label{eq:LDcov}
        f = S_{\fu}^* \ast (f\circ \sigma)\ast S_{\fu}= S_{\fu}^* \ast \pi(f) \ast S_{\fu} = \cL_D\circ \pi (f), 
    \end{equation}
for $f\in C(X), |\fu|^2=D$. Moreover, a similar easy computation shows that $S_\fu^*$ implements $\cL_D$ from inside  $C_c(G(X,\sigma),\lambda)$, that is,
\begin{equation}
    \label{eq:L_D_S_u}
    \cL_D(f)=S_\fu^**f*S_\fu
\end{equation}
for all $f\in C(X)$.
And, in general, $\cL^n_D(f) = (S^*_\fu)^n\ast f \ast (S_\fu)^n$.

\p\label{par:S_u_HC}
It will be convenient to regard the isometry $S_\fu$ as an element of $\cL(\cX_0)$, where $\cX_0$ is the $C^*$-correspondence defined 
by the equivalence $(G,\alpha)$ between $(G,\lambda)$ and $(G,\lambda)$ as described in \ref{subp:Equivalence_as_corresp}. There, the operator $S_\fu\in \cL(\cX_0)$ is
defined via
\begin{equation}
    \label{eq:S_fu_HC}
    S_\fu\xi(x,t,y)=\fu(x)\xi(\sigma(x),t-1,y)
\end{equation}
It is an isometry because $D$ is a unital potential and so
\begin{multline*}
    \langle S_\fu\xi\,,\,S_\fu\eta\rangle_{C_c(G)}(x,t,y)=\sum_{(z,s,x)\in G}\overline{S_\fu\xi(z,s,x)}S_\fu\eta(z,s+t,y)\\
    =\sum_{(z,s,x)\in G}\overline{\fu(z)}\overline{\xi(\sigma(z),s-1,x)}\fu(z)\eta(\sigma(z),s+t-1,y)\\
    =\sum_{(u,s,x)\in G}\sum_{\sigma(z)=u}D(z) \overline{\xi(u,s,x)}\eta(u,s+t,y)=\langle \xi\,,\,\eta\rangle_{C_c(G)}(x,t,y),
\end{multline*}
for all $\xi,\eta\in C_c(Z)$.

\subp It may seem that we have introduced confusion by using ``$S_\fu$'' to denote both the isometry from \eqref{eq:S_fu} in $C^*(G,\lambda)$ and
the isometry from \eqref{eq:S_fu_HC} on $\cL(\cX_0)$. However, this ambiguity should cause no significant problem. The two isometries are the same, up to a unitary equivalence between the two Hilbert modules involved: Since $(G,\alpha)$ is 
a $(G,\lambda)$-$(G,\lambda)$ equivalence, $C^*(G,\lambda)$ acts on the left on $\cL(\cX_0)$ and it can be identified with the compact operators $\cK(\cX_0)$ generated 
by the rank-one operators $\Theta_{\xi,\eta}$, where $\Theta_{\xi,\eta}\zeta=\xi\cdot\langle \eta,\zeta\rangle_{C^*(G,\lambda)}$ for all $\xi,\eta,\zeta\in \cX_0$. The isomorphism 
between $C^*(G,\lambda)$ and $\cK(\cX_0)$ is given by $\Theta_{\xi,\eta}\mapsto _{C^*(G)}\langle \xi\,,\,\eta\rangle$ (see, for example,  the proof of
Proposition 3.8 of \cite{rw:morita}). The isometry $S_\fu$ defined in \eqref{eq:S_fu_HC} belongs to $\cK(\cX_0)$ since $S_\fu=\Theta_{\xi_0,\eta_0}$, where $\xi_0=1_{R_{0,0}}$ 
and 
\[
\eta_0(x,k-l,y)=\begin{cases}
\fu(x) & \text{ if }\sigma(y)=x \text{ and }k-l=-1\\
0 & \text{ otherwise}.
\end{cases}
\]
It is easy to check that $_{C^*(G,\lambda)}\langle \theta_0\,,\,\eta_0\rangle$ equals $S_\fu$ defined via \eqref{eq:S_fu}.

\section{Pullbacks by Projective Limits}\label{sec:PBsIRs}

\p In nearly every operator-theoretic approach to wavelets that we have encountered, there are isometries like our $S_\fu$'s and efforts are made to understand unitary extensions of them.  We want to do that here as well. However, we also want to emphasize that the extensions we produce do not live in $C^*(G(X,\sigma))$, but in larger $C^*$-algebras that we will reveal as certain pullbacks. Their description requires several ingredients.

\p\label{subp:pullback} The first ingredient is a pullback of $G(X,\sigma)$ in the sense of Proposition \ref{lem:Zgisomorphicto} that is determined by a pair consisting of a space $X_{\infty}$ and map $p:X_{\infty}\to X$. The space $X_{\infty}$ is the projective limit space induced by $\sigma$: \[X_{\infty}:=\{(x_0,x_1,x_2,\cdots)\in X^{\infty}\mid x_i=\sigma(x_{i+1})\}.\]
Here, of course, $X^{\infty}$ is the infinite product of copies of $X$ indexed by $\bbN$.  Since $X$ is compact and $\sigma$ is continuous and surjective, $X_{\infty}$ is a non-empty closed, and therefore compact, subset of $X^{\infty}$. We write its elements as $\und{x}:=(x_0,x_1,\cdots)$. The map $p$ is simply the projection onto the first coordinate: $p(\und{x}):=x_0$. Since $\sigma$ is open, $p$ is also open, as well as continuous, and it is surjective because $\sigma$ is surjective.

\begin{defn}\label{def:pullback}
    The \emph{pullback of $G(X,\sigma)$ by $X_{\infty}$ and $p$} 
    is 
\begin{multline}\label{eq:defGinf}G_{\infty}(X,\sigma):= X_{\infty}\ast G(X,\sigma) \ast X_{\infty}:\\ = \{(\und{x},(x_0,k-l,y_0),\und{y})\in X_{\infty}\times G(X,\sigma)\times X_{\infty} \mid p(\und{x})=x_0, p(\und{y})=y_0\}\\=\{(\und{x},k-l,\und{y})\mid (p(\und{x}),k-l,p(\und{y}))\in G(X,\sigma)\}.\end{multline}
    The product $(\und{x},k-l,\und{y})(\und{w},m-n,\und{z})$ is defined only when $\und{y}=\und{w}$ and then 
$$(\und{x},k-l,\und{y})(\und{y},m-n,\und{z}):=(\und{x},(k+m)-(l+n),\und{z}).$$
The inverse of $(\und{x},k-l,\und{y})$ is $(\und{y},l-k,\und{x})$.
\end{defn}

We shall use evident variations on the notation used to describe features of $G_{\infty}(X,\sigma)$ as our analysis progresses.

By Proposition \ref{lem:Zgisomorphicto}, $G_{\infty}(X,\sigma)$ is a locally compact Hausdorff groupoid with unit space $X_{\infty}$.  
The range and source maps are the obvious projections that are, however, not local homeomorphisms. Consequently, $G_{\infty}(X,\sigma)$ is not \'{e}tale. Further, the analogue of the formula \eqref{eq:mult} in which $x$, $y$, and $z$ are replaced by $\und{x}$, $\und{y}$, and $\und{z}$ involves uncountably infinite sums and so does not make sense. Thus, at first glance, it may seem that $G_{\infty}(X,\sigma)$ is of no relevant use.

\subp Indeed, while the groupoid has appeared in the literature (see, for example, \cite{Dea_TAMS95,Arzumanian-Renault} and their many citations), as far as we know, until now it has been overlooked entirely as a possible vessel for wavelets. For our part, when we began our research leading to \cite{im2008} and for many years thereafter, we were of the opinion that the natural and ``correct'' groupoid to study along with the Deaconu-Renault groupoid $G(X,\sigma)$ is the Deaconu-Renault groupoid $G(X_{\infty},\sigma_{\infty})$, where $\sigma_{\infty}$ is the invertible ``extension'' of $\sigma$ that is defined by the formula:
\begin{equation}
    \sigma_{\infty}(\und{x}):= (\sigma(x_0),x_0,x_1\cdots),\qquad \und{x}= (x_0,x_1\cdots)\in X_{\infty}.
\end{equation}\label{eq:siginf}
The inverse of $\sigma_{\infty}$ is given by the equation
\begin{equation}
\sigma_{\infty}^{-1}(x_0,x_1,x_2\cdots):= (x_1,x_2\cdots).
\end{equation}
The map $\sigma_{\infty}$ is not really an \emph{extension} of $\sigma$, but it does satisfy the important ``intertwining'' equation
\begin{equation}\label{eq:intertwine-p}
   \sigma\circ p= p\circ \sigma_{\infty}.
\end{equation}

The point is: the Deaconu-Renault groupoid for $\sigma_{\infty}$ acting on $X_{\infty}$, $G(X_{\infty},\sigma_{\infty})$, \emph{is} \'{e}tale. 
Moreover, it is isomorphic to the much-studied transformation group groupoid determined by $\sigma_{\infty}$. 
So we and others were drawn to it as a natural place to investigate wavelets and other fractal-like phenomena. 
Unfortunately, as we discovered from trying to extend our analysis in \cite{im2008}, it is impossible to track the relations $R_{n,m}$ (defined in \ref{par:Rnm}) \emph{from inside} the groupoid $G(X_{\infty},\sigma_{\infty})$. 
Rather, as we shall show, our analysis of $G_{\infty}(X,\sigma)$ reveals a more nuanced relation between the dynamical system $(X_{\infty},\sigma_{\infty})$  and a \emph{family} of  $C^*$-algebras that are based on  $G_{\infty}(X,\sigma)$ and the \emph{adjoining functions} of Holkar \cite[Definition 2.1]{hol:JOT17_2}.

\p \label{def:WeakHaarSystems}Using a normalized full potential $D$ on $X$ one can construct a Haar system on $G_{\infty}(X,\sigma)$, indexed by $X_{\infty}$, that is denoted $\nu^D*\lambda:= \{(\nu^D*\lambda)^{\und{x}}\}_{\und{x}\in X_{\infty}}$. The symbol $\nu^D$ refers to a system of measures $\{\nu^D_{x_0}\}_{x_0\in X}$ on $X_{\infty}$ that is parameterized by $X$ and is defined as follows. Note that the construction goes through  for non-full normalized potentials $D$. 

\subp \label{subp:defnuD}Recall that the product topology on $X^{\infty}$ is such that the collection of complex-valued  functions of the form $f=f_0\cdot f_1\cdot f_2\cdot \cdots \cdot f_n$, where the $f_i$ lie in $C(X)$ and $$f(\und{x}): = f_0(x_0)\cdot f_1(x_1)\cdot f_2(x_2)\cdots f_n(x_n),\qquad \und{x}=(x_0,x_1,x_2, \cdots)\in X^{\infty},$$ generate $C(X^{\infty})$. Consequently, their restrictions to $X_{\infty}$ generate $C(X_{\infty})$. Recall, too, that Kolmogorov's extension theorem \cite[Appendix A.3]{Rosenblatt1971} guarantees that the formula
\begin{multline}\label{eq:Markov1}
\int_{X_\infty} f_{0}(x_{0})\cdot f_{1}(x_{1})\dots \cdot f_{n}(x_{n})d\nu^D_{x_{0}}(\und{x}):\\= f_{0}(x_{0})\bigl( \sum_{\sigma(x_{1})= x_{0}}D(x_{1})f_{1}(x_{1})\bigl(\sum_{\sigma(x_{2})=x_1}D(x_{2})f_{2}(x_{2})\\ \cdots
 \bigl(\sum_{\sigma(x_{n})=x_{n-1}}D(x_{n})f_{n}(x_{n})\bigr)\cdots\bigr)\bigr),
\end{multline}
which is valid for all $f=f_{0}f_{1}\cdots f_{n}$ with $f_{i}\in C(X)$, yields a well-defined probability measure  supported on $p^{-1}({x_{0}})$ that we have denoted $\nu^D_{{x_{0}}}$.

The formula \ref{eq:Markov1} becomes more manageable when one uses the transfer operator $\cL_D$ to abbreviate the sums. The result is: 

\begin{multline}\label{eq:Markov1.2}
\int_{X_\infty} f_{0}(x_{0})\cdot f_{1}(x_{1})\dots \cdot f_{n}(x_{n})d\nu^D_{x_{0}}(\und{x}):\\
 =\left(f_{0}\cL_D\bigl(f_1\cL_D\bigl(f_2 \cdots
 \cL_D\bigl(f_{n-1}\cL_D(f_n)\bigr)\cdots \bigr)\bigr)\right)(x_0),
\end{multline}
which not only makes it clear that  $\nu^D_{x_0}$ is a probability measure on $X_{\infty}$, it shows also that the family $\{\nu^D_{x_0}\}_{x_0\in X}$ may be viewed as an element of $C(X,\fM_1(X_{\infty}))$, where $\fM_1(X_{\infty})$ is the space of all probability measures on $X_{\infty}$ with the weak-$\ast$ topology, i.e., $\{\nu^D_{x_0}\}_{x_0\in X}$ is a continuous family of probability measures on $X_{\infty}$ indexed by $X$. 
Because the family $\nu^D = \{\nu^D_{x_0}\}_{x_0\in X}$ is a system of measures on $X_{\infty}$ that is fibred over $X$ by $p$ (i.e., each $\nu^D_{x_0}$ is supported on $p^{-1}(x_0), x_0\in X$), we shall refer to $\nu^D$ as the \emph{$p$-system of measures on $X_{\infty}$ determined by $D$}.

\subp Finally,  letting $\lambda:= \{\lambda^x\}_{x\in X}$ denote the counting measure Haar system on $G(X,\sigma)$ (paragraph \ref{p:defDRgpdalg}), the formula 
       \begin{multline}\label{def:nustarlambda}
  \int_{G_{\infty}(X,\sigma)}f(\und{x},(p(\und{x}),t, p(\und{y})),\und{y})\, d(\nu^D\ast \lambda)^{\und{x}}:
  \\=
  \int_{G(X,\sigma)}\int_{X_{\infty}}f(\und{x},(p(\und{x}),t,p(\und{y})), \und{y})\, d\nu^D_{p(\und{y})}(\und{y})\, d\lambda^{p(\und{x})}(p(\und{x}),t,p(\und{y})),
      \end{multline}
    where $f\in C_c(G_{\infty}(X,\sigma))$, defines the system of measures $\nu^D*\lambda: = \{(\nu^D*\lambda)^{\und{x}}\}_{\und{x}\in X_{\infty}}$ on $G_\infty(X,\sigma)$.
Close inspection reveals that \eqref{def:nustarlambda} is a special case of \eqref{eq:HaarSystYGY}, i.e., $\nu^D\ast \lambda = \beta$ specialized to $G_{\infty}(X,\sigma)$. Thus $\nu^D\ast \lambda$ is a Haar system on $G_{\infty}(X,\sigma)$  if $\nu^D$ is full. The following proposition describes when this happens.

\begin{prop}\label{prop:fullness}
    $\nu^D$ is full if and only if $D$ is nonvanshing, i.e., if and only if $D$ is a full potential.
\end{prop}

\begin{proof}
    Observe that equation \eqref{eq:Markov1.2} implies that $\nu^D_{x_0}((U_0\times U_1\times X\times \cdots)\cap p^{-1}(x_0))= 1_{U_0}(x_0)\times \cL_D(U_1\cap p^{-1}(x_0))$. So, if $(U_0\times U_1)\times X \times \cdots )\cap p^{-1}(x_0)$ is nonempty, then $x_0\in U_0$ and there must exist a $y\in \sigma^{-1}(x_0)\cap U_1$. It follows that $\nu^D_{x_0}(U_0\times U_1 \times X\cdots)\cap p^{-1}(x_0)\geq D(y)> 0$. Thus $\nu^D_{x_0}$ gives positive measure to each open set that meets $p^{-1}(x_0)$, which means that $\nu^D$ is full. 
    For the converse, suppose $D(y)=0$ for some $y\in \sigma^{-1}(x_0)$. Then there is an open $U_1$ such that $y\in U_1$ and $\sigma$ restricted to $U_1$ is a homeomorphism onto an open set $U_0$. Then $\nu_{x_0}^D(U_0\times U_1\times X\times X \cdots)=D(y)=0$.
\end{proof}
    \begin{remark}\label{rem:xzero}
    In the discussion so far, we have added the subscript $0$ to $x$ to  emphasize the fact that $\nu_{x_0}^D$ is supported on the zero$^{\text{th}}$ component of $X_{\infty}$.  We shall drop it from now on unless it helps materially to make the distinction between that component and $x$ as an element of  the space $X$.\end{remark}

\section{$p$-systems of measures and Conditional Expectations}\label{sec:CondExps}

\p Before describing the special features of the convolution- and $C^*$-algebra built from $G_{\infty}(X,\sigma)$ and the $p$-system of measures $\nu^D:=\{\nu^D_x\}_{x\in X}$, where $D$ is a fixed full normalized potential, we want to spend a little time to highlight features of these systems. Our discussion will expand and complement what we did in paragraph \ref{p:DefPots}.
We assume that $G_\infty(X,\sigma)$ is endowed with the Haar system 
determined by $D$ as in \eqref{def:nustarlambda}.

\subp \label{subp:nuD}First, note that the defining formula for $\nu^D$ reveals it to be a map from $C(X_{\infty})$ to $C(X)$ that is linear, positive, and unital : 
\begin{equation}\label{eq:nuDMap}
    \nu^D(f)(x): = \int_{X_{\infty}}f(\und{y})\, d\nu_x^D(\und{y}),\qquad f\in C(X_{\infty}).
\end{equation}

Most important for our purposes, \eqref{eq:nuDMap} makes it clear that $\nu^D$ is a left inverse of the map $p_*:C(X)\to C(X_{\infty})$ that is dual to $p$:
\begin{equation}
    p_*(f)(\und{x}):=f(p(\und{x})),\qquad f\in C(X).
\end{equation}
Indeed, 
\begin{multline}\label{eq:leftinvpdual}
    [(\nu^D\circ p_*)(f)](x)
    = [\nu^D(p_*(f))](x) 
    = \int_{X_\infty}p_*(f)(\und{y})\,d\nu^D_{x}(\und{y})\\ 
    = \int_{X_{\infty}}f(p(\und{y}))\, d\nu^D_{x}(\und{y})
    = \int_{X_{\infty}} f(x) \, d\nu^D_x(\und{y}) 
    = f(x).
\end{multline}
The penultimate equation reflects the fact that each $\nu_x^D$ is a probability measure on $X_{\infty}$  supported on $p^{-1}(x)$.
Consequently, also, \eqref{eq:leftinvpdual} shows that $\nu^D$ is a surjective map from $C(\Xinf)$ onto $C(X)$.

For much of our analysis, we find it preferable to write all the variables that appear in \eqref{eq:nuDMap} in terms of points in $X_{\infty}$.  
So \eqref{eq:nuDMap} may be rewritten as:
\begin{equation}\label{eq:nuDMapRw}
    \nu^D(f)(p(\und{x})): = \int_{X_{\infty}}f(\und{y})\, d\nu_{p(\und{x})}^D(\und{y}),\qquad f\in C(X_{\infty}).
\end{equation}

\subp \label{subp:cEzero}In analogy with subparagraph \ref{subp:AddRelRandL}, 
we find that $p_*\circ \nu^D$ is a conditional expectation, 
$\cE_0$, on $C(X_{\infty})$, whose definition expressed in terms of integrals using \eqref{eq:nuDMapRw} is
\begin{multline}\label{eq:E0def}
\cE_0(f)(\und{x})
= [(p_*\circ \nu^D)(f)](\und{x})
= [p_*(\nu^D(f))](\und{x})\\
= (\nu^D(f))(p(\und{x}))
= \int_{\Xinf} f(\und{y})\, d\nu^D_{p(\und{x})}(\und{y}),
\qquad f\in C(\Xinf).
\end{multline}
This equation shows that $\cE_0$ is a unital, positive, map from $C(\Xinf)$ to $C(\Xinf)$, while the fact that $\cE_0$ is idempotent is immediate from \eqref{eq:leftinvpdual}:
\begin{equation}
    \cE_0 \circ \cE_0 = (p_*\circ \nu^D)\circ (p_*\circ \nu^D) = p_*\circ (\nu^D \circ p_*) \circ \nu^D) = p_*\circ  \nu^D= \cE_0. 
\end{equation}
Finally, the fact that the range of $\cE_0$ consists of all functions in $C(\Xinf)$ that depend only on the first variable is clear from \eqref{eq:E0def}. Evidently, such functions form an algebra isomorphic to $C(X)$, and $\cE_0$  is a bimodular map over over it. Consequently, we shall informally, but unambiguously, write $\cE_0(a\cdot f \cdot b)=a\cdot \cE_0(f)\cdot b$ for all $f\in C(\Xinf)$ and $a,b\in C(X)$.

\subp \label{subp:muinfty}The following equation arises later in our analysis (see \eqref{eq:UfuNorm} and \eqref{eq:UfuinverseNorm}). We believe it will be helpful to parse it now. For $f\in C(X_\infty)$, 
\begin{equation}\label{eq:muinfty}
\sum_{\sigma(z)=x}D(z)\int_{X_\infty}f(\sigma_\infty(\und{z}))\,d\nu_{z}^D(\und{z})=\int_{X_\infty}f(\und{x})\,d\nu_{x}^D(\und{x})
\end{equation}
for all $x\in X$.  This equation may be easier to understand, if it is rewritten in terms of the maps that are implicit in it. First, let $\pi_{\infty}$ be defined by the formula
\begin{equation}
    \pi_{\infty}(f):= f\circ \sigma_{\infty}, \qquad f\in C(X_\infty).
\end{equation}
Thus $\pi_{\infty}$ is an analogue of the endomorphism $\pi$ of $C(X)$ defined in Proposition 
\ref{defpi}. 
This time, however, $\pi_{\infty}$ is an automorphism of $C(X_{\infty})$, and its inverse is given by composing with 
$\sigma_{\infty}^{-1}$, i.e., 
$\pi_{\infty}^{-1}(f)= f\circ \sigma_{\infty}^{-1}$ for all 
$f\in C(X_{\infty})$. 
So, now we can parse equation \eqref{eq:muinfty} in terms of maps:
\begin{equation}\label{eq:muinftybis}
    \cL_D\circ \nu^D\circ \pi_{\infty} = \nu^D,
\end{equation}or a bit more symmetrically, we may write
\begin{equation}\label{eq:muinftybis1}
    \cL_D\circ \nu^D = \nu^D\circ \pi_{\infty}^{-1}.
\end{equation}
Keep in mind that both sides of each of these equations are maps from $C(X_{\infty})$ to $C(X)$.

\p For our purposes here, the most important role of the $p$-system of measures $\nu^D$ 
 will be revealed in equation \eqref{def:bbE} that is based on this lemma:

\begin{lem}\label{lem:properPhi}
    The map $\Phi:G_{\infty}(X,\sigma) \to G(X,\sigma)$ defined by the equation
\begin{equation}\label{def:Phi}
    \Phi((\und{x},k-l,\und{y})):=(p(\und{x}),k-l,p(\und{y})),\qquad (\und{x},k-l,\und{y})\in G_{\infty}(X,\sigma).
\end{equation}
is a continuous, open, \emph{proper} and surjective groupoid homomorphism mapping the pullback groupoid $G_{\infty}(X,\sigma)$ onto $G(X,\sigma)$.
\end{lem}
\begin{proof}
    Since $p$ is continuous and open, it is clear that $\Phi$ is also.  
    Also, it is clear that $\Phi$ is a surjective groupoid homomorphism mapping $G_{\infty}(X,\sigma)$ onto $G(X,\sigma)$. 
    Since $G_{\infty}(X,\sigma)$ and $G(X,\sigma)$ are locally compact and Hausdorff, 
    to show that $\Phi$ is proper, 
    it suffices to show that for each compact subset 
    $K\subset G(X,\sigma)$, $\Phi^{-1}(K)$ 
    is compact in $G_{\infty}(X,\sigma)$.  For this purpose, simply note that since each of the sets $R_{n,m}$ from equation \eqref{eq:Rnm} is compact and open in $G(X,\sigma)$ and since $G(X,\sigma)$ is covered by the family of them, any compact subset of $G(X,\sigma)$ is covered by a finite number of them. But also note that $\Phi^{-1}(R_{n,m})$ is clearly compact in $G_{\infty}(X,\sigma)$.  It follows that $\Phi$ is proper.
\end{proof}

\subp \label{subp:Phinulambda} As a consequence of this lemma, $\Phi_*$, defined by the equation
\begin{equation}\label{eq:Phisubstar}
\Phi_*(f):= f\circ \Phi, \qquad f\in C_c(G,\sigma),
\end{equation}
is a linear map from $C_c(G,\sigma)$ into $C_c(G_{\infty},\sigma)$ that is continuous with respect to the inductive limit topologies on each space.  It is also clear from equation \eqref{def:nustarlambda} that for all $f\in C_c(G(X,\sigma))$
\begin{multline}\label{eq:nudlambdalambda}
\int_{G_{\infty}(X,\sigma)}\Phi_*(f)(\und{x},t, \und{y})\, d(\nu^D\ast \lambda)^{\und{x}}:\\=
  \int_{G(X,\sigma)}\int_{X_{\infty}}\Phi_*(f)(\und{x},(x_0,t,y_0), \und{y})\, d\nu^D_{ y_0}(\und{y})\, d\lambda^{x_0}(x_0,t,y_0)\\=\int_{G(X,\sigma)}\int_{X_{\infty}}f(x_0,t,y_0)\, d\nu^D_{ y_0}(\und{y})\, d\lambda^{x_0}(x_0,t,y_0)\\=\int_{G(X,\sigma)}f(x_0,t,y_0)\,  d\lambda^{x_0}(x_0,t,y_0).
\end{multline}
The reason for the last equality is that once $(x_0,t,y_0)$ has been fixed, $f(x_0,t,y_0)$ may be viewed as a constant function on the fibre $p^{-1}(y_0)$ on which the measure $\nu^D_{y_0}$ is supported. Since $\nu^D_{y_0}$ is a probability measure, the inner integral evaluates to $f(x_0,t,y_0)$, which leaves the integral $\int_{G(X,\sigma)}f(x_0,t,y_0)\,  d\lambda^{x_0}(x_0,t,y_0)$ to be evaluated.

\subp \label{pMaps}A succinct way to express \eqref{eq:nudlambdalambda} is to write everything in terms of the maps determined by the systems of measures involved:
\begin{equation}\label{eq:lambdaPhilambdanu}
p_*\circ \lambda = (\nu^D*\lambda)\circ \Phi_*
\end{equation}
where  $\lambda:=\{\lambda^x\}_{x\in X}$ is viewed as a map from $C_c(G(X,\sigma))$ to $C(X)$,
$p_*:C(X)\to C(X_{\infty})$ is the map dual to $p$, $\nu^D * \lambda$ is the map determined by the 
system $\{\nu^D_{p(\und{x})}*\lambda^{p(\und{x})}\}_{\und{x}\in X_{\infty}}$ and $\Phi_*$ is the map 
dual to $\Phi$ acting on functions.

The image of $\Phi_*$ thus consists of those functions in $C_c(G_{\infty}(X,\sigma),\nu^D*\lambda)$ that depend only on the zero$^{\text{th}}$-coordinate. The  equations in  \ref{subp:Phinulambda} and \ref{pMaps} , in particular \eqref{eq:lambdaPhilambdanu}, make it clear that $\Phi_*$ defines an \emph{algebra isomorphism} from $C_c(G(X,\sigma),\lambda)$ onto its image in $C_c(G_{\infty}(X,\sigma),\nu^D\ast \lambda)$ that is a homeomorphism in the inductive limit topologies involved. 

\subp We abbreviate the image, $\Phi_*(f)$, of $f\in C_c(G,\lambda)$ under $\Phi_*$ to $\uund{f}$.

\subp\label{subp:CondExps} Now recall that $1_{R_{0,0}}$ is the identity of $C_c(G(X,\sigma),\lambda)$ (cf.  \ref{par:Rnm}). Consequently, $\uund{1_{R_{0,0}}}$ is a projection \emph{in} $C_c(G_{\infty}(X,\sigma),\nu^D\ast \lambda)$. But a projection \emph{in} a $\ast$-algebra determines a projection \emph{on} the algebra via ``cornering''. So, in our setting, $f\to \uund{1_{R_{0,0}}}\ast f \ast \uund{1_{R_{0,0}}} $ defines an idempotent, positivity-preserving, linear map, which we denote by $\bbE$, whose range is the image of $\Phi_*$. In fact, it is a straightforward calculation to see that 
\begin{multline}\label{def:bbE}
    \bbE(f)(\und{x},k-l,\und{y}):=\uund{1_{R_{0,0}}}\ast f \ast \uund{1_{R_{0,0}}}(\und{x},k-l,\und{y})\\ = \int_{X_{\infty}} \int_{X_{\infty}} f(\und{u},k-l,\und{v})\,d\nu^D_{p(\und{x})}(\und{u})\,d\nu^D_{p(\und{y})}(\und{v})
\end{multline}
for all $f\in C_c(G_{\infty}(X,\sigma),\nu^D\ast \lambda)$. This equation, in turn, coupled with \eqref{eq:E0def} justifies our writing:
\begin{equation}\label{EcErelation}
    \bbE = \cE_0\otimes \iota \otimes \cE_0,
\end{equation}
restricted to $C_c(G_{\infty}(X,\sigma))$, where we are regarding $G_{\infty}(X,\sigma)$ as a locally closed subset of $\Xinf \times \bbZ \times \Xinf$ and where $\iota$ denotes the identity map on $C_c(\bbZ)$.

It is also easy to see that $\bbE$ is bimodular in the sense that $\bbE(\uund{b_1}\ast f \ast \uund{b_2})=\uund{b_1}\ast \bbE(f)\ast \uund{b_2}$ for all $f\in C_c(G_{\infty}(X,\sigma),\nu^D\ast\lambda)$ and $b_i\in C_c(G(X,\sigma),\lambda)$, $i=1,2$.

Thus $\bbE$ has all the trappings of a conditional expectation, except one: $\bbE$ is not unital.  That is not important here. The defining equation \eqref{def:bbE} means that $\bbE$ is a ``cornering'' projection and so extends to a completely positive, contractive and idempotent map on any $C^*$-algebra that is generated by a $\ast$-representation of $C_c(\Ginf(X,\sigma),\nu^D\ast\lambda)$. 

\section{Representations of $G_{\infty}(X,\sigma)$ induced from  $G(X,\sigma)$}\label{sec:GinfInd}

 \p We fix  a full normalized potential $D$ and we consider the Haar system $\nu^D*\lambda$ on 
 $G_\infty(X,\sigma)$ determined by $D$. This Haar system on $G_\infty(X,\sigma)$ is fixed throughout the section. We will induce
 representations of $C^*(G(X,\sigma),\lambda)$ to representations of $C^*(G_\infty(X,\sigma),\nu^D*\lambda)$ using non-full normalized potential derived from $D$
 via the techniques developed in Section \ref{sec:induc-repr}. Since $X$ and $\sigma$ are fixed, we will write $G$ for $G(X,\lambda)$ and $\Ginf$ for $\Ginf(X,\sigma)$
 when the notation becomes  cluttered. 
 
 We consider  a non-negative continuous function $b:X\to\bbR_+$ such that the function
 $D_b(x):=b(x)D(x)$ is another normalized potential for $\sigma$, not necessarily full: $
 \sum_{\sigma(y)=x}D_b(y)=1$ for all $x\in X$. The potential $D_b$ defines a $p$-system of measure $\nu^{D_b}$ on 
 $X_\infty$ via \eqref{eq:Markov1} or, equivalently, \eqref{eq:Markov1.2}. Moreover, $\nu_u^{D_b}$ is 
 absolutely continuous with respect to $\nu_u^D$ since
 \begin{multline}
 \int_{X_\infty} f_0(x_0)f_1(x_1)\dots f_n(x_n)\,d\nu_{x_0}^{D_b}(\und{x}) \\
 =\int_{X_\infty} f_0(x_0)f_1(x_1)\dots f_n(x_n)\prod_{i=1}^n b(x_i)\,d\nu_{x_0}^{D}(\und{x})
 \end{multline}
 for all $f=f_0f_1\cdots f_n$ with $f_i\in C(X)$ and $n\ge 0$. We let $\bbd$ be the Radon-Nykodim derivative such that $\nu_u^{D_b}=\bbd \nu_u^D$ for all $u\in X$. In the notation of
 \eqref{eq:nu_d}, one would write $\nu_{\bbd}^{D_b}$ for $\nu^{D_b}$. We prefer the new notation since $\nu_{\bbd}^{D_b}$ is hard to read and $\bbd$ is determined by $b$.
 We define a groupoid correspondence $(Z,\alpha_b)$ from $(\Ginf(X,\sigma),\nu^D*\lambda)$ to $(G,\lambda)$ using the framework described in 
 paragraph \ref{par:topocorrespnd}, 
 where
 \begin{equation}\label{eq:Z2}
   Z: = X_{\infty}\,{_p\ast_r}\, G(X,\sigma) = \{(\und{x},(p(\und{x}),t,y))\mid \und{x}\in X_{\infty},(p(\und{x}),t,y)\in G(X,\sigma)\}, 
 \end{equation}
and $\alpha_b=\nu^{D_b}*\lambda$.

\p \label{par:GtoGinfty} Our goal now is to relate the representation theory of $G(X,\sigma)$ to the 
representation theory of $\Ginf(X,\sigma)$ using the potential $D_b$. It will be fixed for the remainder 
of this section. 

\subp \label{subp:IntroRieffelsTh} Recall that $C_c(Z)$ can be completed to a $C^*$-correspondence that 
we denote $\cX_\infty$ from $C^*(\Ginf(X,\sigma),\nu^D*\lambda)$ to $C^*(G(X,\sigma),\lambda)$ via the equations 
\eqref{eq:rightZBact}, $\eqref{eq:CcGinnerprod}$, and $\eqref{eq:leftYGYZact}$. In particular, 
$C^*(\Ginf(X,\sigma),\nu^D*\lambda)$ acts as multipliers on $\cX_\infty$ via    
\begin{multline}\label{eq:bbLformula}
    \bbL(a)\xi(\und{x},(p(\und{x}),t,y)):=a\cdot \xi(\und{x},(p(\und{x}),t,y))\\= \int_{G_{\infty}} a(\und{x}, s, \und{w})\xi(\und{w},(p(\und{w}),t-s,y))\bbd(\und{w})^{1/2}/\bbd(\und{x})^{1/2}\, d(\nu^D\ast \lambda)^{\und{x}}(\und{x},s,\und{w}).
\end{multline}

\subp Fix a representation $\hat{L}= (\mu,X\ast \cH, \hat{L})$ of the groupoid $G$ and let $\fL$ be its integrated form representing $C^*(G,\lambda)$ on $L^2(X\ast\cH,\mu)$ as described in \ref{subp:RepsofG}. We want to explain how to induce $\hat{L}$, $\fL$ and all related structure to the pullback groupoid, $\Ginf(X,\sigma)$. The relevance of the normalized potential, $D_b$, will be tracked in the process.

The induced 
Hilbert bundle determined by $\go\ast \cH = X\ast \cH$ is the bundle $X_{\infty}\ast \cH = \Ginf(X,\sigma)^{(0)}\ast \cH$ that is the pull back of $X\ast \cH$ by $p$, i.e., the fibre $H_{\und{x}}:= H_{p(\und{x})}$. This bundle is the same for all the other induced structures we consider. The isomorphism groupoid of $X_\infty\ast \cH$, $\Iso(X_{\infty}\ast \cH)$, is similarly identified:
\begin{multline}
  \Iso(X_{\infty}\ast \cH):=\{(\und{x},U,\und{y})\mid U\in B_{\Iso}(H_{\und{y}}, H_{\und{x}})\}\\= \{(\und{x},U,\und{y})\mid U\in B_{\Iso}(H_{p(\und{y})}, H_{p(\und{x})})\} 
\end{multline}

The representation $\hat{L}_{\infty}$ of $\Ginf(X,\sigma)$ on $X_{\infty}\ast \cH$ induced by $\hat{L}$ is given by the formula 
\begin{equation}
    \hat{L}_\infty(\und{x},t,\und{y})=(\und{x},L_{\infty,(\und{x},t,\und{y})},\und{y}),
\end{equation}
where 
\begin{equation}
    L_{\infty,(\und{x},t,\und{y})}(\fh(\und{y})): = L_{(p(\und{x}),t,p(\und{y}))}\fh(p(\und{y})),
    \qquad (\und{x},t,\und{y})\in \Ginf, \fh\in X_{\infty}\ast \cH.
\end{equation}
Its action on $X_{\infty}\ast \cH$, too, is independent of the potential $D_b$ . 

What depends upon $D_b$ is the quasi-invariant measure on $\go_{\infty}=X_{\infty}$. Viewed as a map acting from $C(\go_{\infty})$ to $\bbC$, it is defined by the equation
\begin{equation}\label{eq:munuD}
    \mu_{\infty}^{D_b}:=\mu\circ \nu^{D_b}.
\end{equation}
Note that $\mu_{\infty}^{D_b}$ is well-defined: $\nu^{D_b}$ is a map from $C(\go_{\infty})$ to $C(\go(X,\sigma))=C(X)$ and $\mu$ is a map from $C(X)$ to $\bbC$.
Recall from \eqref{eq:muindqi} that the modular function is 
\begin{equation}
    \Delta_{\infty}(\und{x},k-l,\und{y}) =\Delta_\mu(p(\und{x}),k-l,p(\und{y}))\bbd(\und{x})/\bbd(\und{y}).
\end{equation}

\p We write $\fL_\infty$
for the integrated representation of $C^*(G_\infty(X,\sigma),\nu^D*\lambda)$ on $L^2(X_\infty*\cH_\infty,\mu_\infty^{D_b})$. For 
$f\in C_c(G_\infty)$,  $\fL_{\infty}(f)$ acts on $L^2(X_\infty*\cH_\infty,\mu_\infty^{D_b})$  via
\begin{multline}
    \label{eq:Ind_infty_int}
    \fL_\infty(f)\ff(\und{x})\\
    = \int_{G_{\infty}}f(\und{x},k-l,\und{y})L_{(p(\und{x}),k-l,p(\und{y}))}\ff(\und{y})\Delta_\infty(\und{x},k-l,\und{y})^{-1/2}\,d(\nu^D*\lambda)^{\und{x}}(\und{x},k-l,\und{y})
    \end{multline}
for all $\ff\in L^2(X_\infty*\cH_\infty,\mu_\infty)$.

Although the formula for $\fL_{\infty}$ looks complicated, it lays bare all the relevant parameters and is manageable in the special cases that we study here. In particular, it will appear in Paragraph \ref{par:Mallatsthm} and the proof of our version of Mallat's Theorem 2, Theorem \ref{thm:Mallat2} is based on it.

\section{Proto-Multiresolution Analyses and Their Representations}\label{sec:PMRAs}

\p\label{par:UnitExt} In this section, we return to the elements of Rieffel's theory that are developed in paragraph \ref{par:Rieffel_pers}. We continue with the same notation and we let $\fu\in C(X)$ be a function such that $|\fu|^2=D_b$.  Our first goal is to define a unitary operator $U_\fu$ on $\cX_{\infty}$ that conjugates the image of $C^*(G_\infty(X,\sigma),\nu^D\ast \lambda)$ in $\cL(\cX_{\infty})$ into itself and is a unitary extension of the isometry $S_\fu$  defined in \eqref{eq:S_fu_HC}. 
This unitary will be used to generate what we call a proto-multiresolution analysis.  In a sense we shall make clear, constructing a proto-multiresolution analysis constitutes the first step when building an MRA. Proto-multiresolution analyses put into evidence precisely the variables that can be ``tuned'' to produce MRAs.

\subp Again, we abbreviate $G(X,\sigma)$ to $G$ and $G_{\infty}(X,\sigma)$ to $G_{\infty}$.
The putative unitary extension of $S_\fu$, $U_\fu$, is defined initially on $C_c(X_\infty*G)$  via
\begin{equation}\label{eq:U_Ginfty}
(U_\fu \xi)(\und{x},(p(\und{x}),n,y)):=\fu(p(\und{x}))\xi(\sigma_{\infty}(\und{x}),(p(\sigma_\infty(\und{x})),n-1,y))
\end{equation}
for all $\xi\in C_c(X_\infty*G))$ and $(\und{x},(p(\und{x}),n,y))\in X_\infty*G$.

Before proving that $U_\fu$ is a unitary extension of $S_\fu$, we first show that $U_\fu$ is an isometry and compute its adjoint.
 Let $\xi\in C_c(X_\infty*G)$ and compute, using \eqref{eq:CcGinnerprod},  \eqref{eq:muinfty}
and its compact form \eqref{eq:muinftybis}: 
\begin{multline}\label{eq:UfuNorm}
    \langle U_\fu \xi,U_\fu\xi\rangle_{C_c(G)}(x,n,y):\\
    =\int_G \int_{X_\infty}\overline{U_\fu \xi(\und{z},(p(\und{z}),m,x))}U_\fu\xi(\und{z},(p(\und{z}),m+n,y))\,d\nu^{D_b}_z(\und{z})\,d\lambda_x(z,m,x)\\
    =\int_G \int_{X_\infty}\overline{\fu(z)}\overline{\xi(\sigma_\infty(\und{z}),(p(\sigma_\infty(\und{z})),m-1,x))}\\
    \cdots\fu(z)\xi(\sigma_\infty(\und{z}),(p(\sigma_\infty(\und{z})),m+n-1,y))\,d\nu^{D_b}_z(\und{z})\,d\lambda_x(z,m,x)\\
    = \int_G \sum_{\sigma(z)=t}D_b(z)\int_{X_\infty} \overline{\xi(\sigma_\infty(\und{z}),(p(\sigma_\infty(\und{z})),m-1,x))}\\
    \cdots\xi(\sigma_\infty(\und{z}),(p(\sigma_\infty(\und{z})),m+n-1,y))\,d\nu^{D_b}_z(\und{z})\,d\lambda_x(t,m-1,x)\\
    = \int_G \cL_{D_b}\circ \nu_t^{D_b}\circ \pi_\infty\bigl(\overline{\xi(\cdot,(\cdot,m-1,x))}\xi(\cdot,(\cdot,m+n-1,y))\bigr)\,d\lambda_x(t,m-1,x)\\
    =\int_G \nu_t^{D_b}\bigl(\overline{\xi(\cdot,(\cdot,m-1,x))}\xi(\cdot,(\cdot,m+n-1,y))\bigr)\,d\lambda_x(t,m-1,x)\\
    =\int_G\int_{X_\infty} \overline{\xi(\und{t},(p(\und{t}),m,x))}
    \xi(\und{t},(p(\und{t}),m+n,y))\,d\nu^{D_b}_z(\und{t})\,d\lambda_x(t,m,x)\\
    =\langle \xi,\xi\rangle_{C_c(G)}(x,n,y),
\end{multline}
where the penultimate equality follows by relabeling the integers, i.e., by the invariance of the Haar system $\lambda$.

To complete the proof that $U_\fu$ is a unitary,
we must examine first how the zero set of $D_b$, $Z_{D_b}$, affects the calculation of the $C_c(G(X,\sigma))$-norm of elements in $C_c(X_\infty*G)$.  Let $z_1\in Z_{D_b}$ and  let $z_0:=\sigma(z_1)$. Consider an open neighborhood $A$
of $z_1$ such that $\sigma|_A:A\to \sigma(A)$ is a homeomorphism onto the open set $\sigma(A)$. Let $f\in C(X_\infty)$ be supported in 
$\sigma(A)*A*X_\infty$. Then 
\[
\nu_{z_0}^{D_b}(f)=D_b(z_1)\nu_{z_1}^{D_b}(f(z_0,z_1,\cdot))=0.
\]
Assume that $\xi\in C_c(X_\infty*G)$ is of the form $\xi(\und{x},(p(\und{x}),n,y))=f_1(\und{x})f_2(p(\und{x}),n,y)$ such that
$f_1$ is supported in $\sigma(A)*A*X_\infty$. Let 
\[
Z_{2,D_b}=\{\und{x}\in X_\infty\,:\,x_1\in Z_{D_b}\}.
\] 
If $\tilde{f}_1=f_1 1_{Z_{2,D}^c}$ and $\tilde{\xi}:=\tilde{f}_1f_2$, then $\langle \xi-\tilde{\xi},\xi-\tilde{\xi}\rangle_{C_c(G)}=0$. That is, in the inner product, an element $\xi\in C_c(X_\infty*G)$ depends only on the values outside the set $Z_{2,D}$.
Hence the operator $U_\fu^*$ defined on $C_c(X_\infty*G)$ via
\begin{equation}
    \label{eq:U_u_star}
    U_\fu^*\xi(\und{x},(p(\und{x}),n,y))=\frac{1}{\fu(p(\sigma_\infty^{-1}(\und{x})))}\xi(\sigma_\infty^{-1}(\und{x}),(p(\sigma_\infty^{-1}(\und{x})),n+1,y))
\end{equation}
for all $(\und{x},(p(\und{x}),n,y))\in X_\infty*G$ such that $\fu(x_1)\ne 0$ is well defined in the $C_c(G)$-inner product \eqref{eq:CcGinnerprod} since $p(\sigma_\infty^{-1}(\und{x}))=x_1$. We claim that $U^*_\fu$ is an isometry and is the inverse of $U_\fu$.
To check that $U_\fu^*$ is an isometry we use version \eqref{eq:muinftybis1} of \eqref{eq:muinfty} and compute as follows:
\begin{multline}\label{eq:UfuinverseNorm}
    \langle U_\fu^*\xi,U_\fu^*\xi\rangle_{C_c(G)}(x,n,y)=\\
    \int_G \int_{X_\infty}\overline{U_\fu^*\xi(\und{z},(p(\und{z}),m,x))}U_\fu^*\xi(\und{z},(p(\und{z}),m+n,y))\,d\nu_z^{D_b}(\und{z})\,d\lambda_x(z,m,x)\\
    =\int_G \int_{X_\infty}\frac{1}{D_b(p(\sigma_\infty^{-1}(\und{z})))}\overline{\xi(\sigma^{-1}_\infty(\und{z}),(p(\sigma_\infty^{-1}(\und{z})),m+1,y))}\cdot \\
    \cdot \xi(\sigma_\infty^{-1}(\und{z}),(p(\sigma_\infty^{-1}(\und{z})),m+n+1,x))\,d\nu_z^{D_b}(\und{z})\,d\lambda_x(z,m,x)\\
    =\int_G \nu^{D_b}_z\circ \pi_\infty^{-1}\bigl(\frac{1}{D_b(p(\cdot))}\overline{\xi(\cdot,(p(\cdot ),m+1,y))}\xi(\cdot,(p(\cdot),m+n+1,x)\bigr)\,d\lambda_x(z,m,x)\\
    =\int_G \cL_{D_b}\circ \nu^{D_b}\bigl(\frac{1}{D_b(p(\cdot))}\overline{\xi(\cdot,(p(\cdot ),m+1,y))}\xi(\cdot,(p(\cdot),m+n+1,x)\bigr)(z)\,d\lambda_x(z,m,x)\\
    =\int_G \sum_{\sigma(z_1)=z}\int_{X_\infty}D_b(z_1)\frac1{D_b(z_1)}\overline{\xi(\und{z},(p(\und{z}),m+1,y))}\cdot \\
    \cdot \xi(\und{z},(p(\und{z}),m+n+1,x))\,d\nu_{z_1}^{D_b}(\und{z})\,d\lambda_x(z,m,x)\\
    =\int_G\int_{X_\infty}\overline{\xi(\und{z},(p(\und{z}),m,y))}\xi(\und{z},(p(\und{z}),m+n,x))\,d\nu_z^{D_b}(\und{z})\,d\lambda_x(z,m,x)\\
    =\langle \xi,\xi\rangle_{C_c(G)}(x,n,y).
\end{multline}
The fact that $U_\fu^*$ is the inverse of $U_\fu$ is a straightforward computation.

\p\label{par:Su_Uu_revisited}
Recall that $\cX_0$ is the $C^*$-correspondence defined by the $(G,\lambda)-(G,\lambda)$ equivalence $(G,\alpha)$, where $\alpha_x=\lambda_x$, as described in
\ref{par:S_u_HC}. The isometry $S_\fu\in \cK(\cX_0)$ is defined via \eqref{eq:S_fu_HC}.
To prove that $U_\fu$ is a unitary extension of $S_\fu$, we define $V_{\infty,0}:\cX_0\to \cX_\infty$ via
\begin{equation}
    \label{eq:V_infty_0}
    V_{\infty,0}\xi(\und{x},(p(\und{x}),t,y))=\xi(p(\und{x}),t,y) \,\,\text{for all}\,\,\xi\in C_c(Z).
\end{equation}
\begin{prop}
    The map $V_{\infty,0}$ extends to an adjointable isometry that interwtwines $U_\fu$ and $S_\fu$, $U_\fu V_{\infty,0}=V_{\infty,0}S_\fu$, such that 
    \begin{equation}
        \label{eq:V_infty_0_*}
        V_{\infty,0}^*\eta(x,t,y)=\int_{X_\infty}\eta(\und{x},(x,t,y))\,d\nu_x^{D_b}(\und{x}),
    \end{equation}
    for all $\eta\in C_c(X_\infty*G)$. Moreover, $V_{\infty,0}(\xi\cdot a)=V_{\infty,0}\xi\cdot a$ for all $\xi \in \cX_0$ and $a\in C^*(G)$.
\end{prop}
\begin{proof}
    The fact that $V_{\infty,0}$ is an isometry follows from the following computation that uses the fact that $D_b$ is unital and, thus, $\nu_x^{D_b}(1)=1$ for all $x\in X$:
    \begin{multline*}
        \langle V_{\infty,0}\xi\,,\,V_{\infty,0}\eta\rangle_{C_c(G)}(x,t,y)\\=\sum_{(z,s,x)\in G}\int_{X_\infty} \overline{V_{\infty,0}\xi(\und{z},(z,s,x))}V_{\infty,0}\eta(\und{z},(z,s+t,y))\,d\nu_z^{D_b}(\und{z})\\
        =\sum_{(z,s,x)\in G}\int_{X_\infty}\overline{\xi(z,s,x)}\eta(z,s+t,y)\,d\nu_z^{D_b}(\und{z})\\
        =\sum_{(z,s,x)\in G}\overline{\xi(z,s,x)}\eta(z,s+t,y) 
        =\langle \xi\,,\,\eta\rangle_{C_c(G)}(x,t,y).
    \end{multline*}
    The fact that $V_{\infty,0}$ intertwines $S_\fu$ and $U_\fu$ is also an easy verification:
\begin{multline*}
    U_\fu V_{\infty,0}\xi (\und{x},(p(\und{x}),t,y))=\fu(p(\und{x}))V_{\infty,0}\xi(\sigma_\infty(\und{x}),(p(\sigma_\infty(\und{x})),t-1,y))\\
    =\fu(p(\und{x}))\xi(\sigma(p(\und{x})),t-1,y)=V_{\infty,0}S_\fu(\und{x},(p(\und{x}),t,y)).
\end{multline*}
The last statement of the proposition follows from the following computation:
\begin{multline*}
    \langle V_{\infty,0}\xi\,,\,\eta\rangle_{C_c(G)}(x,t,y)\\
    =\sum_{(z,s,x)\in G}\int_{X_\infty}\overline{V_{\infty,0}\xi(\und{z},(z,s,x))}\eta(\und{z},(z,s+t,y))\,d\eta_z^{D_b}(\und{z})\\
    =\sum_{(z,s,x)\in G}\overline{\xi(z,s,x)}\int_{X_\infty}\eta(\und{z},(z,s+t,y))\,d\nu_z^{D_b}(\und{z})=\langle \xi\,,\,V_{\infty,0}^*\eta\rangle_{C_c(G)}(x,t,y).
\end{multline*} 
\end{proof}

\p\label{par:PMRA} Recall that we denote by $\uund{f}$ the image of $f\in C_c(G(X,\sigma),\lambda)$ in $C_c(G_\infty(X,\sigma),\nu^D*\lambda)$ under the map
$\Phi_*$ defined in \eqref{eq:Phisubstar}. Define $\bbE_0:=\bbL(\uund{1_{R_{0,0}}})$, where $\bbL$ is the left action of $C^*(G_\infty(X,\sigma),\nu^D*\lambda)$ on $\cX_\infty$ defined in
\eqref{eq:bbLformula}. Then $\bbE_0$ is a projection since $\uund{1_{R_{0,0}}}$ is a projection (see \ref{subp:CondExps}). 
\begin{defn}
The sequence of projections $\{\bbE_j\}^{\infty}_{j= -\infty}$ 
on $\cL(\cX_{\infty})$ defined by the formula
\begin{equation}\label{eq:proto-MRA}
    \bbE_j:=U_\fu^j \bbE U_\fu^{-j}, \qquad j\in \bbZ.
\end{equation} together with $U_\fu$ is called the \emph{proto-multiresolution analysis} determined by $(X,\sigma,D,b)$.
\end{defn}

\begin{thm}\label{thm:proto-res}
    The proto-multiresolution analysis $(\{\bbE_j\}_{j= -\infty}^{\infty},U_\fu)$ satisfies \eqref{eq:pmra}, \eqref{eq:pmrafull}, \eqref{eq:pmrascal}, and \eqref{eq:pmraunitequiv} of \ref{par:Outline}:
\begin{enumerate}
   \item $\bbE_j\geq \bbE_i$ when $j\geq i$,
    \item  $I_{\cX}=\bigvee_{-\infty < j < \infty} \bbE_j$,
    \item  $U_\fu^*\bbE_jU_\fu= \bbE_{j-1}$ for all $j$, and
    \item  the copy of $C^*(G(X,\sigma),\lambda)$  in $C^*(G_{\infty}(X,\sigma),\nu^D*\lambda)$ leaves the range of $\bbE_j\ominus \bbE_{j-1}$ invariant and $U_\fu^{j-i}$ implements a unitary equivalence between the restriction of the copy of $C^*(G(X,\sigma),\lambda)$ to the range of $\bbE_j\ominus \bbE_{j-1}$ and the restriction of the copy to the range of $\bbE_i\ominus \bbE_{i-1}$.
\end{enumerate}
\end{thm}
\begin{proof}
    The proof is a consequence of the definition of the projections $\bbE_j$, the fact that $1_{R_{0,0}}$ is the unit of $C^*(G(X,\sigma),\lambda)$, and the fact that $\bbE_j=\bbL(\uund{1_{R_{j,j}}})$ for all 
    $j\ge 0$. We check this property for $j=1$ for 
    simplicity:
    \begin{multline*}
        U_\fu\bbE_0U_\fu^{-1}\xi(\und{x},(p(\und{x},t,y))=\fu(p(\und{x}))\bbE_0U^{-1}\xi (\sigma_\infty(\und{x}),(p(\sigma_\infty(\und{x})),t-1,y))\\
        =\fu(p(\und{x}))\int_G\int_{X_\infty}\Phi_*(1_{R_{0,0}})(\sigma_\infty(\und{x}),s,\und{w})(U_\fu^{-1}\xi)(\und{w},(w,t-s-1,y))\frac{\bbd(\und{w})^{1/2}}{\bbd(\sigma_\infty(\und{x}))^{1/2}}\\
        \cdots d\nu_w^D(\und{w})\,d\lambda^{\sigma(x)}(\sigma(x),s,w)\\
        =\fu(p(\und{x}))\int_{X_\infty}\frac1{\fu(p(\sigma_\infty^{-1}(\und{x}))}\xi(\sigma_{\infty}^{-1}(\und{x}),(p(\sigma_\infty^{-1}(\und{x})),t,y))\frac{\bbd(\und{w})^{1/2}}{\bbd(\sigma_\infty(\und{x}))^{1/2}}\,d\nu_{\sigma(x)}^D(\und{w})\\
        =\sum_{\sigma(u)=\sigma(x)}\int_{X_\infty}\xi(\und{u},(p(\und{u}),t,y))\frac{\bbd(\und{u})^{1/2}}{\bbd(\und{x})^{1/2}}\,d\nu_u^D(\und{u})=\bbL(1_{R_{1,1}})\xi(\und{x},(p(\und{x}),t,y))
    \end{multline*}
    For example, the last property follows from  the fact that if $a\in C^*(G(X,\sigma),\lambda)$, then 
    \[
    \bbL (\uund{a})=\bbL (\uund{1_{R_{0,0}}*a*1_{R_{0,0}}})=\bbE_0\bbL(\uund{a})\bbE_0.
    \]
\end{proof}

\p\label{par:GM} Let $\hat{L}=(\mu,X*\cH,\hat{L})$ be a representation of $G(X,\sigma)$. As before, we write $\fL$ for the integrated representation of 
$C^*(G(X,\sigma),\lambda)$ on $L^2(X*\cH,\mu)$ defined by $\hat{L}$ via  \eqref{eq:integratedform}. The normalized potential $D_b$ defines a topological correspondence from 
$\Ginf(X,\sigma)$ to $G(X,\sigma)$ (see paragraph \ref{subp:IntroRieffelsTh}). 
Theorem \ref{thm:equivIndRep} implies that the induced $C^*(\Ginf(X,\sigma),\nu^{D}*\lambda)$-representation of $\fL$ is unitarily equivalent to the integrated form $\fL_\infty$ of
the induced $\Ginf(X,\sigma)$ representation $\hat{L}_\infty=(\mu_\infty^{D_b},X_\infty*\cH_\infty,\hat{L}_\infty)$
defined in Section \ref{sec:GinfInd}. Set $\tS_\fu:=\fL(S_\fu)\in \cB(L^2(X*\cH,\mu))$, 
$\tU:=\fL_\infty(U_\fu)\in \cB(L^2(X_\infty*\cH_\infty,\mu_\infty^{D_b}))$, and $V_j:=\fL_\infty(\bbE_j)L^2(X_\infty*\cH_\infty,\mu_\infty^{D_b})$ for all $j\in \bbZ$.
Theorem \ref{thm:proto-res} implies the following result.

\begin{thm}\label{thm:GMRA}
    The family $\{V_j\}$ of subspaces of $L^2(X_\infty*\cH_\infty,\mu_\infty^{D_b})$ satisfies the properties:
    \begin{enumerate}
        \item $V_j\subset V_{j+1}$ for all $j\in\bbZ$;
        \item $\bigcup_{j=-\infty}^\infty V_j=L^2(X_\infty*\cH_\infty,\mu_\infty^{D_b})$;
        \item the copy of $C^*(G(X,\sigma))$ in $C^*(G_{\infty}(X,\sigma))$ leaves $V_j\ominus V_{j-1}$ invariant and $\tU^{j-i}$ implements a unitary equivalence between the restriction of the copy of $C^*(G(X,\sigma))$ to $V_j\ominus V_{j-1}$ and the restriction of the copy to $V_i\ominus V_{i-1}$.
    \end{enumerate}
    If $\tS_\fu$ is a pure isometry then $\bigcap_{j=-\infty}^{\infty} V_j=\{0\}$ and $\{V_j\}$ is a \emph{generalized multiresolution analysis} determined by $(X,\sigma,D,b,\hat{L})$.
\end{thm}
\begin{proof}
    The numbered items follow immediately from the corresponding properties of the proto-resolution analysis $\{\bbE_j\}_{j\in\bbZ}$. 

    For the last part, note that the isometry $V_{\infty,0}:\cX_0\to \cX_\infty$ determines an isometry $W_{\infty,0}:L^2(X*\cH,\mu)\to L^2(X_\infty*\cH_\infty,\mu_\infty^{D_b})$, 
    $W_{\infty,0}\ff(\und{x})=\ff(p(\und{x}))$, such that $\tU_\fu W_{\infty,0}=W_{\infty,0}\tS_\fu$ and $V_0=W_{\infty,0}(L^2(X*\cH,\mu))$. 
    Hence $\bigcap_{j=-\infty}^{\infty}V_j=\bigcap_{j\ge 0}W_{\infty,0}\tS_\fu^j(L^2(X*\cH,\mu))$. Therefore $\bigcap_{j=-\infty}^{\infty}V_j=\{0\}$ if $\tS_\fu$ is a pure isometry.
\end{proof}
\section*{{\Large{Part III}}\\Mallat's Theorem Redux}

\section{Potentials, filters and scaling functions}\label{sec:pot,filt,scalfns}

\p We fix a transfer operator $\cL_{\psi}$ for $\sigma$ as in paragraph \ref{p:DefPots} and we assume the potential $\psi$ is full.
In this section, we define and develop the concept of a filter associated to the potential
$\psi$. A \emph{filter} is a function that determines a new potential that is normalized but not necessarily full. In fact, in order to produce an analog of the scaling function that appears in \eqref{eq:QMF}, the new potential is definitely not full (see \ref{subp:secsupport}).
A  filter defines a $p$-system of measures on the projective limit $X_\infty$, where, 
recall, $p:X_\infty\to X$ is the projection onto the first component. The main result of 
this section is Proposition \ref{prop:scalingfunction}., which shows 
that a scaling function for $(X,\sigma)$ exists if and only if the system of measures is atomic.

\p\label{par:unital_psi} The full potential $\psi$ will be used to define the Haar system on the pullback groupoid $X_\infty*G*X_\infty$ via \eqref{eq:HaarSystYGY}.
To keep the notation shorter, we will assume in the remainder of the paper that $\psi$ is an unital potential in addition to being a full potential. 
If $\psi$ is not unital, then one needs to replace $\cL_\psi$ with $\cL_{D_0}$ where $D_0(x)=\psi(x)/\cL_\psi(1)(\sigma(x))$ in the formula 
below and in the construction of the $p$-system of measures on $X_\infty$ (see Remark \ref{rem:fullnonfull}).

We chose to use $\psi$ as the notation for the initial full potential rather than the more standard notation $D$ to avoid 
confusion when we multiply $\psi$ with 
a function $b$ to obtain another unital, but not necessarily full, potential $D_b$. Many of the formulas below will involve both the potentials $\psi$ and $D_b$, and it might be hard 
for the reader to keep track of which one is which, had we used $D$ instead of $\psi$. Further, later in our narrative, we find it useful to abbreviate $D_b$ to $D$.

\p
    A function $\mf\in C(X)$ is called a $\psi$-\emph{filter} if  $\cL_\psi(\vert \mf\vert^2)=1$. That is,
    \begin{equation}
        \label{eq:psifilter}
        \sum_{\sigma(y)=x}\psi(y)\vert \mf(y)\vert^2=1
    \end{equation}
    for all $x\in X$.    

If $\mf$ is a $\psi$-filter, we write $\fu$ for the continuous function $\fu:X\to\bbC$ defined via $\fu(x)=\sqrt{\psi(x)}\mf(x)$ and we define $D:X\to \bbR_+$ via 
\begin{equation}
\label{eq:D_from_fu}
D(x)=\vert \fu(x)\vert^2=\psi(x)\vert \mf(x)\vert^2\,\text{ for all }x\in X.
\end{equation}
Connecting the notation with the notation used in previous sections, we have $D(x)=D_b(x)=b(x)\psi(x)$ where $b(x)=\vert \mf(x)\vert^2$.

\subp We remark in passing, for those who know about Hilbert $C^*$-modules, that equation \eqref{eq:psifilter} reveals a $\psi$-filter as a unital unit in the Hilbert $C^*$-module over $C(X)$ obtained by completing $C(X)$ in the pre-inner product $\langle \xi,\eta\rangle:=\cL_{\psi}(\overline{\xi}\eta)$. This perspective is taken in \cite{LarRae_CM06}.

\subp By definition $D$ is a normalized
potential for a new unital transfer operator $\cL_D$, but note that $D$ does not have to be full. In fact, the 
failure of $D$ being full provides  our most important applications. For a given potential $\psi$, one 
can find different $\psi$-filters $\mf$. As we will see in examples, some of them will lead to full  
normalized potentials while others will lead to normalized potentials that are not full. 
In a sense, one can view the remainder of the paper as the study of such normalized potentials and how
non full potentials can lead to  wavelets.
 
\subp In the remainder of this section, we fix a $\psi$-filter $\mf$ and the maps $\fu$ and $D$ defined by
$\mf$. Recall from \eqref{eq:Markov1} that $D$  defines a  $p$-system of measures on $X_\infty$ 
denoted by $\{\nu^D_{x}\}_{x\in X}$.
 We prove that a scaling function as defined in \eqref{eq:scalefunction} below exists if and only
 if this $p$-system of measures is atomic. The scaling function will be a key ingredient
in our extension of Mallat's theorem in Section \ref{sec:arbitraryscale2}. One should keep in mind that a scaling function provides a link or ``map'' from analysis on $X$ that is controlled by $\sigma$ to classical harmonic analysis on $\bbR$ that is built on the Fourier transform.

\p 
We say 
that a function  $\fs:X\to X_\infty$ is a 
\emph{section} of the projective limit induced by $(X,\sigma)$
if $\fs$ is a (Borel) measurable function such that $p\circ\fs(x)=x$ for all $x\in X$. We then say that such a 
section $\fs$
is \emph{positively supported} on $X_\infty$ with respect to  $\mf$ if $\nu_{x}^D(\{\fs(x)\})>0$ 
for all $x\in X$.

\subp
Evidently, a function $\fs:X\to X_\infty$ is a section if and only if there is $\und{x}=(x_n)_{n\ge 0}\in X_\infty$ with
$x_0=x$ and $\sigma(x_{n+1})=x_n$ for all $n\ge 0$ such that $\fs(x)=\und{x}$.
Moreover, $\fs$ is positively supported if and only if $\fs(x)$ is an atom for $\nu_{x}^D$ for all $x\in X$.    

\begin{prop}\label{prop:scalingfunction}
    Assume that $\mf$ is a $\psi$-filter and $\fs:X\to X_\infty$ is a section of the  projective limit
    induced by $(X,\sigma)$. 
    Then $\fs$ is positively supported with respect to $\mf$ if and only if the function $\phi:X\to \bbC$
    \begin{equation}
        \label{eq:scalefunction}
        \phi(x):=\prod_{i=1}^\infty \fu(x_i)
    \end{equation}
    is defined for all $x\in X$, where $(x_n)_{n\ge 0}=\fs(x)$.
    The function $\phi$  satisfies the \emph{scaling property}
    \begin{equation}
    \label{eq:scalingprop}
\phi(\sigma(x))=\fu(x)\phi(x)
\end{equation}
and if $\fs$ is positively supported with respect to $\mf$ and if $\mu$ is an arbitrary probability measure on $X$ then
    $\phi\in L^2(X,\mu)$.
\end{prop}

    \begin{proof}.
    Let $x\in X$ and $(x_n)_{n=0}^\infty:=\fs(x)$. It follows from the definition of $\nu_{x}^D$
    that the product $\displaystyle \prod_{i=1}^\infty D(x_i)=\prod_{i=1}^\infty \vert \fu(x_i)\vert^2$ converges if and only if
    \[
     \nu_{x}^D(\{\fs(x)\})=\prod_{i=1}^\infty D(x_i)=\vert \phi(x)\vert^2>0.
    \]
    Therefore $\displaystyle \prod_{i=1}^\infty \fu(x_i)$ converges for all $x\in X$ if and only if $\fs$ is 
    positively supported. 

    For the scaling property, note that $\sigma_\infty(\fs(x))=\fs(\sigma(x))$ for all $x\in X$. Therefore
\[
\phi(\sigma(x))= \prod_{i=1}^\infty \fu(\sigma(x_i))=\fu(x)\prod_{i=1}^\infty \fu(x_i)=\fu(x)\phi(x).
\]
    
    The function defined in \eqref{eq:scalefunction} is measurable since $\fs$ is measurable. Moreover
    \[
    \int_X\vert \phi(x)\vert^2\,d\mu(x)=\int_X \nu_x(\fs(x))\,d\mu(x)\le \int_X\int_{X_\infty} 1\,d\nu_{x}(\und{x})\,d\mu(x)=1.
    \]
    Therefore $\phi\in L^2(X,\mu)$.

    \end{proof}

\begin{defn}
    We call the function $\phi$ of \eqref{eq:scalefunction} the \emph{scaling function} of the system $(X,\sigma,\psi,\mf, \fs)$.
\end{defn}

\subp\label{rm:scalingfunc}
  Using the fact that $\fu(x)=\sqrt{\psi(x)}\mf(x)$ we can rewrite the scaling property as
\[
\frac1{\sqrt{\psi(x)}}\phi(\sigma(x))=\mf(x)\phi(x).
\]
So, for example, if $\psi(x)=1/N$ for all $x\in X$ for some $N\ge 2$ we recover the well known $N$-scaling
property:
\[
\sqrt{N}\phi(\sigma(x))=\mf(x)\phi(x),
\]

\subp \label{subp:secsupport} Note that a necessary condition for a section $\fs:X\to X_\infty$ to be positively supported is that
$\lim_{n\to \infty}\vert \fu(x_n)\vert=1$ for all $x\in X$, where $(x_n)_{n\ge 0}=\fs(x)$. This 
implies that there is at least one $z\in X$ such that $\fu(z)=1$ since $\{x_n\}_{n\ge 0}$ is
a sequence in the compact space $X$ and $\fu$ is assumed to be continuous. Hence $D(z)=1$  and,
thus, $D(x)=0$ for all $x\in \sigma^{-1}(\sigma(z))\setminus\{z\}$. Therefore, a \emph{necessary} 
condition for the existence of a positively supported section is that the normalized potential $D$ 
is not full.

\section{$\psi$-Filters and isometries in $C_c(G(X,\sigma),\lambda)$}
\label{sec:toHilbertspaces}

Recall from Theorem \ref{thm:GMRA} that to build generalized multiresolution analyses it is natural first to build isometries $S_{\fu}$ in $C^*(G(X,\sigma),\lambda)$ and to seek conditions under which there are representations $\fL$ of $C^*(G(X,\sigma),\lambda)$ such that $\fL(S_{\fu})$ is a pure isometry. In this section we study in more detail the fundamental building blocks: the isometry $S_\fu$ defined
in \eqref{eq:S_fu} and the projections in $C^*( G(X,\sigma),\lambda)$ determined by its iterates. Recall, also, that the key ingredients in the definition of $S_\fu$ 
are, in addition to the local homeomorphism $\sigma$, a full potential $\psi$ and a $\psi$-filter $\mf$. 
We first explain how full potentials show up naturally in the context of the Deaconu-Renault groupoid once one chooses a unitary representation of the groupoid.

\p Recall that a unitary representation of a groupoid requires a quasi-invariant measure on its unit space. In the setting of the Deaconu-Renault groupoid $G(X,\sigma)$, Renault gives a detailed analysis in \cite{Ren_2009} that connects quasi-invariant measures on $X$, the unit space of $G(X,\sigma)$, with full, not-necessarily normalized potentials. His Proposition 3.4.1  proves that any such  quasi-invariant measure $\mu$ is invariant for the transpose\footnote{We use the terms ``transpose'', ``adjoint'' and ``dual'' interchangeably to refer to the operator on measures induced by a transfer's action on functions.} of  a transfer operator given by a full, not-necessarily normalized potential $\psi$. That is, writing $\cL_{\psi}$ for the transfer operator given by $\psi$, $\mu$ satisfies
\begin{equation}
    \label{eq:invariant_dual_L}
    \int_X \cL_\psi(f)(x)\,d\mu=\int_X \sum_{\sigma(y)=x}\psi(y)f(y)\,d\mu=\int_X f(x)\,d\mu(x)
\end{equation}
for all $f\in C(X)$. In general, a transfer operator $\cL_{\psi}$ may have many invariant probability measures. However, there are general conditions, which we describe in paragraph \ref{par:pureisodet}, under which the correspondence between transfer operators and quasi-invariant measures for $(X,\sigma)$ is one-to-one. Since all of our examples satisfy these conditions, we shall proceed to write $\mu_{\psi}$ for the ``pairing'' between quasi-invariant measures $\mu$ and full potentials $\psi$.

\subp\label{par:mu_unital_psi} We continue to assume that $\psi$ is a \emph{unital} full potential to keep our notation shorter. Our results are valid in the non-unital case. However, one needs 
to adjust slightly the computations involving $\mu$ using the fact that
\[
\int_X \cL_{D_0}(f)(x)\,d\mu(x)=\int_X \cL_\psi(1)(\sigma(x))f(x)\,d\mu(x),
\]
where $D_0(x)=\psi(x)/\cL_\psi(1)(\sigma(x))$. 
The reader is encouraged to consult, for example, \cite{DuJo_PAMS07} and the papers cited therein for computations involving similar formulas.

\p Proposition 3.4.1 of \cite{Ren_2009} also proves that the Radon-Nykodim derivative of $\mu_\psi$ is given by
\begin{equation}
\label{eq:Delta}
\Delta_{\psi}(x,k-l,y)=\frac{\psi(x)\psi(\sigma(x))\cdots \psi(\sigma^{k-1}(x))}{\psi(y)\psi(\sigma(y))\cdots \psi(\sigma^{l-1}(y))}.    
\end{equation}
Note that $\mu_\psi$ is also invariant for the dual of the transfer operator $\cL_\psi^n$ whose potential 
is given by $\psi_n(x):=\prod_{i=0}^{n-1}\psi(\sigma^i(x))$. We will use later the following lemma which
is undoubtedly well known to some.  However, since we lack a specific  reference and require consequences of it, we state and prove it as follows.
\begin{lem}
    \label{lem:byparts_transferop}
     Assume that $\mu_\psi$ is a probability measure on $X$ invariant for  the dual of the transfer operator $\cL_\psi$ with potential $\psi$. Then  for all $f,g\in C(X)$ we have
    \begin{equation}
    \label{eq:byparts_1}
    \int_X\cL_\psi^n(f)(x)g(x)\,d\mu_\psi(x)=\int_Xf(x)g\circ \sigma^n(x)\,d\mu_\psi(x)
    \end{equation}
    and
    \begin{equation}
    \label{eq:byparts_2}
    \int_X\cL^n_\psi(f)(\sigma^n(x))g(x)\,d\mu_\psi(x)=\int_Xf(x)\cL_\psi(g)(\sigma^n(x))\,d\mu_\psi(x).
    \end{equation}
\end{lem}
\begin{proof}
    The first equation follows since $\cL_\psi^n$ is a conditional expectation on the range of $\pi^n$ and $\mu_\psi$ is invariant for the dual of $\cL_\psi^n$:
    \[
    \int_X\cL_\psi^n(f)(x)g(x)\,d\mu_\psi(x)=\int_X\cL_\psi(fg\circ\sigma^n)(x)d\mu_\psi(x)=\int_Xf(x)g\circ \sigma^n(x)\,d\mu_\psi(x)
    \]
    The second equation follows by applying the first equation twice:
    \begin{multline*}
    \int_X\cL^n_\psi(f)(\sigma^n(x))g(x)\,d\mu_\psi(x)=\int_X \cL^n_\psi(f)(x)\cL_\psi^n(g)(x)\,d\mu_\psi(x)
    \\=\int_Xf(x)\cL_\psi(g)(\sigma^n(x))\,d\mu_\psi(x).
    \end{multline*}
\end{proof}

Note that in the notation of subparagraph \ref{subp:AddRelRandL}, $\projE_n^\psi:=\pi^n\circ \cL_\psi^n$, equation 
\eqref{eq:byparts_2} can be written as
\begin{equation}
    \label{eq:byparts_3}
    \int_X \projE_n^\psi(f)g\,d\mu_\psi=\int_X f\projE_n^\psi(g)\,d\mu_\psi.
\end{equation}
From now on we fix a probability measure $\mu_\psi$ that is invariant for the dual of a transfer operator with full potential $\psi$ and view it as a 
quasi-invariant measure for the Deaconu-Renault groupoid $G=G(X,\sigma)$.

\p\label{p:Su} Next we consider a $\psi$-filter $\mf$ as defined in \eqref{eq:psifilter} and we continue to use the 
 notation $\fu(x)=\sqrt{\psi(x)}\mf(x)$ and $D(x)=\vert \fu(x)\vert^2=\psi(x)\vert \mf(x)\vert^2$ for all $x\in 
 X$ of the previous section. We stress that $D$ is a normalized not necessarily full potential for its transfer operator $\cL_D$. 
 Recall from \eqref{eq:S_fu} that the function $\fu$ defines an isometry $S_\fu$ in $C_c(G(X,\sigma),\lambda)$.
 It is important to keep in mind that $\mu_\psi$ is not invariant for $\cL_D$ unless 
 $\vert \mf(x)\vert=1$ for $\mu_\psi$-a.e. $x$. As we shall see in Theorem \ref{thm:tS_fu_pure} below, choosing $\mf$ such 
 that $\mu_\psi$ is not invariant for $\cL_D$ is the condition one needs to obtain a pure 
Hilbert space isometry  from $S_\fu$.

\subp     To simplify the notation in the remainder of the paper, we will write
      $\fu_n(x):=\prod_{i=0}
    ^{n-1}\fu(\sigma^i(x))$, $D_n(x):=\prod_{i=0}^{n-1}D(\sigma^i(x))$ and $\mf_n(x):=\prod_{i=0}^{n-1}
    \mf(\sigma^i(x))$ for all $n\ge 0$ and $x\in X$. Additionally, will use the notation $\psi_n$ defined above as the potential of $\cL_\psi^n$ and write
     $D_n$ for the potential of $\cL_D^n$.

\p The unitary extension $U_\fu$ of $S_\fu$ defined in \eqref{eq:U_Ginfty} is one of the pillars of the
 proto-multiresolution analysis (see Theorem \ref{thm:proto-res}). Theorem \ref{thm:GMRA} shows that a choice of a representation of $G(X,\sigma)$  defines a generalized multiresolution analysis in the induced Hilbert space. The 
 same theorem proves that the proto-multiresolution analysis becomes a multiresolution
 analysis if the image of $S_\fu$ under the integrated form of the orginal
 representation is a pure isometry. To facilitate the study we undertake in the next section, which is devoted to determining when the image of $S_\fu$ under the integrated form of a unitary representation is a pure isometry,
 we describe the projections, $E_n:=S_\fu^n*(S_\fu^{*})^n$, onto the ranges of the $S_\fu^n$ and develop some of their useful properties. An easy computation shows that 
    \begin{equation}
    \label{eq:E_n_R_nn}
        E_n(x,k-l,y)=\fu_n(x)\overline{\fu_n(y)}1_{R_{n,n}}(x,k-l,y).    
    \end{equation}
    
    Equation \eqref{eq:L_D_S_u} implies that the following equation holds in $C_c(G)$ for 
    all $f\in C(X)$:
    \begin{equation}
        \label{eq:E_nf_En}
        E_n*f*E_n=\bigl(\pi^{n}\circ \cL_D^n(f)\bigr) *E_n=\projE_n^D(f)*E_n,
    \end{equation}
    where  
    recall $\pi:C(X)\to C(X)$ is the composition with $\sigma$ and $\projE_n^D$
    is the projection $\pi^n\circ\cL_D^n$ defined in paragraph \ref{subp:AddRelRandL}.
 
\section{Unitary representations of $G(X,\sigma)$: Deciding when the image of $S_\fu$ is a pure isometry}
In this section we consider a representation $\hat{L}$ of $G(X,\sigma)$ and provide conditions that guarantee 
that the image of $S_\fu$ under the integrated form of $\hat{L}$ is a pure isometry. Our results generalize numerous previous results in the literature.  See, in particular, \cite{brajor_97,Bagg_co_JFAA09,Bag_co_JFA10}.

\p We continue to fix a quasi-invariant measure $\mu_\psi$ that is invariant for the dual of a transfer operator $\cL_\psi$ 
with unital full potential $\psi$. Let $\hat{L}$ be a representation of $G(X,\sigma)$ into a Hilbert bundle $X*\cH$ over 
$X$. Recall that this means that $\hat{L}$ is a Borel homomorphism of $G(X,\sigma)$ into the unitary groupoid of $X*\cH$, $\Iso(X*\cH)$. Hence, the triple $\hat{L}:=(\mu_\psi,X*\cH,\hat{L})$ constitutes a unitary representation of $G(X,\sigma)$ that may be integrated to a $C^*$-representation $\fL$ of $C_c(G(X,\sigma),\lambda)$ as  in \eqref{eq:repgenG}. Next we fix a $\psi$-filter $\mf$ and
 consider how the functions $\fu(x):=\sqrt{\psi(x)}\mf(x)$ and $D(x)=\vert \fu(x)\vert^2$ interact with $\hat{L}$.

\p\label{par:pureisodet} The main goal of this section is to determine whether the isometry $\tS_\fu:=\fL(S_\fu)$, where $S_\fu$ is
defined in \eqref{eq:S_fu}, is a pure isometry. In order to achieve our goal, we are going to use the reverse 
martingale theorem
as adapted to the groupoid setting in \cite[Proposition 3.5, Theorem 3.10 and Theorem 6.1]{Ren_ETDN_05}. We 
describe 
 briefly how Renault's setting applies to our situation. Equation \eqref{eq:E_n_R_nn} 
implies that the range projection of $\tS_\fu$ depends only on the proper equivalence relation $R_{n,n}$. In 
the language of \cite{Ren_ETDN_05}, $R_\infty=\bigcup_{n}R_{n,n}$ is an approximately proper (AP) equivalence
relation.
Recall that we assume that $\psi$ is an unital full normalized potential. 
We also assume that  the hypotheses of Theorem 6.1 of \cite{Ren_ETDN_05} are satisfied. These hypotheses are
\begin{enumerate}[label=(\roman*)]
    \item the AP equivalence $R_\infty$ is minimal;
    \item there exists an integer $L\ge 1$ such that $\sigma^L$ admits a generator;
    \item there exists an integer $M\ge 1$ such that $\psi_m$ satisfies Bowen's condition with respect
    to $\sigma^M$.
\end{enumerate}
Under these hypotheses, the stationary cocycle $\Delta_\psi$ on $R_\infty$ is uniquely ergodic and we may invoke
the main theorems of \cite{Ren_ETDN_05}. While we do not go into details, we note that if $X$ is a
compact metric space and $\sigma$ is a positively expansive surjective local homeomorphism, then the above hypotheses are satisfied (\cite[Example 6.2]{Ren_ETDN_05}). The local homeomorphisms in the examples that 
we consider in this paper are positively expansive and, thus, satisfy these hypotheses.

\p The next theorem, which is the main result of this section, shows that $\tS_\fu$ is a pure isometry 
provided that $\mu_\psi$ is not invariant for the dual of the transfer operator $\cL_D$. To ease the 
notation a bit, we are going to write $\tilde{E}_n$ for $\fL(E_n)$ throughout the proof.
\begin{thm}
\label{thm:tS_fu_pure}
    Assume that $\sigma:X\to X$ is a surjective local homeomorphism such that the above hypotheses are satisfied. Let $
    \mu_\psi$ be a quasi-invariant measure for $G(X,\sigma)$ that is invariant for the dual of the 
    transfer operator $\cL_\psi$ with unital full potential $\psi$. Let $\mf$ be a $\psi$-filter such that
    $\mu_\psi(\{x\in X\,:\,\vert \mf(x)\vert\ne 1\})>0$. Then the isometry $\tS_\fu=\fL(S_\fu)$ is
    a pure isometry.
\end{thm} 
\begin{proof}
Assume, by contradiction, 
 that $\tS_\fu$ is not pure. Then there is a unit vector $\ff\in \bigcap_{n\ge 0} \tS^n_\fu L^2(X*\cH,
 \mu_\psi)$ and, thus, $\tE_n\ff=\ff$ for all $n\ge 0$. Using the notation of the paragraph \ref{subp:AddRelRandL} we write  $\projE_n^\psi:=\pi^{n}\circ \cL_\psi^n$ and $\projE_n^D:=\pi^n\circ \cL_D^n$.
  We claim that
    \begin{equation}
        \label{eq:norm_ffx}
        \Vert \ff(x)\Vert^2=\vert \mf_n(x)\vert^2\projE_n^\psi(\Vert \ff(\cdot)\Vert^2)(x),
    \end{equation}
    for $\mu_\psi$-a.e. $x\in X$ and for all $n\ge 0$. Equation \eqref{eq:norm_ffx} is a generalization of \cite[Equation (3.26)]{brajor_97} and \cite[Lemma 10]{Bagg_co_JFAA09}. 
    We provide a shorter proof using \eqref{eq:E_nf_En} and \eqref{eq:byparts_3}. 
    
    First note that $\projE_n^\psi$ satisfies 
    \eqref{eq:byparts_3} since $\mu_\psi$ is invariant for $\cL_\psi$. Moreover, $\projE_n^D(f)=\projE_n^\psi(\vert \mf_n\vert^2f)$ since $D_n=\psi_n\cdot \vert \mf_n\vert^2$. Let $f\in C(X)$. Then
    \begin{multline*}
        \int_Xf(x)\Vert \ff(x)\Vert^2\,d\mu_\psi(x)=\int_X \langle f(x)\ff(x),\ff(x)\rangle \,d\mu_\psi(x)=\langle f\ff,\ff\rangle\\
        =\langle f\tE_n\ff,\tE_n\ff\rangle=\langle \tE_n f\tE_n\ff,\ff\rangle=\langle \projE_n^D(f) \ff,\ff\rangle\\
        =\int_X \projE_n^D(f)(x)\Vert\ff(x)\Vert^2\,d\mu_\psi(x)        
        = \int_X \projE_n^\psi(\vert \mf_n\vert^2f) \Vert \ff(x)\Vert^2\,d\mu_\psi(x)\\
        =\int_X \vert \mf_n(x)\vert^2f(x)\projE_n^\psi(\Vert \ff\Vert^2)(x)\,d\mu_\psi(x).
    \end{multline*}
    Hence \eqref{eq:norm_ffx} holds.

     Recall from \ref{subp:AddRelRandL} that the sequence of $\projE_n^\psi$ is a reversed martingale: $\projE_m^\psi\projE_n^\psi=\projE_n^\psi \projE_m^\psi=\projE_n^\psi$ for $m\le n$ (c.f. \cite[Proposition 3.5]{Ren_ETDN_05}). 
    Under our hypotheses,  Theorem 3.10 of \cite{Ren_ETDN_05} implies that
    $\projE_n^\psi(\Vert \ff(\cdot)\Vert)$ converges uniformly to a constant function. Part (iii) of \cite[Corollay A.3]{Ren_ETDN_05} implies that the limit of $\projE_n^\psi(\Vert \ff(\cdot)\Vert)$
    equals $\int_X \Vert \ff(x)\Vert^2\,d\mu_\psi(x)=1$. Therefore
    \begin{equation}
        \label{eq:ff_mf_n}
        \Vert \ff(x)\Vert^2=\lim_{n\to \infty}\vert \mf_n(x)\vert^2
    \end{equation}
    for $\mu_\psi$ a.e. $x\in X$. In particular, $\mf_\infty(x):=\lim_{n\to \infty}\mf_n(x)$ exists for 
    $\mu_\psi$ a.e. $x$. Note that $\vert \mf_\infty(x)\vert^2=\Vert \ff(x)\Vert^2$ for $\mu_\psi$ a.e.
    $x\in X$. Hence $\int_X \vert \mf_\infty(x)\vert^2\,d\mu_\psi(x)=1$ since $\ff$ is a unit vector.

    Now let $k,n\in\bbN$ satisfy the inequality $n>k$. Recall that $\cL_\psi^k(\vert \mf_k\vert^2)=1$.  Using the fact 
    that $\mf_n(x)=\mf_k(x)\prod_{i=k}^{n-1}\mf(\sigma^i(x))$ and the fact that $\cL_\psi^k$ is a 
    transfer operator with range equal to that of  $\pi^{k}$, we obtain
    \begin{multline*}
    \cL_\psi^k(\vert \mf_n\vert^2)(x)=\cL_\psi^k\left( \vert \mf_k\vert^2\prod_{i=k}^{n-1}\vert\mf\circ \sigma^i\vert(x)\right)=\cL_\psi^k(\vert \mf_k\vert^2)(x)\vert \mf_{n-k}(x)\vert^2\\
    =\vert \mf_{n-k}(x)\vert^2.
    \end{multline*}
    Taking the limit as $n\to \infty$, we obtain that $\cL_\psi^k(\vert \mf_\infty\vert^2)=\vert \mf_\infty\vert^2$ for all $k\ge 1$. Therefore $\lim_{k\to\infty}\cL_\psi^k(\vert \mf_\infty\vert^2)=\vert \mf_\infty\vert^2$. On the other hand, Corollary A.3 of \cite{Ren_ETDN_05} implies  that 
    \[
    \lim_{k\to\infty}\cL_\psi^k(\vert \mf_\infty\vert^2)=\int_X \vert \mf_\infty(x)\vert^2\,d\mu_\psi(x)=1.
    \]
    Hence $\vert \mf_\infty(x)\vert=1$ for $\mu_\psi$-a.e. $x$. To finish the proof, we note that $\mf_\infty(x)=\mf(x)\mf_\infty(\sigma(x))$ for $\mu_\psi$-a.e. $x$. Therefore $\vert \mf(x)\vert=1$ for 
    $\mu$-a.e. $x$. This is a contradiction and, hence, $\tS_\fu$ is a pure isometry.
    
\end{proof}

\p In order to relate our result with those in the literature, note that under the assumption that $\psi$ is unital, the 
isometry $\tS_\fu$ is given by the formula
\[
\tS_\fu(\fh)(x)=\mf(x)L_{(x,1,\sigma(x))}\fh(\sigma(x))
\]
for all $\fh\in L^2(\go*\cH,\mu)$. In particular, if $\cH=X\times \bbC$ is the trivial bundle and $L=\iota$ is the 
fundamental representation on $\mu$, $\iota_{(x,t,y)}z=z$, then $\tS_\fu\in L^2(X,\mu)$ (see \ref{subp:1dimfibres}) and it is given by the
formula  
\[
\tS_\fu(h)(x)=\mf(x)h(\sigma(x))
\]
for all $h\in L^2(X,\mu)$. Thus, our result recovers \cite[Theorem 3.1]{brajor_97}, \cite[Theorem 8]{Bagg_co_JFAA09} and \cite[Theorem 3.1]{Bag_co_JFA10}.

\section{The Conclusion:  MRAs with non-constant scaling functions and Mallat's Theorem}
\label{sec:arbitraryscale2}
We apply the techniques developed in the previous sections to a class of examples that provide a far reaching generalization of Mallat's multiresolution analysis. 
The examples that we study are based on the general $n$-fold map of $\bbT$ studied in 
\cite[Section 4.5]{IoKu_IUMJ13} and include, among others, finite Blaschke products.

\p\label{def:example}    Let $\varphi:[0,1]\to \bbR$ be a continuous positive function such that $\varphi(0)=\varphi(1)$ and
    $N:=\int_0^1 \varphi(x)\,dx$ is a  positive integer greater or equal than $2$. We extend
    $\varphi$ to be defined on $\bbR$ via periodicity.  We will assume in the following that $\und{\varphi}:=\min_{x\in [0,1]}\varphi(x)>1$.
    Let $X=\bbT$ viewed as
    $\bbR/\bbZ$ with the corresponding metric. Define $\sigma:\bbT\to\bbT$  via
    \[
    \sigma(e^{2\pi ix})=e^{2\pi i\int_0^x \varphi(t)\,dt}\,\text{ for all }x\in [0,1].
    \]
    Then $\sigma$ is an $N$-fold covering map of $\bbT$ and, in particular, it is a local 
    homeomorphism. 
    \begin{example}
        \begin{enumerate}
            \item Let $\varphi:[0,1]\to\bbR$ be defined via $\varphi(t)=N\ge 2$ for all $t\in [0,1]$. 
            Then $\sigma:\bbT\to\bbT$ is given by $\sigma(z)=z^N$ for all $z\in \bbT$.
            \item Let $\bbD:=\{z\in\bbC\,:\,\vert z\vert\le 1\}$ and let $b:\bbD\to\bbD$ be a 
    finite     Blaschke product
    \begin{equation}
        \label{eq:finBlaschke}
        b(z)=C\prod_{i=1}^N \frac{z-a_k}{1-\overline{a_k}z},\quad \text{for all }z\in\bbD,
    \end{equation}
    where $C\in \bbT$ and $a_i\in \bbD$ are such that  $\vert a_i\vert<1$ for all $i=1,\dots,N$ and $b(1)=1$.  Let $\sigma$ be the restriction of $b$ to $\bbT$.
    It is known that $\sigma$ is an $N$-to-1 local homeomorphism on $\bbT$ (see, for example, \cite[Theorem 3.4.10]{GaMaRo_FBPTC18}). \cite[Lemma 4.2]{CMS2012} implies that there is a 
    function $\varphi:[0,1]\to [0,1]$ such that $\sigma(e^{2\pi ix})=e^{2\pi i\int_0^x \varphi(t)\,dt}$.
        \end{enumerate}
    \end{example}
    
\p\label{par:roots}    Let $\bbT_\infty$ be the projective limit induced by
    $(\bbT,\sigma^n)_{n\ge 0}$.      
    We define next a section $\fs:\bbT\to\bbT_\infty$ such that 
    $\lim_{n\to\infty}z_n=e^{2\pi i 0}=1$ for all $z\in \bbT$, where $\fs(z)=(z_n)_{n\ge 0}$.
    In order to accomplish this, we describe explicitly the preimage under $\sigma$ of a point $z\in\bbT$. Let $x\in[0,1)$ be  
    such that $z=e^{2\pi ix}$. If $x\notin [0,1)$ we replace it $x-k$, where $k\in \bbZ$ is such that $x-k\in [0,1)$. There are $N$-points $y_0,y_1,\dots,y_{N-1}\in [0,1)$ 
    such that $\int_0^{y_k}\varphi(t)\,dt=x+k$
    for all $k=0,\dots,N-1$. Therefore $\sigma(e^{2\pi i y_k})=e^{2\pi i (x+k)}=e^{2\pi i x}$ for all $k=0,\dots,N-1$. We say that
    $e^{2\pi i y_0},\dots, e^{2\pi i y_{N-1}}$ are the $N$-roots of $e^{2\pi i x}$ under $\sigma$. With a slight abuse of notation,
    we also say that $y_0,\dots,y_{N-1}$ are the $N$-roots of $x$ under $\sigma$. Moreover, even though the $N$-roots of $x$ under $\sigma$
    depend on $x$, we will not explicitly write this dependence in order to ease the notation. Note that $y_0<y_1<\dots<y_{N-1}$ for all $x\in [0,1)$
    since $\varphi(t)>0$ for all $t\in [0,1]$. We say that $y_0$ is the smallest root of 
    $x$ under $\sigma$ and that $y_{N-1}$ is the largest root of $x$ under $\sigma$. For 
    $x=0$, we will use the special notation $r_0,r_1,\dots,r_{N-1}$ for the roots of $0$ under $\sigma$:   $r_0=0$ and $\int_0^{r_k}\varphi(t)\,dt=k$ for all $k=1,\dots, N-1$.
    We write $r:=\min\{ r_1,1-r_{N-1}\}$.

    The following lemma captures a few basic properties of the roots under 
    $\sigma$ that  will help us define the section $\fs$.
\begin{lem}
    Let $x\in[0,1)$ and let $y_0,\dots,y_{N-1}$ be the roots of $x$ under $\sigma$. Then
    \begin{equation}
        \label{eq:roots_pr}
        0\le y_0\le \frac{x}{\und{\varphi}}<x\,\text{ and }\,0<1-y_{N-1}\le\frac{1-x}{\und{\varphi}}<1-x.
    \end{equation}
\end{lem}
\begin{proof}
    The first formula follows using an easy computation:
    \[
    x=\int_0^{y_0}\varphi(t)\,dt\ge \int_0^{y_0}\und{\varphi}\,dt=y_0\und{\varphi}.
    \]
    Thus $y_0<x/\und{\varphi}$. The fact that $x/\und{\varphi}<x$ follows since $\und{\varphi}>1$ by our assumption. The second formula follows
    from a similar computation:
    \[
    1-x=N-(x+(N-1))=\int_0^1 \varphi(t)\,dt-\int_0^{y_{N-1}} \varphi(t)\,dt=\int_{y_{N-1}}^1\varphi(t)\,dt\ge (1-y_{N-1})\und{\varphi}.
    \]
\end{proof}

\subp\label{def:standard_root} Equation \eqref{eq:roots_pr} implies that if $x\in [0,1/2)$ then $y_0\in [0,1/2)$ and 
 if $x\in [1/2,1)$ then $y_{N-1}\in [1/2,1)$. We call $y_0$ if $x\in [0,1/2)$ and 
 $y_{N-1}$ if $x\in [1/2,1)$ the \emph{standard root} of $x$ under $\sigma$ and denote it by $\tilde{x}$.
 
 \subp \label{def:parametrize} It  will be convenient for our purposes to parametrize the unit circle starting at the point (-1,0)
 instead of (1,0). That is, we will view $z=e^{2\pi i x}$ with $x\in[-1/2,1/2)$ instead of $x\in 
 [0,1)$.   Using this perspective, if $x\in [-1/2,1/2)$
    the standard root $\tilde{x}$ of $x$ under $\sigma$ satisfies 
    $\int_0^{\tilde{x}}\varphi(t)\,dt=x$.

\p\label{def:section}    We can now define the promised section $\fs:\bbT\to\bbT_\infty$.  Let $z=e^{2\pi ix}$ with $x\in [-1/2,1/2)$. Set $x_0=x$ and define $x_1:=\widetilde{x_0}$,  the standard root of $x_0$ under $\sigma$. We proceed then  
    inductively: $x_{n+1}=\widetilde{x_n}$,   the standard 
    root of $x_n$ under $\sigma$,  for all $n\ge 1$. Hence, by definition, $\sigma(e^{2\pi i x_{n+1}})=e^{2\pi i x_n}$ for all $n\ge 0$.   
    \begin{lem}\label{lem:section_varphi}
        Let $x\in[-1/2,1/2)$ and let $(x_n)_{n\ge 0}$ the sequence defined above. Then $(e^{2\pi i x_n})\in \bbT_\infty$ and, thus, the function $\fs:\bbT\to \bbT_\infty$
        defined via  $\fs(e^{2\pi i x}):=(e^{2\pi i x_n})_{n\ge 0}$ for all $x\in[-1/2,1/2)$ is a section. Moreover,  $\lim_{n
        \to\infty}x_n=0$ and, thus,
        $\lim_{n\to \infty}e^{2\pi i x_n}=1$ for all $x\in [-1/2,1/2)$. 
    \end{lem}
    \begin{proof}
    The fact that $(e^{2\pi i x_n})_{n\ge 0}\in \bbT_\infty$ follows from the construction of the sequence: $\sigma(e^{2\pi i x_{n+1}})=e^{2\pi i x_n}$ for all $n\ge 0$.
    If $x\in [-1/2,1/2)$, \eqref{eq:roots_pr} implies that $\vert x_n\vert \le \vert x\vert/\und{\varphi}^n$ for all $n\ge 1$. Therefore $\lim_{n\to \infty}x_n=0$ since $\und{\varphi}>1$. 
    \end{proof}

\subp        We will use tacitly the following notation  in the remainder of this section: for each $x\in [-1/2,1/2)$
        we let $(x_n)_{n\ge 0}$ be the sequence defined in Lemma \ref{lem:section_varphi} and $z_n=e^{2\pi i x_n}$ for all $n\ge 0$. Therefore 
        $\fs(z)=(z_n)_{n\ge 0}$. Moreover, we don't write explicitly the dependence of the sequence $(x_n)$ on $x$ in 
        order to keep the notation  short.

\p    Let $\psi:\bbT\to\bbR$ be the full potential defined via $\psi(e^{2\pi i x})=1/\varphi(x)$ for all $x\in [0,1]$, and let  $\cL_\psi$ be its transfer operator. With respect to the induced metric, the Hausdorff dimension of $\bbT$ is 1 and the Hausdorff measure $\mu:=\mu^1$     is invariant under the transpose of $\cL_\psi$. Consequently, it is a quasi-invariant measure for $G(\bbT,\sigma)$ (\cite[Section 4.5]{IoKu_IUMJ13}). Moreover, if  $\varphi(x)>1$ for all $x\in [0,1]$ and is continuously  differentiable, then $\sigma$ is expansive and exact. Thus, $\mu$ is the unique measure that is a fixed point of the transpose 
of $\cL_\psi$ (\cite[Theorem 2.8]{KuRe_PAMS06}, \cite[Proposition 4.2]{Ren_OAMP01}, \cite[Proposition 3.5]{IoKu_IUMJ13}).  The Hausdorff measure $\mu$ will be the only measure that we consider in this example.

As discussed in \ref{par:unital_psi}, we will assume that $\cL_\psi$ is unital to keep the notation manageable. The results below remain true in the non-unital
    case if one defines $\fu(x)=\mf(x)\sqrt{\psi(x)}/\cL_\psi(1)(\sigma(x))$ and uses the invariance of $\mu$ as 
    described \ref{par:mu_unital_psi}.
    
    Fix now  a $\psi$-filter $\mf\in C(\bbT)$ and set, as usual, $\fu(z)=\mf(z)\sqrt{\psi(z)}$ 
    and $D(z)=\vert \fu(z)\vert^2=\vert \mf(z)\vert^2\psi(z)$ for all $z\in \bbT$. Therefore, $\fu(e^{2\pi i x})=\mf(e^{2\pi i x})/\sqrt{\varphi(x)}$ for 
    all $x\in [0,1]$. Since we extended $\varphi$ to $\bbR$ by periodicity, $\fu(e^{2\pi i x})=\mf(e^{2\pi i x})/\sqrt{\varphi(x)}$ for all $x\in \bbR$. 
    
\p \label{par:Mallat}    Consider the $p$-system of   measures $\{\nu_z^D\}_{z\in \bbT}$ on 
    $\bbT_\infty$ defined by $D$ as in \eqref{eq:Markov1}, where, recall, $p:\bbT_\infty\to
    \bbT$ is the projection onto the first component. 
    Recall also that $r_0=0,r_1,\dots,r_{N-1}\in [0,1)$ are the roots of $0$ under $\sigma$ and $r=\min\{r_1,1-r_{N-1}\}\in (0,1)$.

    The next lemma is key to our analysis. It provides conditions that guarantee that the section $\fs$ defined before is
    positively supported and, thus, a scaling function $\phi$ exists. This lemma is the first step towards our generalization 
    of \cite[Theorem 2]{Mall_TAMS89}. We note that even though our set-up is more general than the one that Mallat 
    considered, the hypotheses of our lemma are virtually identical to Mallat's hypotheses. 
    \begin{lem}\label{lem:pos_sup}
        With the notation as above, assume that $\varphi(x)> 1$ for all $x\in[0,1]$. Assume that
        $\fu:\bbT\to\bbC$, $\fu(e^{2\pi i x}):=\mf(e^{2\pi i x})/\sqrt{\varphi(x)}$ for $x\in [-1/2,1/2)$, 
        satisfies the following hypotheses:
        \begin{align}
         \left\vert \fu(1)\right\vert&=1,\label{eq:Mall1} \\
          \vert \fu(e^{2\pi i x})\vert& \ne  0 \,\text{ if }\vert x\vert<r,\label{eq:Mall2}\\
          \ln D(e^{2\pi i x})=\ln \vert\fu(e^{2\pi i x})\vert^2&=O(x)\,\text{ for $x$ near 0.}\label{eq:Mall3}
        \end{align}                        
        Then the section $\fs:\bbT\to \bbC$ defined above is positively supported. That is, 
        the scaling function $\phi:\bbT\to \bbC$ given via
    \[
    \phi(z)=\prod_{n=1}^\infty \fu(z_n)=\prod_{n=1}^\infty\frac1{\sqrt{\varphi(x_n)}}\mf(e^{2\pi i x_n})
    \]
    for all $z=e^{2\pi i x}\in \bbT$ with $x\in [-1/2,1/2)$ converges for all $z\in \bbT$, where $(x_n)_{n\ge 0}$ is the sequence defined in Lemma \ref{lem:section_varphi}.
    \end{lem}
    \begin{remark}\label{rem:scaling_prop}
    \begin{enumerate}        
        \item If the hypothesis \eqref{eq:Mall1} is satisfied, then  $D(e^{2\pi ir_k})=0$ for all $k=1,\dots, N-1$ 
        since  $\sum_{k=0}^{N-1}D(e^{2\pi i r_k})=1$. 
        Thus $\mf (e^{2\pi i r_k})=0$ for all $k=1,\dots,N-1$. The converse is also true:
        if $\mf(e^{2\pi i r_k})=0$ for all $k=1,\dots, N-1$, then $\fu$ satisfies \eqref{eq:Mall1}. If the
        hypothesis \ref{eq:Mall2} is satisfied as well, then $e^{2\pi i r_1},\dots,e^{2\pi ir_{N-1}}$ are the only points where $\mf$ vanishes.
        \item The scaling function $\phi$ satisfies the relation
        \begin{equation}
            \label{eq:phi_rec}
            \phi(z)=\left(\prod_{i=1}^n \fu(z_i)\right)\phi(z_n)\,\text{ for all }n\ge 1\text{ and }z\in\bbT.
        \end{equation} 
    \end{enumerate}
    \end{remark}
    \begin{proof}[Proof of Lemma \ref{lem:pos_sup}]
        We continue to use the notation $\und{\varphi}:=\min_{x\in [0,1]}\varphi(x)>1$. 
        Note first that $\vert \phi(e^{2\pi i 0})\vert^2=\vert\phi(1)\vert^2=1$ by \eqref{eq:Mall1}
        since $\fs(1)=(\und{1})$ by construction. 
        Second, if $x\in [-1/2,1/2)$ the product $\prod_{n=1}^\infty 
        \vert\fu(z_n)\vert^2=\prod_{n=1}^\infty D(z_n)$ converges since $0\le D(z)\le  1$ for all $z\in \bbT$.
        Moreover, using \eqref{eq:roots_pr} and $\und{\varphi}>1$,  $\vert x_n\vert<r$ for all $n\ge 1$ and, thus, $D(x_n)> 0$ for all $n\ge 1$ by \eqref{eq:Mall2}. Since  
        \[
        \vert\phi(e^{2\pi i x})\vert^2=e^{\sum_{n\ge 1}\ln  D(e^{2\pi i x_n})},
        \]
        using \eqref{eq:Mall3} and the fact that $\lim_{n\to \infty} x_n=0$,  we obtain the formulas
        $$\lim_{x\to 0}\vert \phi(e^{2\pi i x})\vert^2=1=\vert \phi(1)\vert^2.$$ Let $0<\varepsilon<1$ and let $\delta>0$ be such that 
        $\vert \phi(x)\vert^2>1-\varepsilon$ for all $\vert x\vert<\delta$. Let $x\in [-1/2,1/2)$. Since $\lim_{n\to \infty}x_n=0$
         there is $M\ge 1$ such that $\vert x_n\vert<\delta$  for all $n\ge M$.  Therefore, using \eqref{eq:phi_rec}, we have
        \[
        \vert\phi(e^{2\pi ix})\vert^2=\left(\prod_{i=1}^{M}D(x_i)\right)\vert \phi(e^{2\pi ix_{M}})\vert^2>(1-\varepsilon)\left(\prod_{i=1}^{M}D(x_i)\right)>0
        \]
        for all $x\in[-1/2,1/2)$. The conclusion follows.
    \end{proof}
\p     Next we lift $\phi$   to  a "non-constant"  scale, scaling function $\tilde{\phi}$ on $\bbR$:
    \begin{equation}
        \label{eq:variablescale}
        \tilde{\phi}\left( \int_0^x\varphi(t)\,dt)\right)=\frac1{\sqrt{\varphi(x)}}\mf(e^{2\pi i x})\tilde{\phi}(x)
    \end{equation}
     or, equivalently,
    \[
    \sqrt{\varphi(x)} \tilde{\phi}\left( \int_0^x\varphi(t)\,dt)\right)=\mf(e^{2\pi i x})\tilde{\phi}(x)
    \]
    for all $x\in \bbR$. To do this, we extend the construction of the sequence $(x_n)$ from $x\in [-1/2,1/2)$ to all $x\in \bbR$. Specifically,  for $x\in \bbR$, let $\{x_n\}_{n\ge 0}$ be the sequence defined via: $x_0=x$, and,  $x_{n+1}$ is 
    such that $\int_0^{x_{n+1}}\varphi(t)\,dt=x_{n}$ for $n\ge 0$. Equation $\eqref{eq:roots_pr}$ implies that $\lim_{n\to \infty}x_n=0$. Moreover, 
    $\sigma(e^{2\pi ix_{n+1}})=e^{2\pi i x_n}$ for all $n\ge 0$.  The sequence $\{x_n\}_{n\ge 0}$ depends on $x$; however, as before, we don't explicitly write 
    the dependence of the sequence $\{x_n\}$ on $x$. Again, this should not lead to any confusion. Define
    \begin{equation}
        \label{eq:phi_tilde}
        \tilde{\phi}(x)=\prod_{n\ge 1}\fu(e^{2\pi i x_n})=\prod_{n\ge 1}\frac{\mf(e^{2\pi i x_n})}{\sqrt{\varphi(x_n)}}
    \end{equation}
    for all $x\in \bbR$. Hence $\vert \tilde{\phi}(x)\vert^2=\prod_{n\ge 1}D(e^{2\pi ix_n})$,  $\tilde{\phi}(x)=\phi(e^{2\pi ix})$ for all $x\in [-1/2,1/2)$ and
    $\tilde{\phi}(x)=\left(\prod_{i=1}^n \fu(e^{2\pi ix_i})\right)\tilde{\phi}(x_n)$ for all $n\ge 1$. Moreover, if \eqref{eq:Mall1} is satisfied, then
    $\vert \tilde{\phi}(0)\vert=1$. 
    \begin{example}
        Assume that $\varphi:[0,1]\to\bbR$ is the constant function $\varphi(t)=2$ for all $x\in [0,1]$. Then the local homeomorphism $\sigma:\bbT\to
        \bbT$ is given by $\sigma(z)=z^2$. Moreover, $\psi(z)=1/2$ and $\cL_\psi$ is the standard transfer operator. Then a $\psi$-filter is a classical 
        filter and, if it satisfies the hypotheses of Lemma \ref{lem:pos_sup}, it is a QMF filter in the sense of Mallat. If 
        $x\in \bbR$ then $x_n=x/2^n$ for all $n\ge 1$. 
        The scaling function $\tilde{\phi}$ is  given by the same equation as in (\cite[Theorem 2]{Mall_TAMS89}):
        \[
        \tilde{\phi}(x)=\prod_{n\ge 1}\fu(e^{2\pi i x/2^n})=\prod_{n\ge 1}\frac{\mf(e^{2\pi i x/2^n})}{\sqrt2}
        \]
        and it satisfies the $2$-scaling property: $\tilde{\phi}(2x)=m(e^{2\pi i x})/\sqrt2\tilde{\phi}(x)$ for all $x\in \bbR$. This example generalizes
        immediately to the function $\varphi(t):=N\ge 2$. In this case, the scaling function has the well-known form
        \[
        \tilde{\phi}(x)=\prod_{n\ge 1}\fu(e^{2\pi i x/N^n})=\prod_{n\ge 1}\frac{\mf(e^{2\pi i x/N^n})}{\sqrt{N}}
        \]
        and it satisfies the $N$-scaling property: $\tilde{\phi}(Nx)=m(e^{2\pi i x})/\sqrt{N}\tilde{\phi}(x)$ for all $x\in \bbR$.
    \end{example}
        The following theorem is a generalization of Mallat's Theorem 2 of \cite{Mall_TAMS89} to our more general set-up.
    \begin{thm}\label{thm:Mallat}
        Assume the hypotheses of Lemma \ref{lem:pos_sup}. Then the function $\tilde{\phi}: \bbR\to \bbR$ defined in \eqref{eq:phi_tilde} is a unit vector in
        $L^2(\bbR)$ and
        \begin{equation}
            \sum_{j\in\bbZ}\vert\tilde{\phi}(x+j)\vert^2=1\,\text{ for all }\,x\in [-1/2,1/2).
        \end{equation}
    \end{thm}
    The proof will be completed using a series of lemmas. 
    \begin{lem}
        \label{lem:Mall1}
        For $m\ge 1$, let 
        \[
        \tilde{\phi}_m(x)=\begin{cases}
            \prod_{i=1}^m\fu(x_i) & \text{ if }\vert x\vert \le \frac{N^m}2\\
            0 & \text{ otherwise}
        \end{cases}        
        \]
        and let 
        $\displaystyle I_m:=\int_{-\frac{N^m}{2}}^{\frac{N^m}{2}} \vert \tilde{\phi}_m(x)\vert^2\,dx$. Then $I_m=1$ for all $m\ge 1$.
    \end{lem}    
    \begin{proof}
        Let $m\ge 1$. For $x\in [-1/2,1/2)$, let $$J_{m,x}=\{j\in\bbZ\,:\,x+j\in [-N^m/2,N^m/2)\}.$$ The cardinality of $J_{m,x}$ equals $N^m$ for all 
        $x\in [-1/2,1/2)$. For each $j\in J_{m,x}$, let $y_j\in [-1/2,1/2)$ be the $m$-th term in the sequence $(x_n)_{n\ge 0}$ corresponding to
        $x+j$. Then $\sigma^m(e^{2\pi i y_j})=e^{2\pi i x}$ for all $j\in J_{m,x}$. Therefore $\{e^{2\pi i y_j}\}_{j\in J_{m,x}}$ is the set of  pre-image points
        of $e^{2\pi i x}$ under $\sigma^m$. (Once again, while $y_j$ depends on $x$ we don't write explicitly the dependence in the notation.) 
        Therefore       
        \[
            \int_{-\frac{N^m}{2}}^{\frac{N^m}{2}} \vert \tilde{\phi}_m(x)\vert^2\,dx
            =\int_\bbT \sum_{\sigma^m(w)=z} D_m(w)\,d\mu(z)=1
        \]
        since $D_m$ is the normalized potential of $\cL_D^m$.
    \end{proof}
    \begin{lem}\label{lem:phi_not_at_0}
        With the notation as above,  $\vert \tilde{\phi}(0)\vert^2=1$ and $\vert \tilde{\phi}(j)\vert^2=0$ for all $j\in \bbZ\setminus\{0\}$.
    \end{lem}
    \begin{proof}
        The fact that $\vert\phi(0)\vert=1$ was proved in Lemma \ref{lem:pos_sup}. Let $j\in \bbZ\setminus \{0\}$. Assume first that  $j\in [-N/2,N/2)$ 
        (using the notation from the proof of Lemma \ref{lem:Mall1}, $j\in J_{1,0}$). Let $y_j\in [-1/2,1/2)$ such that $\int_0^{y_j}\varphi(t)\,dt=j$. Therefore 
        $y_j$ equals the term $x_1$ of the sequence $\{x_n\}_{n\ge 0}$ used to define $\tilde{\phi}(j)$. Since 
        $\sigma(e^{2\pi i y_j})=e^{2\pi i j}=1$, it follows that $D(e^{2\pi i y_j})=0$ (see the second part of Remark \ref{rem:scaling_prop}). Hence
        $\tilde{\phi}(j)=0$. 
        
         Assume now that $j\in [-N/2+kN,N/2+kN)$ for some $k\in \bbZ\setminus\{0\}$. Then $j-kN\in [-N/2,N/2)$. If $y_j\in[-1/2,1/2)$ is such that 
         $\int_0^{y_j}\varphi(t)\,dt =j-N$, then the term $x_1$ of the sequence $(x_n)_{n\ge 0}$ used to define $\tilde{\phi}(j)$ equals $y_j+k$ since
         $\int_0^1 \varphi(t)\,dt=N$. Therefore $D(e^{2\pi i(y_{j}+k)})=D(e^{2\pi iy_j})=0$ and, thus, $\tilde{\phi}(j)=0$.
    \end{proof}
    \begin{proof}[Proof of Theorem \ref{thm:Mallat}]
        The fact that $\tilde{\phi}\in L^2(\bbR)$ and $\Vert \tilde{\phi}\Vert_2=1$
        follows from  Lemma \ref{lem:Mall1}, Proposition \ref{prop:scalingfunction} and Fatou's lemma.

       Let $x\in[-1/2,1/2)$. For $m\ge 1$, we use the notation from the proof of Lemma \ref{lem:Mall1}:
       $J_{m,x}=\{j\in\bbZ\,:\,x+j\in [-N^m/2,N^m/2)\}$ and $y_j$ is the $m$-term of the sequence $\{x_n\}_{n\ge 0}$ that defines $x+j$ for all $j\in J_{m,x}$.
       Then
       \[
       \sum_{k=-N^m/2}^{N^m/2}\vert \tilde{\phi}(x+k)\vert^2=\sum_{j\in J_{m,x}} D_m(y_j)\vert \tilde{\phi}(y_j)\vert^2. 
       \]
       One can prove inductively that $\vert y_j\vert\le (N/2)/\underline{\varphi}^m$ for all $j\in J_{m,x}$, where, as before, $\und{\varphi}=\min_{t\in [0,1]}\varphi(t)>1$. Recall from the proof of Lemma \ref{lem:pos_sup} that condition \eqref{eq:Mall3} implies that $\tilde{\phi}$ is continuous
       at $0$ and that $\tilde{\phi}(0)=1$ by \eqref{eq:Mall1}. Let $\varepsilon>0$ and let $\delta>0$ be such that 
       $\vert \vert\tilde{\phi}(t)\vert^2-1\vert<\varepsilon$ for all $t\in (-\delta,\delta)$. Then there is $M\ge 1$ such that $\vert y_j\vert<\delta$
       for all $m\ge M$ (recall that $y_j$ depends both on $x$ and on $m$ but we suppressed the dependence in the notation).  Lemma 
       \ref{lem:phi_not_at_0} implies
       \begin{multline*}
           \left\vert \sum_{k=-N^m/2}^{N^m/2}\vert \tilde{\phi}(x+k)\vert^2-1\right\vert=\left\vert \sum_{j\in J_{m,x}} D_m(y_j)\vert \tilde{\phi}(y_j)\vert^2-1\right\vert\\
           =\left\vert \sum_{j\in J_{m,x}} D_m(y_j)\bigl(\vert \tilde{\phi}(y_j)\vert^2-1\bigr)\right\vert
           \le \sum_{j\in J_{m,x}} D_m(y_j)\bigl\vert \vert \tilde{\phi}(y_j)\vert^2-1\bigr\vert\\
           <\sum_{j\in J_{m,x}} D_m(y_j)\varepsilon=\varepsilon
       \end{multline*}
       for all $m\ge M$.
       Therefore $\sum_{j\in\bbZ}\vert \tilde{\phi}(x+j)\vert^2=1$ for all $x\in [-1/2,1/2)$.
    \end{proof}

\p \label{par:Mallatsthm}     Next we obtain the full generalization of Mallat's theorem in our context and we construct
     multiresolution analyses with scaling 
     functions of non-constant scale. For this, we consider   a unitary representation $\hat{L}=(\mu,\bbT*\cH,\hat{L})$ of the Deaconu-Renault groupoid $G(\bbT,
     \sigma)$, where, recall, $\mu$ is the Hausdorff measure on $\bbT$. We continue to assume that $\mf$ is a $\psi$-filter that satisfies the hypotheses of 
     Lemma \ref{lem:pos_sup}. For simplicity and comparison with the previous results in the 
     literature, assume that $\hat{L}=\hat{\iota}$ is 
     the standard fundamental representation on $\mu$: $\hat{\iota}=(\mu,\bbT\ast \bbC,\hat{\iota}) $. Therefore
     the integrated form $\fI$ of $\hat{\iota}$ acts on $L^2(\bbT,\mu)$ and the isometry $\tS_\fu=\fI(S_\fu)$ 
     is given by $\tS_{\fu}\xi(z)=\mf(z)\xi(\sigma(z))$. That is
     \[
     \tS_{\fu}\xi(e^{2\pi i x})=\mf(e^{2\pi i x})\xi (e^{2\pi i \int_0^x \varphi(t)\,dt})\,\text{ for all }x\in [-1/2,1/2).
     \]
     We endow $G_\infty(\bbT,\sigma)$ with the Haar system $\nu^\psi*\lambda$ defined as in \eqref{def:nustarlambda}, where $\nu^\psi$ is
     the $p$-system of measures on $X_\infty$ defined by the full unital potential $\psi$ via \eqref{eq:Markov1}. 
     The unital potential $D=\vert \fu\vert^2$ defines via the same formula a $p$-system of measures $\nu^D$ on $X_\infty$ that is not 
     full. Recall that 
     $\nu^D_u$ is absolutely continuous with respect to $\nu^\psi_u$ since $D(z)=\vert \mf(z)\vert^2 \psi(z)$ for all $z\in \bbT$. Therefore, we can use the techniques of \ref{par:topocorrespnd} 
     and Section \ref{sec:GinfInd} to
     induce $\hat{\iota}$ to a representation $\hat{\iota}_\infty=(\mu_\infty,\bbT_\infty\ast \bbC,\hat{\iota}_\infty)$ of $G_\infty(\bbT,\sigma)$. Moreover
     $C^*(G_\infty(\bbT,\sigma),\nu^\psi*\lambda)$ acts on the completion of $C_c(X_\infty*G)$ under the $C^*(\bbT)$-inner product defined in \eqref{eq:CcGinnerprod} via adjointable operators as  in \eqref{eq:leftYGYZact}.
     Equations \eqref{eq:Ind_infty_int} and \eqref{eq:equiv_inducedrep} imply that  
     the unitary $\tU_\fu:=\fI_\infty(U_\fu)$, where $\fI_\infty$ is the integrated form of $\hat{\iota}$,  acts on $L^2(\bbT_\infty,\mu_\infty)$
     via 
     \[
     \tU_\fu(\xi)(\und{z})=\mf(p(\und{z}))\xi(\sigma_\infty(\und{z}))\,\text{ for all }\,\und{z}\in \bbT_\infty\,\text{ and }\,\xi\in L^2(\bbT_\infty,\mu_\infty).
     \]  
     \begin{thm}\label{thm:Mallat2}
         Using the above notation,  assume the hypotheses of Lemma \ref{lem:pos_sup}. Let $\phi:\bbT\to \bbT_\infty$ be the corresponding 
         scaling function \eqref{eq:scalefunction} and $\tilde{\phi}:\bbR\to\bbT_\infty$ its lift
         defined via \eqref{eq:phi_tilde}.  Let $\hat{\iota}=(\mu,\bbT\ast \bbC,\hat{\iota})$ be the standard
         fundamental representation  on $\mu$ of the Deaconu-Renault groupoid $G(\bbT,\sigma)$ and 
         let $\hat{\iota}_\infty=(\mu_\infty,\bbT_\infty\ast \bbC,\hat{\iota}_\infty)$ be the induced 
         representation of the pullback groupoid $G_\infty(\bbT,\sigma)$. Then there is an isometric isomorphism 
         $R:L^2(\bbT_\infty,\mu_\infty)\to L^2(\bbR)$ that intertwines the unitary $\tU_\fu$ and the unitary $\fU_\fu$ on $L^2(\bbR)$ defined
         via the formula
         \begin{equation}
             \label{eq:V_Mallat}         
         \fU_\fu(\xi)(x)=\sqrt{\varphi(x)}\xi\left(\int_0^x\varphi(t)\,dt\right)
         \end{equation}for all $\xi\in L^2(\bbR)$. 
     \end{thm}
     \begin{remark}
                 A proof similar to  the proof of \cite[Theorem 6.4]{Bag_co_JFA10} shows that the pair, $(L^2(\bbT_\infty,\mu_\infty),\tU_\fu)$, is the inductive limit of the inductive system $(H_n,S_n)$, where $H_n=L^2(\bbT,\mu)$ and $S_n=\tS_\fu$ for all $n\ge 0$ (see also \cite{LarRae_CM06}). 
     \end{remark}
     \begin{proof}
        We prove the existence of the isometric isometry $R$. The remainder of the statements follow immediately. To prove
        the existence of $R$ we define a sequence of isometries $R_n:L^2(\bbT,\mu_n)\to L^2(\bbR)$ such that $R_{n+1}\tS_\fu=R_n$
         for all $n\ge 0$. The conclusion follows from the universal properties of inductive
         limit of Hilbert spaces.
         
        Recall from the definition of $\tilde{\phi}$, \eqref{eq:phi_tilde}, that for each $x\in \bbR$
        we define the sequence $\{x_n\}_{n\ge 0}$ where $x_0=x$ and $\int_0^{x_{n+1}}\varphi(t)\,dt=x_{n}$
        for all $n\ge 0$. Define $R_n:L^2(\bbT,\mu)\to L^2(\bbR)$ via
         \[
         R_n(f)(x)=\frac1{\sqrt{\varphi(x_n)\cdots \varphi(x_1)}}f(e^{2\pi i x_n})\tilde{\phi}(x_n)
         \]
         for all $x\in \bbR$, $f\in L^2(\bbT,\mu)$, and $n\ge 0$. For $n=0$ we assume that the 
         denominator of the fraction is $1$. It is relatively easy to check that $R_{n+1}\tS_\fu=R_n$:
         \begin{multline*}
             (R_{n+1}\tS_\fu)(f)(x)=\frac1{\sqrt{\varphi(x_{n+1})\cdots \varphi(x_1)}}\tS_\fu(f)(e^{2\pi i x_{n+1}})\tilde{\phi}(x_{n+1})\\
             =\frac1{\sqrt{\varphi(x_n)\cdots \varphi(x_1)}}f(e^{2\pi i x_n})\frac{m(e^{2\pi i x_{n+1}})}{\sqrt{\varphi(x_{n+1})}}\tilde{\phi}(x_{n+1})\\
             =\frac1{\sqrt{\varphi(x_n)\cdots \varphi(x_1)}}f(e^{2\pi i x_n})\tilde{\phi}(x_n)=R_n(f)(x).
         \end{multline*}
         We prove next  that $R_n$ is an isometry for all $n\ge 1$. Let $f\in L^2(\bbT,\mu)$. Then
         \[
         \Vert R_n(f)\Vert_2^2=\int_{-\infty}^\infty \frac{1}{\varphi(x_n)\cdots \varphi(x_1)}\vert f(e^{2\pi i x_n})\vert^2 \vert \tilde{\phi}(x_n)\vert^2\,dx.
         \]
        We claim that 
        \[
        \Vert R_n(f)\Vert_2^2=\int_{-1/2}^{1/2} \vert f(e^{2\pi ix})\vert^2 \bigl( \sum_{j\in \bbZ}\vert \tilde{\phi}(x+j)\vert^2\bigr)\,d\mu(x).
        \]
         We prove this claim first for $n=1$ for simplicity. We use the notation from the proof of Lemma \ref{lem:Mall1}. Let $x\in 
         [-1/2,1/2)$ and let $\{y_j\}_{j\in J_{1,x}}\in [-1/2,1/2)$  be the roots of $x$ under $\sigma$: $\int_0^{y_j}\varphi(t)\,dt=x+j$ for all 
         $j\in J_{1,x}$. Then, since 
         $\mu$ is invariant for the transpose of the transfer operator $\cL_\psi$, we have
         \begin{multline*}
             \int_{-N/2}^{N/2} \frac1{\varphi(x_1)}\vert f(e^{2\pi i x_1})\vert^2\vert \tilde{\phi}
             (x_1)\vert^2\,dx=\int_{-1/2}^{1/2} \sum_{j\in J_{1,x}}\frac{1}{\varphi(y_{j})}\vert 
             f(e^{2\pi i y_{j}})\vert^2 \vert\phi(y_{j})\vert^2\,d\mu(x)\\
             =\int_{-1/2}^{1/2} \cL_\psi(\vert f\vert^2\vert \phi\vert^2)(x)\,d\mu(x)=\int_{-1/2}^{1
             /2}\vert f(e^{2\pi i x})\vert^2\vert \phi(x)\vert^2\,d\mu(x).
         \end{multline*}
         Let $k\in\bbZ$ such that $k\ne 0$  and let $x\in [-N/2+kN,N/2+kN)$. Then $x-kN\in [-1/2,1/2)$ and we let $\{y_j\}_{j\in J_{1,x-kN}}$  be the 
         roots of $x-kN$ under
         $\sigma$. It follows that if $x_1$ is such that $\int_0^{x_1}\varphi(t)\,dt=x+j$, $j\in J_{1,x-kN}$, then $x_1=k+y_{j}$. Therefore
         \begin{multline*}
             \int_{-N/2+kN}^{N/2+kN} \frac1{\varphi(x_1)}\vert f(e^{2\pi i x_1})\vert^2\vert \tilde{\phi}(x_1)\vert^2\,dx\\
             =\int_{-1/2}^{1/2} \sum_{j\in J_{1,x-kN}}\frac{1}{\varphi(y_{j})}\vert f(e^{2\pi i y_{j}})\vert^2 \vert\tilde{\phi}(k+y_{j})\vert^2\,d\mu(x)\\
             =\int_{-1/2}^{1/2} \vert f(e^{2\pi i x})\vert^2\vert \tilde{\phi}(k+x)\vert^2\,d\mu(x).
        \end{multline*}
         The claim for $n=1$ follows. For  $n\ge 2$, one can repeat the 
        same arguments using the $n$-roots of $x$, $\{y_j\}_{j\in J_{n,x}}$, since $\mu$ is invariant 
        for the transpose of the transfer  operator 
        $\cL_\psi^n$ for all $n\ge 2$. Thus 
        \begin{multline*}
        \int_{-\infty}^\infty \frac{1}{\varphi(x_n)\cdots \varphi(x_1)}\vert f(e^{2\pi i x_n})\vert^2 \vert \tilde{\phi}(x_n)\vert^2\,dx\\
        =\int_{-1/2}^{1/2} \vert f(e^{2\pi ix})\vert^2 \bigl( \sum_{j\in \bbZ}\vert \tilde{\phi}(x+j)\vert^2\bigr)\,d\mu(x).
        \end{multline*}
        Since $\sum_{j\in \bbZ}\vert \tilde{\phi}(x+j)\vert^2=1$ for all $x\in [-1/2,1/2)$ by Theorem \ref{thm:Mallat},
         $R_n$ is an isometry for all $n\ge 0$ and the conclusion follows.
     \end{proof}

\p Combining the above results with the proto-multiresolution analysis, one obtains 
multiresolution analyses in our set-up.
\begin{thm}
    \label{thm:GMRA2}
    Assume that $\varphi:[0,1]\to\bbR$ satisfies the hypotheses described at the beginning
    of the section and let $\sigma:\bbT\to\bbT$ be defined via $\sigma(e^{2\pi i x})=e^{2\pi i\int_0^x\varphi(t)}\,dt$. Assume that $\mu$ is the Hausdorff measure on $\bbT$ that
    is invariant for the transfer operator $\cL_\psi$ with potential $\psi(e^{2\pi i x})=1/\varphi(x)$. Fix a $\psi$-filter $\mf$ and assume that the function $\fu\in C(X)$, 
    $\fu(x)=\sqrt{\psi(x)}\mf(x)$ satisfies the hypotheses of Lemma \ref{lem:pos_sup}. Let $\hat{\iota}=(\mu,\bbT\times\bbC,\hat{\iota})$ be the standard fundamental representation on $\mu$ of $G(\bbT,\sigma)$ and let 
    $\hat{\iota}_\infty=(\mu_\infty,\bbT_\infty,\hat{\iota}_\infty)$ be the induced representation
    of $G_\infty(\bbT,\sigma)$ defined by $D=\vert \fu\vert^2$. Let $(\{\bbE_n\},U_\fu)$ be the proto-resolution analysis in $C_c(G_\infty(\bbT,\sigma),\nu^\psi*\lambda)$ 
    defined in \ref{par:PMRA}  and let $\fU_\fu$ be the unitary in $L^2(\bbR)$ defined in \eqref{eq:V_Mallat}. Set $\cH_n=i_\infty(\bbE_n)L^2(\bbT_\infty,\mu_\infty)$ and  $\cV_n:=R(\cH_n)\subset L^2(\bbR)$. Then
    \begin{enumerate}
    \item $\cV_n\subseteq \cV_m$, when $n\leq m$.
    \item $\vee_{n=-\infty}^{\infty}\cV_n=L^2(\bbR)$ and
      $\bigwedge \cV_n=\{0\}$
    \item $\fU_\fu(\cV_{n+1})=\cV_n$ for all $n\in\bbZ$.
     \item $ R(i_\infty\circ \Phi(\cdot))$ is reduced by
        $\cV_0\setminus\cV_{-1}$ and the restriction is a
        representation of $C_c(G)$.
    \end{enumerate}
    Thus $\{\cV_n\}_{n\in\bbZ}$ is a  multiresolution
    analyses in $L^2(\bbR)$ with possibly non-constant scale.
\end{thm}
\begin{example}\label{ex:MallatNew}
\begin{enumerate}
\item If $\varphi(t)=2$ for all $t\in [0,1]$, then Theorems \ref{thm:Mallat} and \ref{thm:Mallat2} 
recover Mallat's theorem \cite[Theorem 2]{Mall_TAMS89}. Our techniques are very different from his. 
He uses complicated Fourier analysis techniques while we use only the quasi-invariance of the Hausdorff measure.
\item If $\varphi(t)=N\ge 2$ for all $t\in [0,1]$, our results also capture known results in the 
literature. In particular, even though the techniques of our proofs are different, the big steps in 
the proof of Theorem \ref{thm:Mallat2} are surprisingly similar to those of \cite{LarRae_CM06}.
\item If $\sigma$ is the restriction of a finite Blaschke product \eqref{eq:finBlaschke}, then under 
the assumption that  $\displaystyle \sum_{i=1}^N \frac{1-\vert a_j\vert}{1+\vert a_j\vert}>1$, $\sigma$
is an expansive local homeomorphism (\cite[Corollary on p. 344]{Mar_BLMS83}). Therefore the Hausdorff measure
$\mu$ is the unique invariant measure for the dual of the transfer operator $\cL_\psi$. 
One can easily find $\psi$-filters $\mf$ that satisfy the hypotheses of Lemma \ref{lem:pos_sup}. Indeed, let $m:\bbT\to\bbC$ 
be any continuous function such that $m(1)\ne 0$ and $m(e^{2\pi i r_j})=0$ for all $j=1,\dots,N-1$, where $r_j$, $j=0,\dots N-1$, are the units of $0$
with $r_0=0$ as defined earlier in the section. Assume that $m(z)\ne 0$ for all $z\notin\{e^{2\pi i r_0},\dots e^{2\pi i r_{N-1}}\}$ and
assume that $m$ is smooth near $0$. Set $\xi(z):=\sum_{\sigma(w)=z}\psi(w)\vert m(w)\vert^2$. Then one can easily check that the function defined via
\[
\mf(z)=\frac{m(z)}{\sqrt{\xi(\sigma(z)))}}\,\text{ for all }\,z\in\bbT
\]
is a $\psi$-filter that satisfies the required hypotheses.
\end{enumerate}
\end{example}

\p \textbf{An example of a generalized MRA without a scaling function}
Lemma \ref{lem:pos_sup} was key to defining the scaling function $\phi$ which was the main ingredient in our identification of $L^2(\bbT_\infty,\mu_\infty)$ with 
$L^2(\bbR)$. The next example shows that if the first hypothesis \eqref{eq:Mall1} of the Lemma fails, then from the perspective of harmonic analysis 
$L^2(\bbT_\infty,\mu_\infty)$ might be very different from $L^2(\bbR)$. The main reason for this 
failure is the lack of a scaling function. The example recovers the wavelets on fractals of \cite{DuJo_RMI06} from our set-up.
\begin{example}\label{ex:wavelet_fractals}
    Let $X=\bbT$ and $\sigma:\bbT\to\bbT$ be defined via $\sigma(z)=z^3$. Consider the standard
    transfer operator $\cL$ given by the potential  $\psi(z)=1/3$ for all $z\in \bbT$. Consider the $\psi$-filter
    $\mf:\bbT\to\bbC$ defined via $\mf(z)=2^{-1/2}(1+z^2)$ for all $z\in \bbT$ (\cite{DuJo_RMI06, 
    Bag_co_JFA10}) and let $\fu(z)=\mf(z)/\sqrt3$. Note that $\vert \fu(z)\vert\le \sqrt2/\sqrt3$ and, thus, it does 
    not satisfy hypothesis \eqref{eq:Mall1}. Hence one can not define a scaling function
    in the sense of Proposition \ref{prop:scalingfunction}.  Consider the standard 
    fundamental representation $\hat{\iota}=(\mu,\bbT\ast \bbC,\hat{\iota})$, where $\mu$ is the Haar measure.  
    Hence the   isometry $\tS_\fu=\fI(S_u)$ on $L^2(\bbT,\mu)$ is 
    given via $\tS_\fu(\xi)(z)=\mf(z)f(z^3)$ (c.f. \cite[Section 5]{Bag_co_JFA10}). Let $\hat{\iota}_\infty=(\mu_\infty^D,
    \bbT_\infty*\bbC,\hat{\iota}_\infty)$ be the induced representation of $G_\infty(\bbT,\sigma)$ defined by the 
    potential $D(z)=\vert \fu(z)\vert^2$.
    Proposition 5.1 of \cite{Bag_co_JFA10} implies that 
    $(L^2(\bbT_\infty,\mu_\infty^D),\tU_\fu)$ is unitarily equivalent to $(L^2(\cR,\nu),\fU_\fu)$, where $\cR$ 
    is the
    "inflated" Cantor set defined in \cite[Definition 1.1]{DuJo_RMI06}, $\nu$ is the Hausdorff measure
    on $\cR$, and $\fU_\fu:L^2(\cR,\nu)\to L^2(\cR,\nu)$ is defined via $\fU_\fu(\xi)(x)=2^{-1/2}\xi(x/3)$.
    Thus one recovers the MRA for the wavelets studied in \cite{DuJo_RMI06}.
    
\end{example}

\bibliographystyle{amsplain}
\bibliography{references}

\end{document}